\documentclass[11pt,reqno]{amsproc}
\usepackage[margin=1in]{geometry}
\usepackage{amsmath, amsthm, amssymb}
\usepackage{times, esint,stackrel}
\usepackage{enumitem, graphicx}
\usepackage{color}
\usepackage{enumitem}   
\usepackage{hyperref}
\usepackage{etoolbox}
\usepackage{mathrsfs}
\usepackage{mathabx}
\usepackage{stackengine}
\usepackage{subcaption,graphicx}
\usepackage{float}

\usepackage{tikz}
\usetikzlibrary{decorations.pathreplacing,decorations.markings}
\tikzset{
	on each segment/.style={
		decorate,
		decoration={
			show path construction,
			moveto code={},
			lineto code={
				\path [#1]
				(\tikzinputsegmentfirst) -- (\tikzinputsegmentlast);
			},
			curveto code={
				\path [#1] (\tikzinputsegmentfirst)
				.. controls
				(\tikzinputsegmentsupporta) and (\tikzinputsegmentsupportb)
				..
				(\tikzinputsegmentlast);
			},
			closepath code={
				\path [#1]
				(\tikzinputsegmentfirst) -- (\tikzinputsegmentlast);
			},
		},
	},
	mid arrow/.style={postaction={decorate,decoration={
				markings,
				mark=at position .5 with {\arrow[#1]{stealth}}
	}}},
}

\usepackage{color} 
\usepackage{xcolor} 
\usepackage{hyperref}
\usepackage{bbm}

\makeatletter

\def\@tocline#1#2#3#4#5#6#7{\relax
  \ifnum #1>\c@tocdepth 
  \else
    \par \addpenalty\@secpenalty\addvspace{#2}%
    \begingroup \hyphenpenalty\@M
    \@ifempty{#4}{%
      \@tempdima\csname r@tocindent\number#1\endcsname\relax
    }{%
      \@tempdima#4\relax
    }%
    \parindent\z@ \leftskip#3\relax \advance\leftskip\@tempdima\relax
    \rightskip\@pnumwidth plus4em \parfillskip-\@pnumwidth
    #5\leavevmode\hskip-\@tempdima
      \ifcase #1
       \or\or \hskip 1em \or \hskip 2em \else \hskip 3em \fi%
      #6\nobreak\relax
    \dotfill\hbox to\@pnumwidth{\@tocpagenum{#7}}\par
    \nobreak
    \endgroup
  \fi}
\makeatother

\newtheorem{theorem}{Theorem}
\newtheorem{proposition}{Proposition}[section]
\newtheorem{lemma}[proposition]{Lemma}
\newtheorem{corollary}[proposition]{Corollary}

\theoremstyle{definition}
\newtheorem{definition}[proposition]{Definition}

\newtheorem{question}[proposition]{Question} 
 
\newtheorem{remark}[proposition]{Remark}

\numberwithin{equation}{section}

\newcommand\eps{\varepsilon}
\newcommand\e{{\rm e}}
\newcommand\dd{{\rm d}}

\newcommand\de{{\partial}}

\newcommand{\norm}[1]{\left\lVert #1 \right\rVert}

\newcommand{\ZZ}{\mathbb{Z}}
\newcommand\TT {{\mathbb T}}
\newcommand\RR {{\mathbb R}}

\newcommand\bv{{\boldsymbol v}}

\newcommand\rmd{{\rm d}}

\newcommand\cO{{\mathcal O}}

\newcommand{\be}{\begin{equation}}
\newcommand{\ee}{\end{equation}}

\def\wcc{\rightharpoonup}
\newcommand\wc{\mathrel{\stackrel{\makebox[0pt]{\mbox{\normalfont\tiny *}}}{\wcc}}}

\definecolor{redroma}{RGB}{142,0,28}
\definecolor{yroma}{RGB}{255,179,0}

\makeatletter
\def\@xfootnote[#1]{%
	\protected@xdef\@thefnmark{#1}%
	\@footnotemark\@footnotetext}
\makeatother

\title[On maximally mixed equilibria of two-dimensional perfect fluids] {
On maximally mixed equilibria of two-dimensional perfect fluids
}

\author[M. Dolce]{Michele Dolce}
\address{Department of Mathematics, Imperial College London, London, SW7 2AZ, UK}
\email{m.dolce@imperial.ac.uk}
\author[T. D. Drivas]{Theodore D. Drivas}
\address{ Department of Mathematics, Stony Brook University,
Stony Brook, NY, 11794, USA}
\email{tdrivas@math.stonybrook.edu}

\address{ School of Mathematics, Institute for Advanced Study, 1 Einstein Dr., Princeton, NJ 08540, USA}
\email{ tdrivas@ias.edu}

\begin{document}
\begin{abstract}
	The vorticity of a two-dimensional perfect (incompressible and inviscid) fluid is transported by its area preserving flow. Given an initial vorticity distribution $\omega_0$, predicting the long time behavior which can persist is an issue of fundamental importance. In the infinite time limit, some irreversible mixing of $\omega_0$ can occur. Since kinetic energy $\mathsf{E}$ is conserved, not all the mixed states are relevant and it is natural to consider only the ones with energy $\mathsf{E}_0$ corresponding to $\omega_0$. The set of said vorticity fields, denoted by $\overline{\mathcal{O}_{\omega_0}}^*\cap    \{ {\mathsf E}= {\mathsf E}_0\}$, contains all the possible end states of the fluid motion.  A. Shnirelman introduced the concept of  \textit{maximally mixed states} (any further mixing would necessarily change their energy), and proved they are perfect fluid equilibria. We offer a new perspective on this theory by showing that any minimizer of any strictly convex Casimir in $\overline{\mathcal{O}_{\omega_0}}^*\cap    \{ {\mathsf E}= {\mathsf E}_0\}$ is maximally mixed, as well as  discuss its relation to classical statistical hydrodynamics theories. Thus, (weak) convergence to equilibrium cannot be excluded solely on the grounds of vorticity transport and conservation of kinetic energy.  On the other hand, on domains with symmetry (e.g. straight channel or annulus), we exploit all the conserved quantities and the characterizations of $\overline{\mathcal{O}_{\omega_0}}^*\cap    \{ {\mathsf E}= {\mathsf E}_0\}$ to give examples of open sets of initial data which can be arbitrarily close to any shear or radial flow in $L^1$ of vorticity but do not weakly converge to them in the long time limit. 
  \end{abstract}

\maketitle

\section{Introduction}
Let $M\subset \mathbb{R}^2$ be a bounded domain possibly with boundary $\partial M$ having exterior unit normal $\hat{n}$, 
e.g., the flat two-torus $\mathbb{T}^2$, the periodic channel $\mathbb{T}\times [0,1]$ or the disk $\mathbb{D}$. The  Euler equations governing the motion of a fluid which is perfect (inviscid and incompressible) and confined to $M$ read \cite{majda2002vorticity}
  \begin{alignat}{2}\label{eeb}
\partial_t u + u \cdot \nabla u &= -\nabla p,   &\qquad \text{in} \quad M,\\
\nabla \cdot u &=0,   &\qquad \text{in} \quad M,\\
u|_{t=0} &=u_0 ,   &\qquad\text{in} \quad M,\\
u\cdot \hat{n} &=0,   &\qquad\text{on}\   \partial M. \label{eef}
\end{alignat}
In terms of the vorticity $\omega := \nabla^\perp \cdot u$  where $\nabla^\perp:=(-\partial_2, \partial_1)$, the system above can be reformulated as
  \begin{alignat}{2}\label{eevb}
\partial_t\omega + u \cdot \nabla \omega&=0 &\qquad \text{in} \quad M,\\
\omega|_{t=0} &=\omega_0 ,   &\qquad\text{in} \quad M, \label{eevf}
\end{alignat}
where $u=K_M[\omega] = \nabla^\perp \Delta^{-1} \omega$ is recovered by the Biot-Savart law.  Equations \eqref{eevb}--\eqref{eevf} say that the vorticity is transported by particle trajectories, namely the solution admits the representation 
\be\label{vortlag}
\omega(t) = \omega_0\circ \Phi_t^{-1}
\ee
where  
\be
\frac{\rmd}{\rmd t} {\Phi}_t = u(\Phi_t,t), \qquad \Phi_0={\rm id}
\ee
 is the Lagrangian flowmap.   See Figure \ref{figeuler} for a visualization of the motion of the vorticity field starting from Gaussian random data.   The apparent emergence of large-scale order via the inverse energy cascade is a principle mystery of two-dimensional fluids which demands explanation from first principles.

	\begin{figure}[h!]
		\includegraphics[scale=.3]{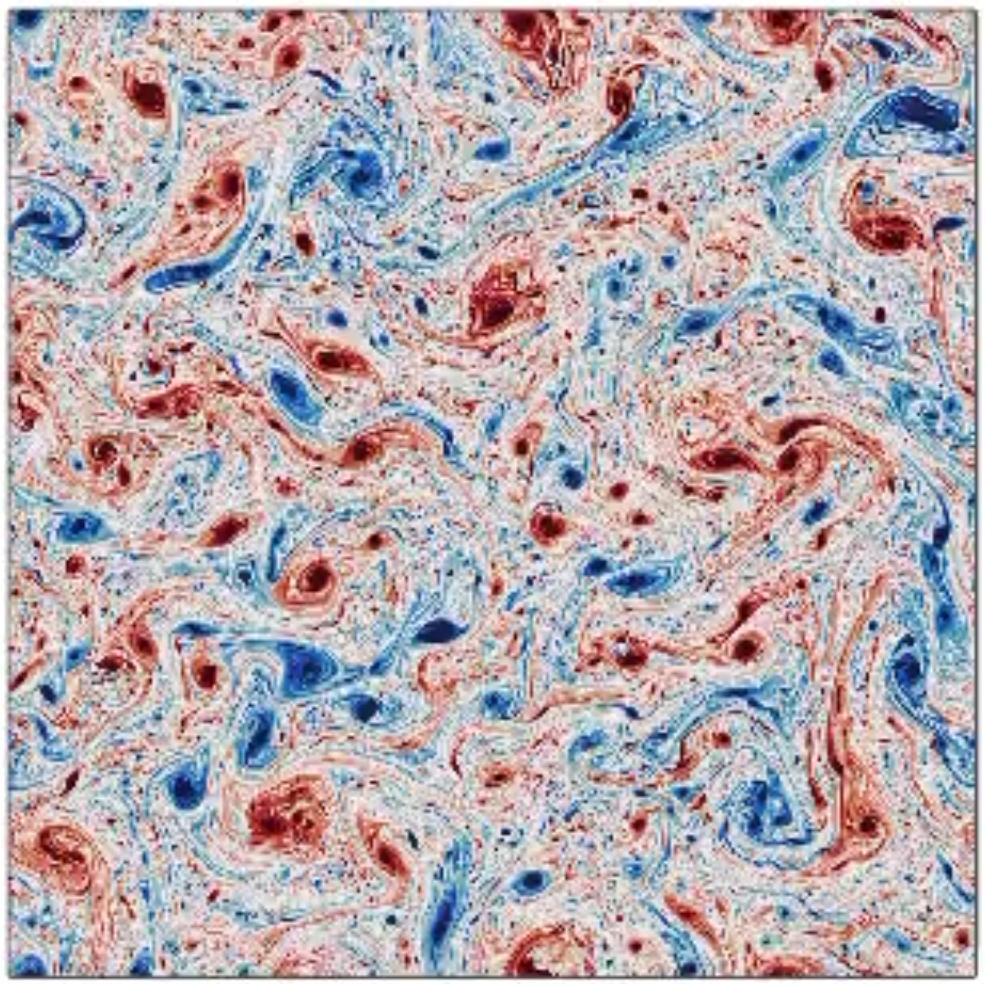} \hspace{2mm} \includegraphics[scale=.3]{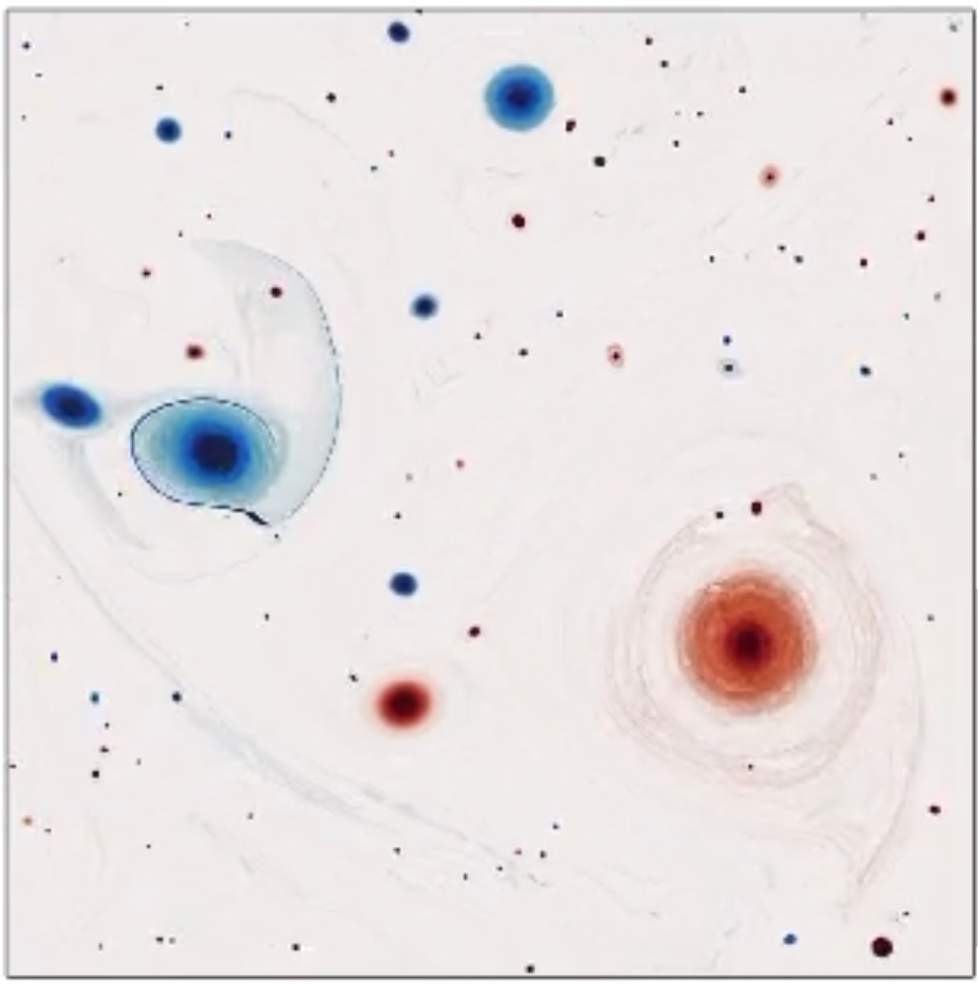}   \hspace{2mm}  \includegraphics[scale=.3]{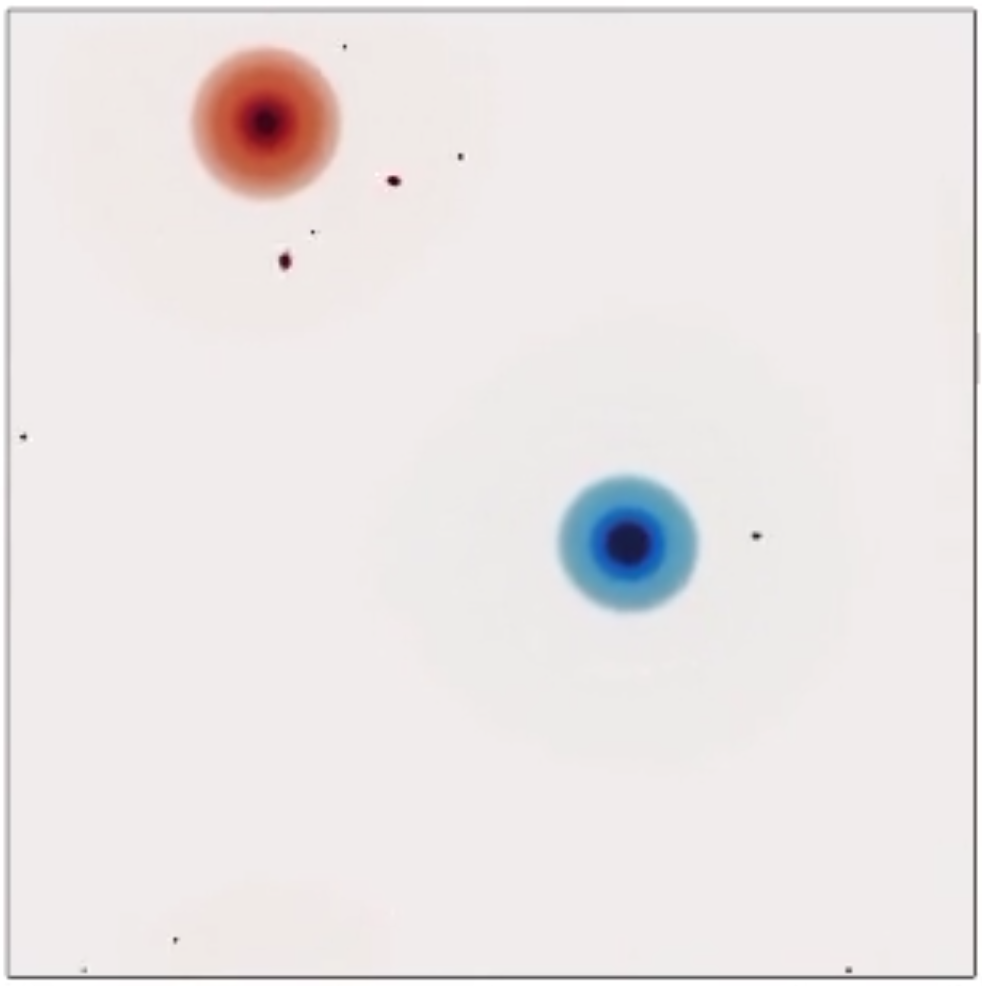}\\ 	
		\caption{Direct numerical simulations \cite{CD21,C21} of the time evolution (from left to right) of the 2d Euler vorticity field starting from an instance of Gaussian random initial data.}
		\label{figeuler}
	\end{figure}

	To formalize the study of the dynamics of two-dimensional fluids, we recall the following classical result on global wellposedness of bounded vorticity solutions.  Specifically, let $X$ be a ball in $L^\infty$ 
\be
\label{def:X}
X := \{ \omega\in L^\infty(M) \ : \ \|\omega\|_{L^\infty(M)} \leq 1\}.
\ee
Yudovich \cite{yudovich1963non} proved that the compact (with the weak-$*$ topology) metric space $X$ is a good phase space for the Euler equations in that \eqref{eevb}--\eqref{eevf} forms an infinite dimensional, time reversible, dynamical system on $X$ for all time. We call the time $t\in \mathbb{R}$ solution operator $S_t: X\righttoleftarrow$. Our interest is the long time behavior of this dynamical system.
 Since  $\omega(t)=S_t(\omega_0)$ satisfies $\|\omega(t)\|_{L^\infty(M)} = \|\omega_0\|_{L^\infty(M)} \leq 1$, we have 
   $$\omega(t_i)  \wc\overline{\omega} \quad  \text{along subsequences} \quad t_i\to \infty$$ 
   where, we recall that   weak-$*$ convergence is defined for $f_n\in L^\infty(M)$ by
\be\label{weakstar}
\lim_{n\to\infty} \int_{M} \varphi(x) f_n(x) \rmd x = \int_{M} \varphi(x) \overline{f}(x)\rmd x,  \qquad \forall \varphi\in L^1(M).
\ee
If $\|f_n\|_{L^\infty}$ is uniformly bounded (as is the case for $\|\omega(t_n)\|_{L^\infty(M)}$), then this notion of weak convergence agrees with others such as weak convergence in $L^2$.
   Weak-$*$ limits $\overline{\omega}$ can forget oscillations, leaving only a ``coarse-grained" representative of the vorticity.  As such, one could hope to describe and predict coherent structures arising at very long times by studying weak-$*$ limits (e.g., capturing the vortices and not the fine-scale filaments that can be observed in the rightmost panel of Figure \ref{figeuler}).
Denoting the weak-$*$ closure in $L^\infty(M)$ by $\overline{(\cdot )}^*$, we introduce the Omega limit set
 \be\label{omlimset}
\Omega_+(\omega_0) := \bigcap_{s\geq 0} \overline{\{ S_t(\omega_0), t\geq s \}}^*,
\ee
which is the collection of all such weak-$*$ limits as $t\to \infty$ along the solution $\omega(t)= S_t(\omega_0)$ passing through $\omega_0\in X$ at time $0$. The set $\Omega_+(\omega_0)$ represents all possible `coarsened' persistent motions launched by $\omega_0$.   

 Our interest is in understanding the structure of $\Omega_+(\omega_0)$: what kind of 2D perfect fluid motions can persist indefinitely? 
We want to see what can be ruled out \emph{kinematically},  based solely on the transport structure of the vorticity evolution, after accounting for some robust conserved quantities.  Recall that the conservation laws for the Euler equation which hold on general planar domains (possibly multiply connected) are
\begin{align*}
\text{energy:} \quad{\mathsf E}(\omega(t)) &= {\mathsf E}(\omega_0), \qquad \ \ {\mathsf E}(\omega) :=\frac{1}{2}\int_M |K_M[\omega]( x)|^2 \rmd x= \frac{1}{2}\int_M |u( x)|^2 \rmd x,\\
\text{Casimirs:} \quad {\mathsf I}_f(\omega(t)) &= {\mathsf I}_f(\omega_0), \qquad \ \ {\mathsf I}_f(\omega) := \int_Mf(\omega(x)) \rmd x, \qquad \text{for any continuous } \ f: X\to \mathbb{R},\\
\text{circulation:} \quad {\mathsf K}_i(\omega(t)) &= {\mathsf K}_i(\omega_0), \qquad  {\mathsf K}_i(\omega) := \int_{\Gamma_i} u\cdot \rmd \ell, \qquad \text{for connected  components $\Gamma_i$ of $\partial M$} .
\end{align*}
If the domain has additional symmetries there can be additional invariants such as linear momentum on the torus and channel\footnote{These can be related to conservation of circulation of the harmonic component of $u$ around fixed non-contractible loops.} and angular momentum on the disk:
\begin{align} \nonumber \stackengine{4pt}{\text{linear momentum}}{\text{on } $M= \mathbb{T}\times [0, 1]$}{U}{c}{F}{T}{S}: \quad {\mathsf M}(\omega(t)) &= {\mathsf M}(\omega_0), \qquad {\mathsf M}(\omega) := \int_{\{y=0\}}  u_1 \rmd x  -\int_M x_2 \omega (x) \rmd x = \int_M e_1\cdot u \rmd x,\\
	\notag \stackengine{4pt}{\text{angular momentum}}{\text{on } $M= \mathbb{D}$}{U}{c}{F}{T}{S}:
	\quad {\mathsf A}(\omega(t)) &= {\mathsf A}(\omega_0), \qquad {\mathsf A}(\omega) := -\int_M \tfrac{1}{2}(1-|x|^2) \omega(x) \rmd x= \int_Mx^\perp\cdot  u (x)  \rmd x.
\end{align}
\vspace{.25mm}

 \noindent However, for domains without Euclidean symmetries, linear and angular momentum conservation are lost due to pressure effects.  Casimirs and circulations are the \emph{only} invariants for general area preserving transformations of the vorticity (see Izosimov and Khesin  \cite{izosimov2017classificationa,izosimov2017characterization}). Together with energy, they are the only \emph{known} conservation laws (first integrals) for perfect fluids in 2D which hold for all data and on arbitrary domains.  For simplicity of presentation, we will primarily work on simply connected domains where the circulation does not impose any additional constraints beyond the constancy of the domain-averaged vorticity.

It is informative to consider the constraint that the vorticity function is, at every instant, an area preserving rearrangement of its initial datum imposes on the long time behavior.
Let $\mathscr{D}_\mu(M)$ denote the group of area preserving  diffeomorphisms on $M$ and denote the orbit of $\omega_0\in X$ in $\mathscr{D}_\mu(M)$ by
\be
\mathcal{O}_{\omega_0}:= \{ \omega_0 \circ \varphi \ :\ \varphi  \in \mathscr{D}_\mu(M)  \},
\ee
where we understand $\varphi$ to be in the component of the identity.
The fact that the diffeomorphisms $\varphi$ are area preserving implies that the Casimirs ${\mathsf I}_f$ are constant along orbits $ \mathcal{O}_{\omega_0} $ just as they are for Euler (in fact, the term Casimir implies they are invariants for the whole orbit).
To get closer to the Euler dynamics, we consider the intersection of this orbit with constant energy fields
\be
\mathcal{O}_{\omega_0,{\mathsf E}_0}:= \mathcal{O}_{\omega_0}\cap \{ {\mathsf E}= {\mathsf E}_0\}.
\ee
According to the representation \eqref{vortlag} we have that $\omega(t)= S_t(\omega_0)\in \mathcal{O}_{\omega_0,{\mathsf E}_0}$ for all $t\in \mathbb{R}$.

 In the coarse-grained infinite-time picture captured by weak-$*$ limits, there is a marked difference between the energy (also circulations and, on domains with symmetry, momentum) and the Casimirs: the energy is  weak-$*$ continuous whereas the non-linear Casimirs are not.  This means all the weak-$*$ limits $\overline{\omega}\in \Omega_+(\omega_0)$ have the same energy as the initial data -- it can thus be termed as a \emph{robust invariant}. Weak-$*$ limits, on the other hand, need not remember the initial Casimirs.  In fact, if $\omega(t_i)\wc \overline{\omega}$ we can only deduce  by lower-semicontinuity that
 \begin{equation}
 	\label{bd:lscweak}
 	\mathsf{I}_f(\overline{\omega})\leq \liminf_{i\to \infty}\mathsf{I}_f(\omega(t_i))=\mathsf{I}_f(\omega_0) \quad \text{ for any convex } f.
 \end{equation}
 Loss of enstrophy on weak limits, namely $\norm{\bar\omega}_{L^2}^2\leq\norm{\omega_0}_{L^2}^2 $ (or. more generally speaking, with a strict inequality in \eqref{bd:lscweak} for any convex Casimir) is associated to fine-scale \textit{mixing}.  This behavior is often observed in the long time limit of the Euler evolution and 
 is conjectured to be typical \cite{vsverak2012selected}.

\newpage 
Consequently, we have the following containments
\be\label{limOmlim}
\Omega_+(\omega_0) \subset \overline{\mathcal{O}_{{\omega_0,{\mathsf E}_0}}}^* \subset \overline{\mathcal{O}_{\omega_0}}^*\cap    \{ {\mathsf E}= {\mathsf E}_0\}
\ee
where the last containment is a consequence of  energy being weak-$*$ continuous. Due to mixing, on set $ \overline{\mathcal{O}_{\omega_0}}^*$ the Casimirs are no longer constant but convex Casimirs do not increase in view of \eqref{bd:lscweak}.
Thanks to this, given a strictly convex Casimir, we have a preorder structure
(a ``mixing order") on $\overline{\mathcal{O}_{\omega_0}}^*\cap    \{ {\mathsf E}= {\mathsf E}_0\}$:
\begin{definition}
	\label{def:po}
    Let $f$ be a strictly convex function. Given $\omega_1, \omega_2 \in \overline{\mathcal{O}_{\omega_0}}^*\cap \{ {\mathsf E}= {\mathsf E}_0\}$, we say $\omega_1 \preceq_f \omega_2$ if $\mathsf{I}_f(\omega_1) \leq \mathsf{I}_f(\omega_2)$.
\end{definition}
Note that a preorder naturally gives rise to a partial order once we quotient the space with respect to the induced equivalence relation.
Moreover, it is natural to introduce the notion of a \emph{minimal element}:
\begin{definition}
	\label{def:convexminimal}
An $\omega^*\in\overline{\mathcal{O}_{\omega_0}}^*\cap    \{ {\mathsf E}= {\mathsf E}_0\}$ is $f$-\textit{minimal} if for all $\omega$ such that $\omega\preceq_f\omega^*$ then $\omega^*\preceq_f \omega$.
\end{definition}

Namely, an $f$-minimal element (termed $f$-\emph{minimal flow}) $\omega^*$ has the property  that if $\omega\preceq_f\omega^*$ then  $\mathsf{I}_f(\omega)= \mathsf{I}_f(\omega^*)$.
 We can therefore think that $f$-minimal elements are \textit{maximally mixed} versions of $\omega_0$ at a given fixed energy, where mixing is measured through the loss of a given strictly convex Casimir.
 \begin{remark}
The idea of looking at maximally mixed states originates from a different preorder structure on $\overline{\mathcal{O}_{\omega_0}}^*\cap    \{ {\mathsf E}= {\mathsf E}_0\}$ introduced by Shnirelman in \cite{shnirelman1993lattice}, which we recall in Definition \ref{def:orderShnirelman} below. Shnirelman \cite{shnirelman1993lattice} then establishes the existence of minimal flows, according to his preorder, as an application of Zorn's lemma, and we comment more about this in \S \ref{sec:char}.   See also the discussion by Arnold and Khesin \cite{arnold2021topological}.
     In fact, let us observe that we could also define another preorder by requiring that the inequality $\mathsf{I}_f(\omega_1)\leq \mathsf{I}_f(\omega_2)$ holds for \emph{all} convex Casimirs. This gives rise to yet another notion of minimal elements, whose existence is obtained as well by applying Zorn's lemma\footnote{Indeed, the set $\overline{\mathcal{O}_{\omega_0}}^*\cap\{\mathsf{E}=\mathsf{E}_0\}$ is weakly-* compact. The weak-* lower semi-continuity of $\mathsf{I}_f(\omega)$ for convex $f$ ensures that weak-* accumulation points of totally ordered chains are lower bounds and therefore minimal elements.}. The main advantage of working with Definition \ref{def:po} is the underlying variational characterization of $f$-minimal elements, and that it provides a tool to measure the level of mixing that has occurred.
 \end{remark}
 \begin{remark}[Steady states: minimal and non-minimal flows]
 An important class of flows in regard to long term behavior of \eqref{eeb}--\eqref{eef} are those that are independent of time (steady). In view of $u$ be a two-dimensional divergence--free vector field, it is useful to introduce a streamfunction $\psi: M\to \mathbb{R}$ such that $u=\nabla^\perp \psi$. If the velocity is tangent to the boundary then $\psi$ must be constant on connected components of $\partial M$. Being a stationary state imposes the condition  that vorticity gradients are locally parallel to gradients of the streamfunction ($\nabla^\perp \psi\cdot\nabla \omega=0$).  A large class of steady solutions have additional structure, namely that the  vorticity $\omega$ is a given function of the stream function $\psi$ \emph{globally}, e.g. $\omega = F(\psi)$ for some $F:\mathbb{R}\to \mathbb{R}$.   Together with $\omega=\nabla^\perp\cdot u=\Delta \psi$ this means $\psi$ satisfies
\begin{alignat}{2}
 \Delta \psi &= F(\psi), \quad &&\text{ in } M,\label{gseuler1}\\
 \psi &= {\rm (const.)} , \quad &&\text{ on } \partial M.
 \label{gseuler2}
\end{alignat}
Any solution of the above problem is the stream function of a steady solution to 2D Euler which is tangent to the boundary.  
As it turns out, \emph{all} $f$-minimal flows are stationary solutions possessing a global $F$, see \cite{shnirelman1993lattice} and Theorem \ref{varprop} (i).
 Another  privileged family (a subclass of $f$-minimal flows) are called \emph{Arnold stable}.  They satisfy \eqref{gseuler1}--\eqref{gseuler2} for a Lipschitz $F$ satisfying either of the following two conditions
\begin{equation}
-\lambda_1<F'(\psi) <0, \qquad \text{or} \qquad 0< F'(\psi)<\infty
 \label{arnoldscond}
\end{equation}
where $\lambda_1:=\lambda_1(\Omega)>0$ is the smallest eigenvalue of $-\Delta$ in $M$.  These flows are Lyapunov stable in the $L^2$ topology of vorticity under the 2D Euler evolution. 
Any Arnold stable steady state is an $f$-minimal flow, since any mixing of them necessarily results in a change of energy. For an area preserving rearrangement, this follows by the fact that they are local maximizers or minimizers of the energy on their isovortical sheet $\mathcal{O}_{\omega}$, see e.g.  \cite{arnold2021topological,gallay2021arnold}. 
 On the other hand, any shear flow (on the channel) or circular flow (on the disk) having an inflection point cannot be an $f$-minimal flow. This is implied by \cite{shnirelman1993lattice} and Lemma \ref{prop:minstead} herein.
 \end{remark}
 
  One of the main purposes of this paper is to offer a different perspective of certain maximally mixed flows, specifically to those that naturally minimize the value of \emph{one} given strictly convex function.
  In \S 5 we prove the following:

\begin{theorem}\label{varprop}
Let $M\subset \mathbb{R}^2$ be a bounded planar simply connected domain with smooth boundary and let $f: X \to \mathbb{R}$ be a strictly convex function. Given any $\omega_0\in X$ with energy ${\mathsf E}_0$, there exists a minimizer $\omega^*\in X$ such that
\be\label{vp}
{\mathsf I}_f(\omega^*)=\min_{\omega\ \in\  \overline{\mathcal{O}_{\omega_0}}^*\cap    \{ {\mathsf E}= {\mathsf E}_0\} } {\mathsf I}_f(\omega).
\ee
Any such minimizer $\omega_*$ is both $f$-minimal and a minimal flow in the sense of Shnirelman in Definition \ref{def:orderShnirelman} and enjoys the following properties:
\begin{itemize}
\item[(i)]
 $\omega_*$   is a stationary solution of the Euler equation having the property that there exists a bounded monotone function $F:\mathbb{R}\to \mathbb{R}$ such that $\omega_*= F(\psi_*)$,
\item[(ii)]  there exists a continuous convex function $\Phi$ and scalars $\alpha,\beta,\gamma\in \mathbb{R}$ with $\alpha^2 + \beta^2\neq 0$ such that  $\omega_*$ is a minimizer on $X$ (the unconstrained space) of the functional 
\begin{equation}\label{Jfunc}
	J_\Phi(\omega)=\mathsf{I}_{\Phi+\alpha f }(\omega)+\beta(\mathsf{E}(\omega)-\mathsf{E}_0)+ \gamma \int_M (\omega - \omega_0)\rmd  x.
\end{equation}
\end{itemize}
\end{theorem}
\begin{remark}\label{alphanonzero}
If $\alpha\neq 0$, then $\Phi+\alpha f$ is strictly convex and thus the minimizer  of \eqref{Jfunc} is unique and satisfies $\omega= F(\psi)$ where $F(z) := ( \partial\Phi+\alpha f')^{-1}(-\beta z-\gamma )$ with $\partial$ denoting the subdifferential. In this case, $F$ is strictly increasing or decreasing, and under extra assumptions on $f$, one might be able to prove that it becomes Lipschitz.
\end{remark}
The theorem above gives a method to produce stationary and $f$-minimal solutions of the Euler equations by solving a variational problem on $\overline{\mathcal{O}_{\omega_0}}^*\cap    \{ {\mathsf E}= {\mathsf E}_0\}$.  Certain characterizations of this set are available, see \S \ref{sec:charOmset}.  Using these, in Appendix \ref{finiteconst}, we give a concrete and explicit instance of Shnirelman's maximal mixing theory as it applies to vortex patches with a finite number of regions. In this case, \eqref{vp} can be seen as an optimization problem with a \textit{finite} number of inequality constraints. Point (ii) of the Theorem  gives the natural extension of such characterization for a general $\omega_0$, inspired by work of Rakotoson and Serre \cite{rakotoson1993probleme}.
\begin{remark}
    In non-simply connected domains, one should consider the whole set $\overline{\mathcal{O}_{\omega_0}}^*\cap\{\mathsf{E}=\mathsf{E}_0\}\cap\{\mathsf{K}_i=\mathsf{K}_0\, ;\,  i=1,\dots,N\}$ where $N$ is the number of connected components of $\partial M$. This is necessary to have compatibility conditions to define the streamfunction as $\Delta\psi=\omega$ (in simply connected domains it is enough to fix the average of the vorticity, a condition included in $\overline{\mathcal{O}_{\omega_0}}^*$). Similarly, on domains with symmetries as the channel (disk) one imposes the conservation of the linear (angular) momentum. However, the proof of Theorem \ref{varprop} only requires straightforward modifications to account for these extra constraints.
\end{remark}

\begin{remark}[Shnirelman minimal flows as minimizers]
A remarkable property of Shnirelman's minimal flows  is that they are \emph{all} stationary solutions of the Euler equation having the property $\omega^*=F(\psi^*)$ for some bounded monotone function $F$. Moreover, if one is able to show that such $F$ is \emph{strictly} monotone, then the flow is trivially $f$-minimal when $f$ is the primitive of $F.$ This follows by the standard Lagrange multiplier rule in the larger set $X\cap \{ {\mathsf E}={\mathsf E}_0\}$.  It remains an open issue to understand whether Shnirelman's minimal flows in  $\overline{\mathcal{O}_{\omega_0}}^*\cap    \{ {\mathsf E}= {\mathsf E}_0\}$  (particularly those having regions of constant vorticity, see Appendix \ref{constvort}) can be realized as minimizers of some strictly convex functional.
\end{remark}

\begin{remark}[Non-uniqueness and regularity of minimal flows]
Given $\omega_0\in X$ with energy ${\mathsf E}_0$, there is no reason for the $f$-minimal flow in  $\overline{\mathcal{O}_{\omega_0}}^*\cap    \{ {\mathsf E}= {\mathsf E}_0\} $ to be unique. Moreover, in view of Remark \ref{2patchrem}, one can construct $f$-minimal flows in $\overline{\mathcal{O}_{\omega_0}}^*\cap    \{ {\mathsf E}= {\mathsf E}_0\} $ with better regularity than the datum $\omega_0\in X$.
\end{remark}

Theorem \ref{varprop} sheds light on some questions concerning relaxation to equilibrium. It is interesting to ask whether or not an initial datum can be kinematically isolated from all stationary solutions.  In this direction, 
 Ginzburg and Khesin \cite{ginzburg1994steady,ginzburg1992topology} showed that if  $M$ is a simply connected planar domain and $\omega_0$ is Morse, positive and has both a local maximum and minimum in the interior, then $\mathcal{O}_{\omega_0}$ contains no smooth Euler steady state. In the other direction,  Choffrut and {\v{S}}ver{\'a}k \cite{choffrut2012local} gave a full characterization of the steady states nearby certain Arnold stable ones on annular domains by showing that they are in one-to-one correspondence with their distribution functions, i.e. for all $\omega_0$ sufficiently close to $\omega$, there exists a unique stationary solution on 
$
\mathcal{O}_{\omega_0}$. Later,
Izosimov and Khesin \cite{izosimov2017characterization}  gave necessary conditions on the vorticity $\omega_0$ in order for a smooth steady Euler solution to exist on $\mathcal{O}_{\omega_0}$  for
any metric, as well as a sufficient condition for the existence of a steady solution for some metric. 

A consequence of Theorem \ref{varprop} is that, for any initial data with bounded vorticity, there  \emph{always exist} stationary solutions (with bounded vorticity, but not necessarily smooth) in the set $\overline{\mathcal{O}_{\omega_0}}^*\cap    \{ {\mathsf E}= {\mathsf E}_0\}$.  In light of the above discussion, the space $\overline{\mathcal{O}_{\omega_0}}^*\cap    \{ {\mathsf E}= {\mathsf E}_0\}$ represents the finest, coarse representative (accounting for all known kinematic constraints on the solution as well as the conservation laws) of the Omega limit set $\Omega_+(\omega_0)$ which we generally have. As such, this information alone is not enough to rule out  relaxation to equilibrium via Euler evolution for any initial datum, at least in a weak-$*$ sense. 

Convergence to equilibrium at long time can occur, although theorems (and likely scenarios) are very rare. The results \cite{bedrossian2015inviscid,ionescu2020nonlinear,masmoudi2020nonlinear,ionescu2022axi} are the only to fully characterize the Omega limit sets for Euler, albeit for very smooth perturbations of special equilibria.  For instance, if $\overline{\omega}$ is the vorticity of a (class of) strictly monotone shear flow on $\mathbb{T}\times [0,1]$ or $\mathbb{T}\times \mathbb{R}$, then for any $\omega_0$ in a Gevrey-2 
 neighborhood of $\overline{\omega}$, one has
\be\label{invisciddamplim}
\Omega_+(\omega_0) = \{ \overline{\omega}_{\omega_0}\} 
\ee
where $ \overline{\omega}_{\omega_0}$ is the vorticity of a (slightly modified) shear flow nearby $ \overline{\omega}$. The convergence happens weakly, not strongly, in $L^2$ so some amount of mixing definitively occurs. Thus, these remarkable results -- termed \textit{inviscid damping} -- show that certain full neighborhoods in the (Gevrey) phase space relax to equilibrium at long time, a feature consistent with Theorem \ref{varprop} and the theories of Statistical hydrodynamics described in \S \ref{sec:char}.  The fact that these stable equilibria are symmetric is no accident. On domains with symmetry, also the Arnold stable steady states must inherit the symmetry of the domain they occupy \cite{constantin2021flexibility}; all such on the channel are shears, while on the annulus they are  circular. It is unclear if the flows $ \overline{\omega}_{\omega_0}$ in \eqref{invisciddamplim} are $f$-minimal for some particular $f$. 

On the other hand, convergence to symmetric equilibria (even in this weak sense) seems to be the exception rather than the rule more generally.   Lin and Zeng \cite{lin2011inviscid} discovered non-shear Catseye steady states nearby (at low regularity) to the Couette shear flow, and there are also travelling waves with an order 1 velocity as showed by Castro and Lear \cite{castro2021traveling}. Such steady structure have recently been identified nearby the Kolmogorov flow (in the analytic topology) and the Poiseuille flow by Coti Zelati, Elgindi and Widmayer \cite{zelati2020stationary} and by Nualart \cite{nualart2022zonal} on the rotating sphere nearby zonal flows. These results provide an obstruction to inviscid damping back to a shear flow for general perturbations nearby certain shear flows of a given structure.  However, they do not rule out convergence to shear flow for some perturbations. 

 In a similar spirit,  we show here that there exist open sets of small, sufficiently coarse, perturbations of any shear flow on the periodic channel (actually, of any bounded vorticity field on the channel) that cannot possibly damp back to a shear flow.  Unlike those previous works, we do this not by finding other nearby steady states, but rather by excluding shear flows directly from a set containing the Omega limit set.

\begin{theorem}\label{damp}
Let $M= \mathbb{T}\times [0,1]$ and $\omega_b\in L^\infty(M)$.  For any $\delta>0$, there exists  $\xi \in C^\infty(M)$   such that 
\begin{equation}
	\label{bd:smallness}
	\norm{\xi-\omega_{b}}_{L^1}\lesssim \delta
\end{equation}
and for which the set $\overline{\mathcal{O}_{\xi}}^*\cap    \{ {\mathsf E}=  \mathsf{E}(\xi)\} \cap  \{ {\mathsf M}=  \mathsf{M}(\xi)\} $ contains no shear flows.  
\end{theorem}
The field $\xi$ is comprised of highly peaked vortices embedded in the background $\omega_b$, i.e. there is $0<\varepsilon:=\varepsilon(\|\omega_b\|_{L^\infty})<\delta$ so that
\be\nonumber
\|\xi- \omega_b\|_{L^\infty}\approx \eps^{-2},\qquad |{\rm supp} (\xi-\omega_b)| \lesssim \varepsilon^2,  \qquad \mathsf{E}(\xi)-\mathsf{E}_b\approx \delta^2 |\log(\eps)|, \qquad |\mathsf{M}(\xi)-\mathsf{M}_b|\lesssim \delta.
\ee 
 See Figure \ref{figsym}.    In fact, Theorem \ref{damp} holds for fields $\tilde{\xi}$ in an open neighborhood of $\xi$ in $L^\infty$.
 \begin{remark}[Asymmetry of $f$-minimal flows]
 	By including the constraint on the momentum in \eqref{vp}, we can combine Theorem \ref{varprop} and Theorem \ref{damp} to see that all the $f$-minimal flows obtained as minimizers of strictly convex functionals in the set $\overline{\mathcal{O}_{\xi}}^*\cap    \{ {\mathsf E}=  \mathsf{E}(\xi)\} \cap  \{ {\mathsf M}=  \mathsf{M}(\xi)\} $ cannot be shear flows, thus providing examples of $f$-minimal flows  which do not conform to the symmetries of the domain. 
 \end{remark}
The idea behind our construction, carried out in \S 6, is to insert a large perturbation at small spatial scales in the form of  regularized point vorticies of width $\varepsilon$.  In view of the Biot-Savart law, from which the velocity is recovered from the vorticity by $u= \nabla^\perp \Delta^{-1}\omega$, these perturbations exploit the (logarithmic) singularity of the Green's function of the Laplacian in two-dimensions and thus have energy $|\log \varepsilon|$. We show that for $\varepsilon$ sufficiently small, one cannot rearrange such a configuration into a shear flow while conserving the energy. This is because shear flows are fundamentally one-dimensional objects in the sense that the Biot-Savart kernel is non-singular acting on functions of one variable. Similarly, radial flows can be excluded on the annulus by exploiting conservation of angular momentum.

In view of the containment \eqref{limOmlim}, Theorem \ref{damp} implies that the Euler solution starting from this data cannot weakly converge to a shear flow.
 These results show that the Euler dynamics cannot totally ``shear out" highly peaked coherent vortices, but they do not rule out damping to some asymmetric equilibria.   However, numerical simulations (see Figures  \ref{figeuler} \& \ref{figevo}) suggest that it is more likely that the Euler solutions relax to some time dependent (but recurrent) states. For additional discussion, see \cite{shnirelman2013long,drivas2022singularity}.
\begin{remark}[Perturbations of shear flows]
The $\omega_b$ of Theorem \ref{damp} can  be any shear flow  $u_b(x_1,x_2):=(v(x_2), 0)$ with bounded vorticity $\omega_b(x_1,x_2):= -v'(x_2)$. Our result shows that not only is the regularity important for convergence back to a shear flow, but also the proximity must be measured in a quite strong sense. In fact, our perturbation is extremely large in any $L^p$ (on vorticity) with $p>1$ and also in $L^2$ velocity.  
\end{remark}
	\begin{figure}[h]
		\includegraphics[scale=.4]{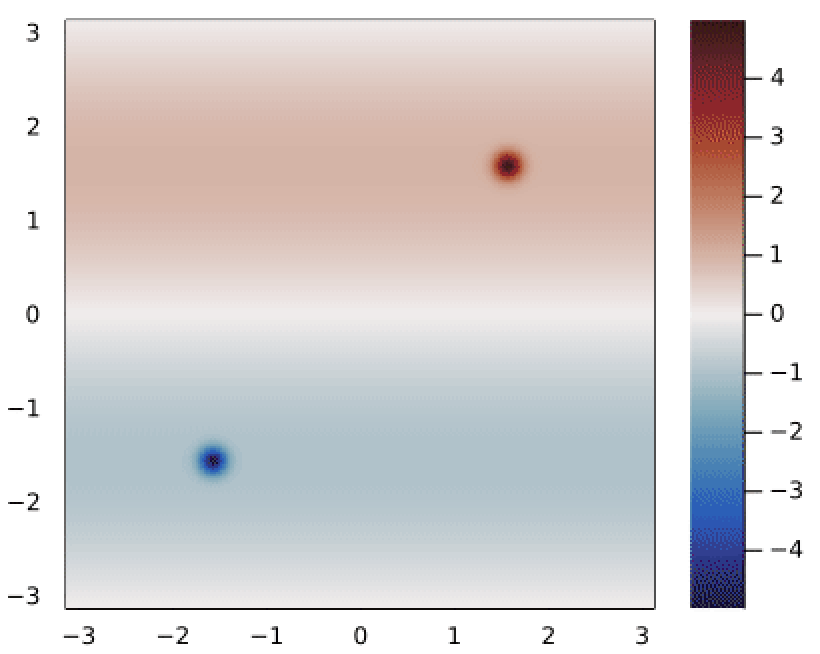} \ \ \ \ \ \includegraphics[scale=.4]{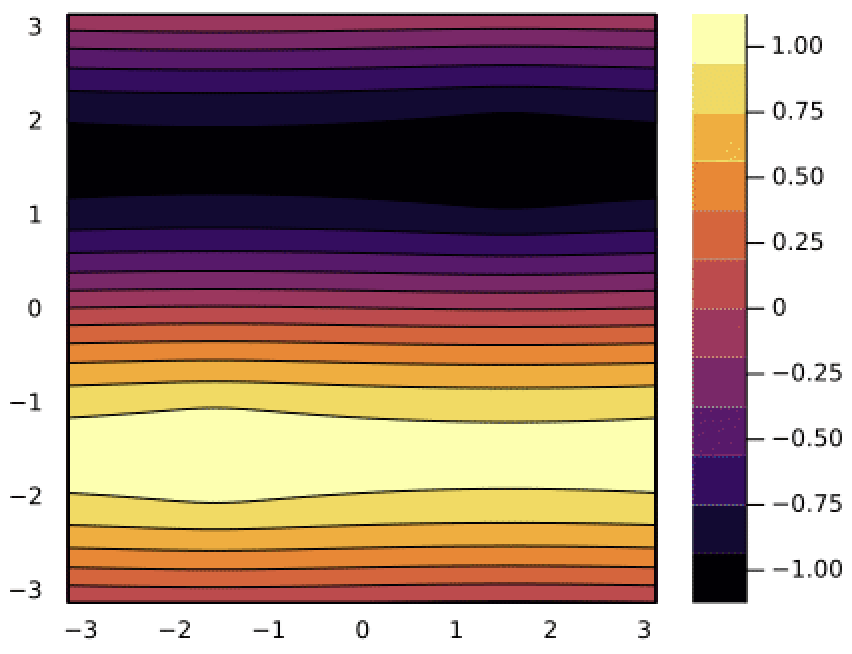}
		\caption{Example of a datum $\omega_0$ from Theorem \ref{damp} -- a perturbation (at the level of the streamfunction) of the Kolmogorov flow $\omega_{\mathsf{s}} = \sin(y)$ by two equal and opposite approximate point vortices. Vorticity colormap (left) and streamfunction contour plot (right).}
		\label{figsym}
	\end{figure}
	
	\begin{figure}[h]
		\includegraphics[scale=.4]{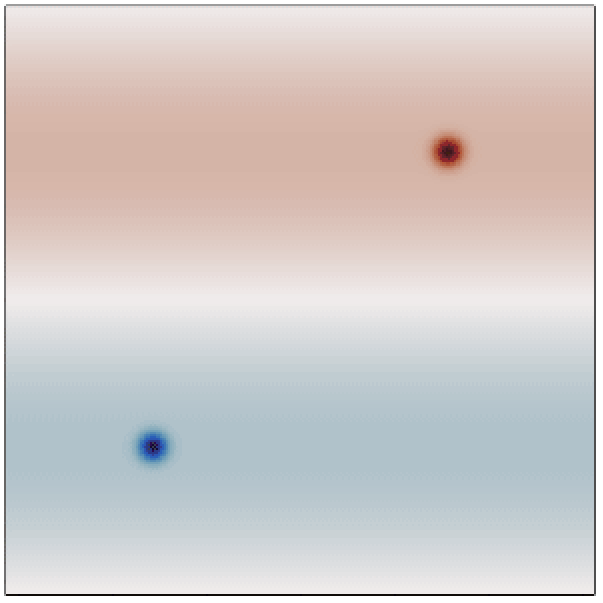} \hspace{2mm} \includegraphics[scale=.4]{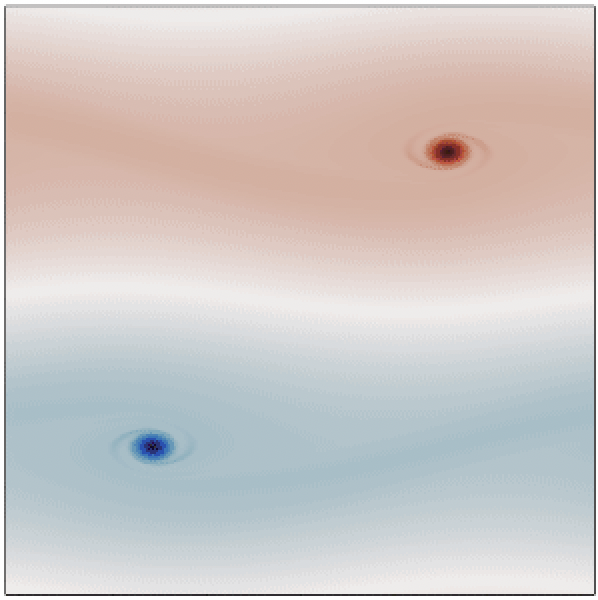}   \hspace{2mm}  \includegraphics[scale=.4]{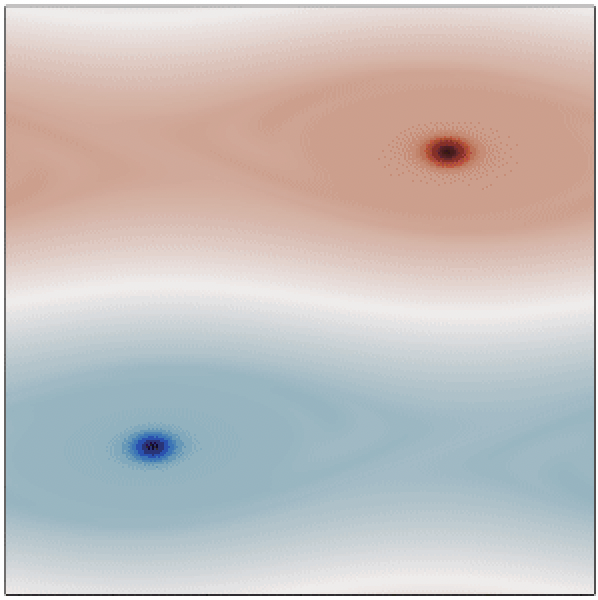}\\ \vspace{2mm}
				\includegraphics[scale=.4]{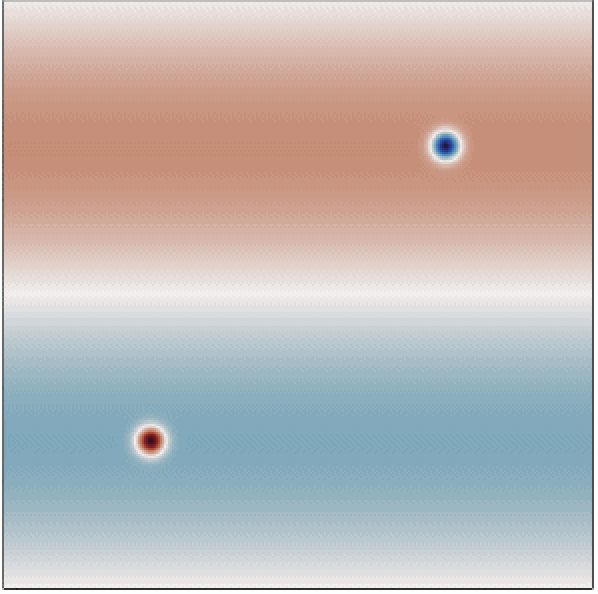} \hspace{2mm}  \includegraphics[scale=.4]{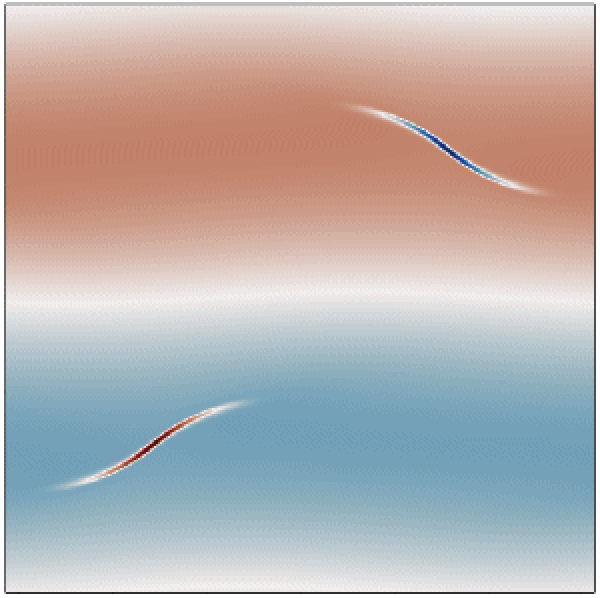}  \hspace{2mm}  \includegraphics[scale=.4]{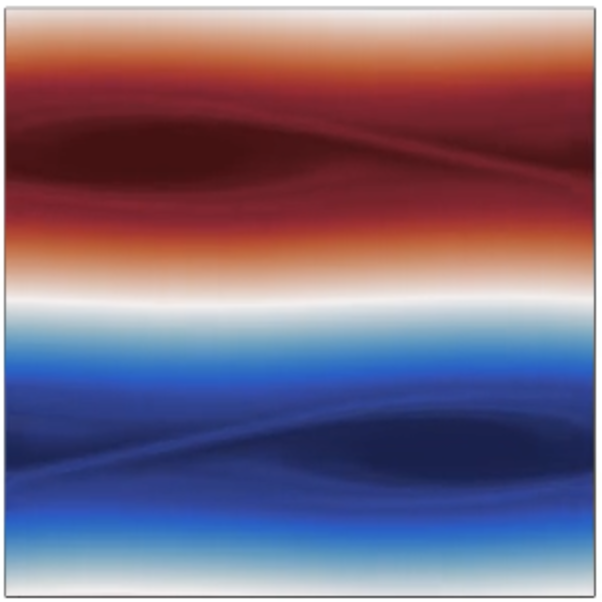}
		\caption{Direct numerical simulations \cite{CD21,C21} of the time evolution (from left to right) of initial data with localized vortices rotating with and against the background Kolmogorov shear flow under Navier-Stokes evolution with Reynolds number ${\rm Re}\approx 10^3$.  The long time behavior exhibits no tendency to return to shear.  In the case of the co-rotating vortices, it appears possible that the evolution weakly damps to a non-shear equilibrium, possibly an $f$-minimal flow.  However, in both case, periodic or quasi-period structures appear to be present at small scales it is unclear whether those can disappear in the long time limit.}
		\label{figevo}
	\end{figure}

\section{Connection to Statistical Hydrodynamics}
\label{sec:char}

The weak limit $ \overline{\omega}$ is a natural candidate object to describe  coarse-scale features of the fluid flow at late times.  In particular, as discussed in the introduction,  provided only $\omega_0\in L^\infty(M)$, one has convergence in the weak--$*$ sense, i.e.
\be\label{Onsagertheory}
\lim_{n\to\infty} \int_{M} \varphi(x) \omega(x,t_n) \rmd x = \int_{M} \varphi(x) \overline{\omega}(x)\rmd x,  \qquad \forall \varphi\in L^1(M)
\ee
for some $\overline{\omega}\in L^\infty(M)$ and some subsequence $t_n\to \infty$ as $n\to \infty$.
However, in light of the phenomenon of mixing discussed here,  fast oscillations can average out in this limit.  In particular, this convergence  does not imply for all continuous functions $f$ that $f(\omega(x,t_n))$ converges to $f(\overline{\omega}(x))$ and thus the weak limit need not have the same Casimir invariants as the initial datum $\omega_0$ (recall discussion around \eqref{bd:lscweak}).  

 In order to understand the structure of  $\overline{\omega}$, Onsager proposed a strategy based on the principles of equilibrium statistical mechanics in his foundational paper \cite{onsager1949statistical}.  Philosophically, his idea was to  study the finite ($N$) dimensional approximations (in his case, the point vortex approximation) to the infinite dimensional fluid system.  Assuming that the dynamics preserve phase space volume (which is true for point vortices) and are sufficiently ergodic\footnote{This is generally false for finite collections of point vortices \cite{khanin1982quasi}, but this may not be a fundamental issue \cite{eyink1993negative}.  However, it is also known that the infinite dimensional Euler equation is not ergodic and far from equilibrium due to the presence of wandering domains \cite{nadirashvili1991wandering} and Lyanpunov functions \cite{shnirelman1997evolution}.  See further discussion in \cite{shnirelman2013long,drivas2022singularity}. As such, it remains unclear to what extent equilibrium statistical mechanical ideas can be applied.  } so that at long times the vorticity fields are sampled according to the phase space volume available to them.  Under this assumption,  Onsager suggested  that the most probable vorticity field arising in the limit $N\to \infty$ should be that which maximizes (Boltzmann counting) entropy subject to the given initial energy, as usual in equilibrium statistical mechanics.   
 These ideas can be partly formalized in terms of concentration of measure, see the lecture notes of {\v{S}}ver{\'a}k \cite{vsverak2017aspects}.

This beautiful idea led to an explosion of work under the title ``statistical hydrodynamics" \cite{robert1991maximum,robert1991statistical,michel1994large, sommeria1991final,eyink1993negative,bouchet2008simpler} (see \cite{robert1995statistical, bouchet2012statistical} for a review).    Many of these
correspond to the following  variational problems (see Bouchet \cite{bouchet2008simpler} and Bouchet--Venaille \cite{bouchet2012statistical}) to determine the coarse-grained (weak--$*$ limit)  vorticity $\bar{\omega}$:
\be\label{variationalprobs}
\min_{\omega\ \in\  X\cap    \{ {\mathsf E}= {\mathsf E}_0\} } {\mathsf I}_f(\omega).
\ee
for a suitable choice of the Casimir $ {\mathsf I}_f$ with $f$ strictly convex (or $f$ concave and a corresponding maximization). 
 These equilibrium theories are generally consistent with the ``ultraviolet catastrophe" caused by  irreversible mixing.
 For example, Kraichnan's theory \cite{kraichnan1967inertial}, based on the principle that Euler should maximize mixing (subjectively defined here by maximally reducing the enstrophy) subject to constant energy, corresponds to the strictly convex Casimir $f(x)= \tfrac{1}{2} |x|^2$.  This is the basis for the selective decay theory (Bretherton and Haidvogel \cite{bretherton1976two}). 
For given initial vorticity $\omega_0\in X$ on a simply connected domain $M$, it predicts convergence to the stationary state $\omega$ given by the first eigenfunction of the Laplacian:
\be
\omega:=\Delta \psi = -\lambda_1 \psi, \qquad \|\psi\|_{L^2}^2 = \tfrac{ 1}{\lambda_1}{\mathsf E}_0,
\ee
where $\lambda_1= \lambda_1(M)>0$ is the smallest eigenvalue of the (negative) Dirichlet Laplacian $-\Delta$ on $M$.  On multiply connected domains, e.g. $M=\mathbb{T}^2$, for which the first eigenfunction is not unique, the theory is consistent with some slow motion on the first shell. 
Also, following Onsager, 
Joyce and Montgomery \cite{joyce1973negative} studied the case of the Boltzmann entropy for which $f(x)= - x \ln x$. On simply connected domains $M$, they predict that the end state for a given initial vorticity $\omega_0\in X$ is the stationary state $\omega$ given by the solution of the Liouville equation\footnote{It is worth remarking that, on the unit disk $\mathbb{D}$, this equation can be explicitly solved \cite{caglioti1992special} by 
\be
\omega(x) = \tfrac{1-A}{\pi} \tfrac{1}{(1-A|x|^2)^2} , \qquad A:= \tfrac{\beta}{8\pi+\beta}, \qquad \beta>-8\pi.
\ee}:
\be
\omega:=\Delta \psi = \tfrac{1}{\mathcal{Z}} e^{\beta \psi} , \qquad \mathcal{Z}= \tfrac{\int_M e^{\beta \psi(x)}\rmd x}{\int_M \omega_0(x) \rmd x} ,
\ee
where the equal energy condition implicitly determines the constant $\beta\in \mathbb{R}$ to fix the energy $ {\mathsf E}_0$.
On domains without boundary $\partial M=\emptyset$ (such as the torus or sphere) one cannot consider initial data with non-trivial sign-definite vorticity. Maximizing instead an entropy defined on the positive and negative parts of the vorticity, Joyce and Montgomery \cite{montgomery1974statistical} predicted convergence to a solution of the sinh-Poisson equation:
\be
\omega:=\Delta \psi = \tfrac{1}{\mathcal{Z}} \sinh({\beta \psi}) , \qquad \mathcal{Z}= \tfrac{\int_M  \sinh({\beta \psi}) \rmd x}{\int_M \omega_0(x) \rmd x} ,
\ee
where again the number $\beta$ is implicitly determined to ensure equal energy to $ {\mathsf E}_0$.
 A common thread through all of these variational problems is that they predict that at long times, the solutions will begin to look like some stationary fluid motion. However, these motions may not be accessible dynamically since they do not account for all known constraints on the structure of the solution given initial data\footnote{
Indeed, in view of the strict inclusion $\overline{\mathcal{O}_{\omega_0}}^*\cap    \{ {\mathsf E}= {\mathsf E}_0\} \subset   X\cap    \{ {\mathsf E}= {\mathsf E}_0\} $ for any $\omega_0\in X$ together with the fact that $\Omega_+(\omega_0)\subset \overline{\mathcal{O}_{\omega_0}}^*\cap    \{ {\mathsf E}= {\mathsf E}_0\}$, it may be that the predictions of the theories described below are dynamically inaccessible for many $\omega_0$ to which they indiscriminately apply.  As such, their domain of applicability (if any) should be carefully considered and their conclusions must be viewed with appropriate skepticism.}. 
Notable exceptions not generally conforming to \eqref{variationalprobs} are the theories of Miller, Robert, Sommeria \cite{miller1990statistical,miller1992statistical,robert1991maximum,robert1995statistical,robert1991statistical} and Turkington \cite{turkington1999statistical}. The former are richer and aim to recover the entire long time vorticity distribution (in the sense of Young measures).  We shall explain this theory in detail at the end of the section. Turkington's theory is instead different from a statistical hydrodynamic point of view to the MRS one. As we explain in Remark \ref{rem:truncation}, it is interesting to notice that his finite dimensional approximation of Euler's equations gives rise to elements in $\overline{\cO_{\omega_0}}^*$. Both theories predict an end state that could be dynamically accessible.

Taking one step closer to the Euler equation, Shnirelman \cite{shnirelman1993lattice} considered a variational problem akin to 
\be\label{variationalprobs2}
 \min_{\omega\ \in\  \overline{\mathcal{O}_{\omega_0}}^*\cap    \{ {\mathsf E}= {\mathsf E}_0\} }  {\mathsf I}_f(\omega).
\ee
In fact, his ideas were not stated directly in these terms but can be connected through our Theorem \ref{varprop}.  Roughly, his theory corresponds to maximizing ``mixing" in an objective sense (without choosing a particular Casimir such as enstrophy or entropy) subject to being in the weak-$*$ closure of the orbit in the diffeomorphism group (on slices of equal energy)  $\overline{\mathcal{O}_{\omega_0}}^*\cap    \{ {\mathsf E}= {\mathsf E}_0\}$.  Thus, Shnirelman's predictions are not obviously dynamically inaccessible in the same way as some of the predictions  arising from other statistical hydrodynamics theories are.

To understand this theory, it is important to study the structure of $\overline{\mathcal{O}_{\omega_0}}^*\cap    \{ {\mathsf E}= {\mathsf E}_0\}$. In fact, this problem was considered in different mathematical contexts \cite{brenier2003p,chong1976doubly,chong1974some,ryff1965orbits,vsverak2012selected,shnirelman1993lattice}. In \S \ref{sec:charOmset} we present a self-contained description of the different characterizations which we will exploit. Let us recall here a particular characterization of the weak-$*$ closure of the orbit of a scalar function in the group of area preserving diffeomorphisms $\mathscr{D}_\mu(M)$, used  in \cite{brenier2003p,chong1976doubly,shnirelman1993lattice}.  Denote the collection of evaluation maps along area preserving  diffeomorphisms  by
\be
{\mathscr{E}_\mu(M)}:= \{i_\varphi \ :\ \varphi  \in \mathscr{D}_\mu(M)  \} 
\ee
where $i_\varphi $ is the evaluation map, i.e. if $f:M\to \mathbb{R}$ then $(i_\varphi f)(x) = \int_M f(y) \delta(y- \varphi(x)) \rmd y = f(\varphi(x))$. We associate to $i_\varphi$ the positive measure $\delta(y-\varphi(x) )\rmd y/|M|$. The following is established in \cite{brenier2003p,brenier2020examples}:
\begin{proposition} 
	\label{prop:weak} We have
	\be
	\label{eq:weakK}\overline{{\mathscr{E}_\mu(M)}}^* = {\mathscr{K}(M)},
	\ee
	where ${\mathscr{K}(M)}$ is the convex space of polymorphisms or bistochastic  operators
	\be
	\label{def:bist}
	{\mathscr{K}(M)}:= \left\{K \in \mathcal{P}(M\times M) \, :\, \ \int_{y \in M} K(\dd x,\dd y) = \dd x, \ \int_{x \in M} K(\dd x,\dd y) = \dd y \right\},
	\ee
    where $\mathcal{P}(M\times M)$ denotes the  space of probability measures.
\end{proposition}

\begin{remark}[Examples of polymorphisms]
	\label{rem:poly}
	Bistochastic operators are the infinite dimensional extension of bistochastic matrices. Few important examples are the following:
	\begin{itemize}
		\item[\textbf{1)}] Let $\varphi  \in \mathscr{D}_\mu(M)$. Then the insertion operator $K_\varphi=i_\varphi $ is bistochastic and $(K_\varphi\omega)(x)= \omega(\varphi(x))$
		\item[\textbf{2)}] The \textit{complete mixing} operator $K_{\textsf{mix}}$ given by 
		\begin{equation}
			(K_{\textsf{mix}} \omega)(x)=\frac{1}{|M|}\int_{M} \omega(y)\dd y
		\end{equation}
		is bistochastic. On $M=\mathbb{T}^2$, this operator is the projection onto the zero Fourier mode.
		\item[\textbf{3)}]  On $M=\mathbb{T}^2$, some frequency cut-offs define bistochastic operators, for example the Fejér kernel:
		\begin{equation}
			F_N(x)=\frac{1}{(2\pi)^2}\sum_{k_1,k_2=-N}^{N} \left(1-\frac{|k_1|}{N}\right)\left(1-\frac{|k_2|}{N}\right)e^{i(k_1x_1+k_2x_2)}.
		\end{equation}
		In view of  $\int_{\TT^2}F_N(x)\rmd x=\widehat{F}(0)=1$, this kernel has  the following properties:
		\begin{itemize}
			\item[a)] $F_N(x)\geq 0$,
			\item[b)]$\widehat{F_N}(k)=\begin{cases}\left(1-\frac{|k_1|}{N}\right)\left(1-\frac{|k_2|}{N}\right) &\qquad 1\leq |k_1|,|k_2|\leq N \\
				0 &\qquad \max\{ |k_1|,|k_2|\}> N
			\end{cases}$,
			\item[c)] $\int_{\TT^2}F_N(x)\dd x=1$.
		\end{itemize}
		 Given $\omega\in L^2$, define a frequency cut-off as  follows
		\begin{equation}
			\label{def:Kn}
			(K_N\omega)(x)=\int_{\TT^2} F_N(x-y)\omega(y)\dd y= (F_N*\omega)(x).
		\end{equation}
		Thanks to the properties of $F_N$, it can be verified that $K_N$ is a bistochastic operator (see \cite[\S 3.1]{turkington1999statistical}). 
	\end{itemize}
\end{remark}
The set of polymorphisms is relevant to the weak-$*$ closure of the orbit since (see Proposition \ref{prop:sverak} in \S \ref{sec:charOmset}), given any $\omega_0\in X$, we have
\begin{align}\label{def:shnirset1}
	\overline{\mathcal{O}_{\omega_0}}^*&=\{\omega \in X: \ \omega=K\omega_0\ \text{ for } K\in\mathscr{K}\} .
\end{align}

Shnirelman uses the characterization  \eqref{def:shnirset1} of $\overline{\cO_{\omega_0}}^*$ to impart a preordering
in $\overline{\cO_{\omega_0}}^*\cap\{ \mathsf{E} =\mathsf{E}_0\}$;
\begin{definition}
	\label{def:orderShnirelman}
	Given $\omega_1, \omega_2 \in \overline{\mathcal{O}_{\omega_0}}^*\cap \{ {\mathsf E}= {\mathsf E}_0\}$, we say that $\omega_1 \preceq_s \omega_2$ if there exists a polymorphism $K \in \mathscr{K}$ such that $\omega_1 = K\omega_2$. An $\omega^*\in\overline{\mathcal{O}_{\omega_0}}^*\cap \{ {\mathsf E}= {\mathsf E}_0\}$ is \textit{minimal in the sense of Shnirelman} if for all $\omega$ such that $\omega\preceq_s\omega^*$ then $\omega^*\preceq_s \omega$.
\end{definition}
Minimal elements (flows) are defined in the same way as Definition \ref{def:convexminimal}, only now using this preorder.  
However, we will prove in \S \ref{seccharmin} that if you are $f$-minimal in the sense of Definition \ref{def:po}, then you are also minimal in the sense of Shnirelman, as can be deduced from the following.
\begin{lemma}\label{prop:eqstric}
	Given $\omega_2 \in X$ and $K\in \mathscr{K}$, let $\omega_1=K\omega_2$.  There exists $\widetilde{K}\in \mathscr{K}$ such that $\omega_2=\widetilde{K}\omega_1$ if and only if there exists a strictly convex function   
    $f:\mathbb{R}\to \mathbb{R}$ such that $\mathsf{I}_f(\omega_1)=\mathsf{I}_f(\omega_2)$.
\end{lemma}
Unfortunately, the lemma above does not imply an equivalence between the two different notions of minimal flows. Indeed, we cannot guarantee that Shnirelman's minimal elements globally minimize a given Casimir. For instance, a Shnirelman minimal state might act as a ``saddle point'' or local minimum, which is a configuration where any further mixing would change the energy even though the global minimum of the Casimir has not yet been reached. We are currently unable to rule out this scenario, nor can we produce a concrete example where it occurs.

The intuition that $f$-minimal flows are maximally mixed, quantified by the conservation of a given Casimir in their weak-$*$ closure with our definition, can also be interpreted using bistochastic operators. The application of a bistochastic operator $K$ could mix $\omega^*$, but  mixing is an irreversible process that prevents us from recovering $\omega^{*}$ from $K\omega^{*}$. On the other hand, for a minimal flow, we can always recover the initial state. We are therefore excluding any  irreversible mixing of $\omega^{*}$. Thus, the class of available transformations of a minimal flow reduces to a subset of measure-preserving maps (not necessarily diffeomorphisms). In fact, Lemma \ref{prop:eqstric} shows that if $\omega=K\omega^*$ and $\omega^*=\widetilde{K}\omega$, then $\omega$ and $\omega^*$ are equimeasurable. However, it remains unclear how to quantify the degree of mixing exhibited by Shnirelman's minimal elements, as there is no practical tool (such as a specific Casimir) to compare them with other states in $\overline{\mathcal{O}_{\omega_0}}^*\cap\{\mathsf{E}=\mathsf{E}_0\}$.
\begin{remark}[Truncations of Euler]
	\label{rem:truncation}
	We point out a natural connection between Turkington's theory \cite{turkington1999statistical} and that of Shnirelman \cite{shnirelman1993lattice}. In \cite{turkington1999statistical}, Turkington defines the finite dimensional approximation of the 2D Euler equations through a Fourier truncation obtained via the  Fejér kernel, see \eqref{def:Kn}, which is a bistochastic operator. Thus, the truncated dynamics used by Turkington is in $\overline{\cO_{\omega_0}}^*$. Moreover, thanks to standard properties of the Fejér kernel, it belongs to   $\overline{\cO_{\omega_0}}^*\cap\{ \mathsf{E} =\mathsf{E}_0^{(N)}\}$ where $\mathsf{E}_0^{(N)}\approx \mathsf{E}_0+C/N$ where $N\to \infty$ is the truncation parameter.
    The truncation through Fejér kernel seems a natural choice to investigate long-time behavior questions, both theoretically and numerically. In fact, a standard truncation in frequency defined as $\widehat{T}_N(k)=\chi_{[-N,N]}(k)$ is not bistochastic since the associated kernel is not a nonnegative measure. In particular, a truncation with $T_N$ changes the sign of the initial vorticity and therefore it might give rise to a completely different dynamics, since the sign of the vorticity has a fundamental role in the evolution (as the merging of likely signed vortices for instance).  We remark also that Zeitlin's \cite{zeitlin1991finite} geometric quantization based on an approximation to the group of area preserving diffeomorphism has proved to be very useful in numerical simulations of long time 2D Euler flows   \cite{modin2020casimir}.
\end{remark}

Finally, we describe a more complete picture of the vorticity at late times (beyond the weak-$*$ limit $\bar\omega$) and its connection to the Miller, Robert and Sommeria theory as well as to $f$-minimal flows. 
 To understand these theories, we first recall that for uniformly bounded vorticity fields, the \emph{fundamental theorem of Young measures} (see \cite{valadier1994course,florescu2012young}) guarantees the existence of a (measurably) parametrized measure $ \nu_x(\rmd \sigma)$ such that for any $ f\in C([-\mathsf{m},\mathsf{m}])$, 
\be\label{youngmeasurelim}
\lim_{n\to\infty} \int_{M} \varphi(x) f(\omega(x,t_n))  \rmd x= \int_{M} \varphi(x) \int_{-\mathsf{m}}^\mathsf{m} f(\sigma) \nu_x(\rmd \sigma) \rmd x , \qquad \forall\varphi\in L^1(M),
\ee
with $\mathsf{m}=\|\omega_0\|_{L^\infty(M)}$. It turns out that this Young measure defined by \eqref{youngmeasurelim} always assumes the form
\be \label{MRmeasure}
\nu_x(\rmd \sigma) = \rho(x,\sigma) \rmd \sigma.
\ee
In fact, by Proposition \ref{prop:weak} and equation \eqref{eq:convergence}, $\rho$ is  represented in terms of a bistochastic kernel
\be\label{rhodef}
 \rho(x,\sigma)=  \int_M \delta (\sigma -\omega_0(y)) K(x,y) \rmd y
\ee
for some bistochastic kernel $K\in \mathscr{K}$ (which encodes dependence on the subsequence of long times). 

Having introduced the Young measure $\nu_x(\rmd \sigma)$, the convergence \eqref{Onsagertheory} holds with the weak limit 
\be \label{barom}
 \overline{\omega}(x)=  \int_{-\mathsf{m}}^\mathsf{m} \sigma \nu_x(\rmd \sigma) = \int_M \omega_0(y) K(x,y) \rmd y,
\ee
in accord with \eqref{def:shnirset1}.
We note that from \eqref{rhodef} and \eqref{barom}, there may be many measures having the same average vorticity with different ``fine-scale" behaviors. Since energy is weak-$*$ continuous, the energy of $\overline{\omega}$ is the same as the data
\be\label{energyconst}
\mathsf{E}[\overline{\omega}] :=-\frac{1}{2} \int_M \overline{\psi}(x)\overline{\omega}(x)\rmd x = \mathsf{E}[\omega_0].
\ee
In view of the properties of marginals of bistochastic kernels, the distribution is normalized: 
\begin{align}\label{normalization}
 N[\rho](x):=  \int_{-\mathsf{m}}^\mathsf{m} \rho(x,\sigma) \rmd \sigma &=  1.
\end{align}
Moreover, the vorticity distribution function is preserved in the sense that the marginal satisfies
\begin{align}\nonumber
 \mathsf{D}[\rho](\sigma):= \int_M \rho(x,\sigma) \rmd x &=  \int_M \delta (\sigma -\omega_0(y))  \rmd y\\
 & = \frac{\rmd}{\rmd \sigma} \int_M \chi_{\{\omega_0(x)\leq \sigma\}} \rmd x=: \mathsf{d}[\omega_0](\sigma). \label{distribution}
\end{align}
Miller, Robert and Sommeria suggested that the long-time vorticity distribution is a maximizer of an entropy (minimizer of negentropy) subject to the constraints of energy \eqref{energyconst}, normalization \eqref{normalization} and distribution function \eqref{distribution}.  The entropy quantity measures the number of ``microscopic" vorticity fields which are compatible with a distribution $\rho(x,\sigma)$.  The relevant measure (derived from ``first principles" in finite dimensions) is assumed\footnote{Aside from the foundational issue of ergodicity which must be established to justify its use for perfect fluids,  there is some debate as to whether it is justified to use this counting entropy to understand the long time behavior of real-world flows for which non-ideal effects, however slight, are present.   In particular, it is not clear that the entire distribution function \eqref{distribution} should be remembered in the formulation of a long time theory.  Turkington argued for a modified entropy which accounts for some non-ideal effects \cite{turkington1999statistical}. We remark however that for arbitrarily long time horizons, it can be shown that inviscid limits of Navier-Stokes solutions \emph{do} remember their intial vorticity distribution functions \cite{constantin2022inviscid}.} to be the Maxwell-Boltzmann entropy
\be\label{boltent}
\mathcal{S}[\rho]:=- \int_M   \int_{-\mathsf{m}}^\mathsf{m}  \rho(x,\sigma) \log \rho(x,\sigma) \rmd \sigma\rmd x.
\ee
The theory predicts that the most probable distribution $\bar{\rho}(x,\sigma)$ will be the minimizer of the negentropy \eqref{boltent} subject to the $\bar{\rho}$ satisfying the constraints \eqref{energyconst}, \eqref{normalization} and \eqref{distribution}, namely 
\begin{equation}
	\label{MRSoriginal}
	\mathcal{S}[\bar{\rho}]=\min\{-\mathcal{S}[\rho]\, :\,\mathsf{E}[\overline{\omega}]= \mathsf{E}[\omega_0], \quad \mathsf{D}[\rho](\sigma)= \mathsf{d}[\omega_0](\sigma), \quad N[\rho](x)=  1 \}
\end{equation}

Notice that in \eqref{MRSoriginal} the information on all ideal invariants is retained at the level of the  predicted equilibrium distribution \eqref{MRmeasure}, \eqref{rhodef} but not the weak limit \eqref{barom}.   In particular, the theory is consistent with strict inequalities  \eqref{bd:lscweak} due to irreversible mixing.

The variational problem \eqref{MRSoriginal} has been considered by many authors, see \cite{bouchet2008simpler,bouchet2012statistical,chavanis2009dynamical} and references therein. Indeed, it can be explicitly solved. Following \cite{chavanis2009dynamical}, a minimizer to \eqref{MRSoriginal} can be found as a critical point of the Lagrangian 
\begin{equation}
	\label{def:LMRS}
	\mathcal{L}(\rho,\beta,\alpha,\zeta):=	-\mathcal{S}[\rho]-{\beta}(\mathsf{E}[\overline{\omega}]- \mathsf{E}[\omega_0])- \widetilde{\alpha}(\sigma)(\mathsf{D}[\rho](\sigma)-\mathsf{d}[\omega_0](\sigma))-\widetilde{\zeta}(x) (N[\rho](x)- 1).
\end{equation}
Computing the first variation of $\mathcal{L}$ with respect to the first variable we have 
\begin{align}\nonumber
	\frac{\dd}{\dd \eps}(\mathcal{L}(\rho+\eps h,\beta,\alpha,\zeta))|_{\eps=0}=\,&\int_{M}\int_{-\mathsf{m}}^{\mathsf{m}}(\ln(\rho(x,\sigma))+1+{\beta}\sigma \overline{\psi}(x))h(x,\sigma)\dd x \dd \sigma\\
	&\quad -\widetilde{\alpha}(\sigma)\int_Mh(x,\sigma)\dd x-\widetilde{\zeta}(x)\int_{-\mathsf{m}}^{\mathsf{m}}h(x,\sigma)\dd \sigma=0.
\end{align}
Since $|M|$ and $\mathsf{m}$ are both finite, integrating in $x,\sigma$ the identity above and defining $\alpha=\widetilde{\alpha}/(2\mathsf{m}|M|), \, \zeta=\widetilde{\zeta}/(2\mathsf{m}|M|)$, by the arbitrariness of $h$ we find that a critical point $\bar{\rho}$ is given by
\begin{equation}
	\bar{\rho}(x,\sigma)=\frac{1}{Z(x)}g(\sigma)\e^{-\beta \sigma\overline{\psi}(x)}, \qquad g(\sigma)=\e^{-\alpha(\sigma)}, \qquad Z(x)=\e^{1+\zeta(x)}.
\end{equation}
The Lagrange multipliers $\beta,\alpha, \zeta$ are found as usual by imposing the constraints \eqref{energyconst}, \eqref{normalization} and \eqref{distribution} (which arise by taking the first variation of $\mathcal{L}$ with respect to the other variables). The normalization \eqref{normalization} imposes that 
\begin{equation}\label{normalization1}
	Z(x)=\int_{-\mathsf{m}}^{\mathsf{m}}g(\sigma)\e^{-\beta \sigma\overline{\psi}(x)}\dd \sigma.
\end{equation}
Thus, the {coarse grained vorticity} $\overline{\omega}$ associated to the distribution $\bar{\rho}$ is 
\begin{equation}
	\label{eq:defF}
	\overline{\omega}=\int_{-\mathsf{m}}^{\mathsf{m}}\sigma\bar{\rho}(x,\sigma)\dd \sigma=\frac{\int_{-\mathsf{m}}^{\mathsf{m}}\sigma g(\sigma)\e^{-\beta \sigma\overline{\psi}(x)}\dd \sigma}{\int_{-\mathsf{m}}^{\mathsf{m}}g(\sigma)\e^{-\beta \sigma\overline{\psi}(x)}\dd \sigma}:=F(\overline{\psi}),
\end{equation}
which readily implies that $\overline{\omega}$ is a stationary solution of the Euler equations. Using the formula \eqref{normalization1}, the  functional relation \eqref{eq:defF} can be written as 
\begin{equation}
	\overline{\omega}=-\frac{1}{\beta}\frac{\dd}{\dd \psi}\ln(Z).
\end{equation}
Denoting $\overline{\omega^2}=\int_{-\mathsf{m}}^\mathsf{m}\sigma^2\bar{\rho}(x,\sigma)\dd \sigma$, a direct computation shows that 
\begin{align}
	&F'(\overline{\psi})=-\frac{1}{\beta}\frac{\dd^2}{\dd \psi^2}\ln(Z)=-\beta(\overline{\omega^2}-\overline{\omega}^2).
\end{align}
Namely, the function $F$ is related to the variance of the distribution $\bar{\rho}$. By the Jensen's inequality 
\begin{equation}
	\label{eq:inqF}
	\overline{\omega^2}(x)-\overline{\omega}^2(x)\geq 0,	\qquad \text{ a.e. in } M,
\end{equation} 
so $F$ is a monotone function. 

\begin{remark}
	In the literature \cite{bouchet2008simpler,chavanis2002statistical}, it is often assumed that the inequality in \eqref{eq:inqF} is strict or  that the function $F$ is strictly monotone. However, it is not always possible to conclude that $F$ is strictly monotone.  Indeed, assume that $F$ is strictly monotone. Then, in \eqref{eq:inqF} we have a strict inequality and integrating in $x$ \eqref{eq:inqF}, using \eqref{rhodef}-\eqref{barom}, we obtain 
	\begin{equation}
		0< \int_M (\overline{\omega^2}(x)-\overline{\omega}^2(x))\dd x=\iint_{M\times M} \omega_0^2(y)K(x,y)\dd x \dd y-\int_M\overline{\omega}^2(x)\dd x=\int_{M} (\omega_0^2(x)-\overline{\omega}^2(x))\dd x.
	\end{equation}
	 However, since $\overline{\omega}=\bar{K}\omega_0$, if $\omega_0$ is an $f$-minimal flow the previous inequality is not possible in view of Lemma \ref{prop:eqstric}. Thus, $F$ cannot be strictly monotone if  $\omega_0$ is an $f$-minimal flow. Moreover, if $\omega_0$ is an $f$-minimal flow the only possibility in \eqref{eq:inqF} is that equality holds a.e. in $M$, which implies
	 \begin{equation}
	 	\int_{M} \omega_0^2(x)=	\int_{M}\overline{\omega}^2(x)\dd x.
	 \end{equation}
 As shown in the proof of Lemma \ref{prop:eqstric} (see \S \ref{subsec:minimin}), a consequence of the latter identity is that $\omega_0$ and $\overline{\omega}$ are equimeasurable. Since  $F$  is constant in this particular case, we have also $\overline{\omega}$ is a constant and therefore $\omega_0$ must be a constant. But it is not true in general that  $f$-minimal flows are constant. This apparent paradox has a simple solution. In \eqref{def:LMRS} the Lagrange multiplier in front of $\mathcal{S}$ has been omitted, implicitly assuming a nondegeneracy condition for the problem \eqref{MRSoriginal}. However, if we start with an $f$-minimal flow $\omega_0$, the problem can be degenerate and the coarse grained vorticity $\bar{\omega}$ must be associated to a rearrangement of $\omega_0$ with the same energy, since no other elements are present in  $\overline{\cO_{{\omega_0}}}^*\cap \{\mathsf{E}=\mathsf{E}_0\}$ for a $f$-minimal flow. The possibility of such degenerate behavior of the variational problem is included in our characterization \textit{(ii)} in Theorem \ref{varprop}, in which we cannot exclude the case $\alpha=0$.
 
 We conclude with the observation that in simulations of 2d Euler at long times, it appears that there can be regions of constant vorticity embedded in the non-constant background.  If such possibilities persist indefinitely and in a weak-$*$ sense, it would exhibit the necessity for allowing  not strictly monotone $F$.
\end{remark}

On the other hand, if the function $F$ defined in \eqref{eq:defF} is strictly monotone, we have the following observation due to Bouchet \cite{bouchet2008simpler} (see also \S 7.4 of Chavanis \cite{chavanis2002statistical}):
\begin{proposition}
Let $\overline{\omega}=F(\overline{\psi})$ be given as in \eqref{eq:defF}. Assume that $F'>0$. Let $G(s)=\int_{-\mathsf{m}}^sF^{-1}(s)\dd s$. Then, $\overline{\omega}$ is a minimizer of the problem 
\begin{equation}
	\label{def:vpbar}
	\mathsf{I}_G(\overline{\omega})=\min\{ \mathsf{I}_G(\omega) \, :\, \omega \in X, \quad \mathsf{E}[\omega]=\mathsf{E}[\omega_0]\}.
\end{equation}
Moreover, $\overline{\omega}$ is a $G$-minimal flow.
\end{proposition}
\begin{proof}
	The fact that $\overline{\omega}$ is a minimizer of \eqref{def:vpbar} directly follows by the Lagrange multiplier rule, which give us that a solution of \eqref{def:vpbar} can be found by solving
	\begin{equation}
		G'(\widetilde{\omega})-\beta \widetilde{\psi}=F^{-1}(\widetilde{\omega})-\beta \widetilde{\psi}.
	\end{equation} 
Indeed, $\widetilde{\omega}=\overline{\omega}=F(\overline{\psi})$ and $\beta=1$ satisfy the previous identity. To prove that $\overline{\omega}$ is a $G$-minimal flow, let $\bar{K}$ be a bistochastic operator that represents the Young's measure $\overline{\rho}$ as in \eqref{rhodef}, where $\overline{\rho}$ is the solution of \eqref{MRSoriginal}. Appealing to \eqref{rhodef}, we get  $\overline{\omega}=\bar{K}\omega_0$. Since we also know that $\mathsf{E}[\omega]=\mathsf{E}[\omega_0]$ then $\overline{\omega}\in \overline{\cO}^*_{\omega_0}\cap \{\mathsf{E}=\mathsf{E}_0\}$ and in particular is a minimizer of the strictly convex functional $\mathsf{I}_G$ in $\overline{\cO}^*_{\omega_0}\cap \{\mathsf{E}=\mathsf{E}_0\}$, thus solving a variational problem as in \eqref{vp}. By Theorem \ref{varprop}, we infer that $\overline{\omega}$ is a $G$-minimal flow. Note that in this particular case we can choose $\Phi=0, \alpha=1, \beta=1, \gamma=0$ in point \textit{(ii)} of Theorem \ref{varprop}.
\end{proof}
\begin{remark}
It would be interesting to determine whether or not the vorticity $\overline{\omega}$ arising from a Miller, Robert and Sommeria theory always corresponds to  an $f$-minimal flow also when $F$ is not strictly monotone.
\end{remark}

\section{Some Questions}

In general, statements about the set  $\mathcal{O}_{\omega_0,{\mathsf E}_0}$ cannot say anything definitive about long time dynamics of Euler.  However, some qualitative information may be gained if certain questions are answered. Specifically

\begin{question} [Weak but not strong relaxation to equilibrium]
Can one characterize $\omega_0\in X$ such that the resulting Euler solution $S_t(\omega_0)$ cannot converge strongly in $L^2$ to equilibrium.   This is related to the works \cite{ginzburg1992topology,ginzburg1994steady,izosimov2017characterization} since, to answer the above,  one needs to  show the strong $L^2$ closure of set $\mathcal{O}_{\omega_0,{\mathsf E}_0}$ contains no stationary Euler solutions.  The quoted papers consider this question without the closure.
\end{question}

\begin{question} [Isolation from smooth stationary states]
Can one characterize (or give conditions on) $\omega_0\in X$ such that there exists a smooth stationary Euler solution in $ \overline{\mathcal{O}_{\omega_0}}^*\cap    \{ {\mathsf E}= {\mathsf E}_0\}$?  {More generally, under which conditions are minimizers of \eqref{vp} smoother (or rougher) than the data $\omega_0$?  Note that smoothing takes place if $\omega_0$ is comprised of, say, two patches of constant vorticity.  See Appendix \ref{finiteconst}.}
\end{question}

\begin{question} [Vortex mergers]
When do minimizers of \eqref{vp} have different vortex line topology than that of $\omega_0$?
\end{question}

The next two questions will indirectly concern special stable stationary Euler solutions; constant vorticity states  ${\omega}_*$ and Arnold stable states  $\omega_{\mathsf{A}}$.  It is easy to see that constant vorticity states are the unique functions in $X$ having the property that
\be
\mathcal{O}_{{\omega}_*}  = \mathcal{O}_{{\omega}_*,{\mathsf E}_*}= \overline{\mathcal{O}_{{\omega}_*,{\mathsf E}_*}}^*= \overline{\mathcal{O}_{\omega_0}}^*\cap    \{ {\mathsf E}= {\mathsf E}_*\} = \Omega_+(\omega_*) = \{ {\omega}_*\} .
\ee
 On the other hand, Arnold stable states $\omega_{\mathsf{A}}$ have the property $\Omega_+(\omega_{\mathsf{A}}) =\overline{\cO_{\omega_{\mathsf{A}},\mathsf{E}_{\mathsf{A}}}}^*=\{\omega_{\mathsf{A}}\}$, since any mixing of them must change the energy.

\begin{question} 
Does there exist $\omega_*\in X$ such that 
\be
 \overline{\mathcal{O}_{{\omega}_*,{\mathsf E}_*}}^*\neq \overline{\mathcal{O}_{\omega_*}}^*\cap    \{ {\mathsf E}= {\mathsf E}_*\}?
\ee
\end{question}
A positive answer to the above question could imply that some of the maximal mixing states discussed here are dynamically inaccessible.

\begin{question} [Structure in perturbative regime]
Can anything more be said about the structure of the sets $ \overline{\mathcal{O}_{\omega_0}}^*\cap    \{ {\mathsf E}= {\mathsf E}_0\}$ and, in particular, the minimizers in \eqref{varprop} for $\omega_0\in X$  which are perturbations of constant vorticity fields or Arnold stable steady states?
\end{question}

\begin{question} [weak-$*$ Ergodicity?]
Other than Arnold stable steady states and those with constant vorticity, is it ever the case that  $\Omega_+(\omega_0) =  \overline{\mathcal{O}_{\omega_0}}^*\cap    \{ {\mathsf E}= {\mathsf E}_0\}$ for some $\omega_0\in X$? 
\end{question}

We remark that if a $\omega_0$ is a saddle-type critical point of the energy on $\mathcal{O}_{\omega_0}$ (as opposed to a minimum or maximum as in the Arnold stable case), then $\{\omega_0\}=\Omega_+(\omega_0) \neq  \overline{\mathcal{O}_{\omega_0}}^*\cap    \{ {\mathsf E}= {\mathsf E}_0\}$.

Numerical simulations and physical experiments often point to the conclusion that at long times the solution does not become truly stationary, but rather enters some ordered time dependent regime.  We ask

\begin{question} [Existence of recurrent solutions]
For which $\omega_0\in X$ do there exist time periodic or quasi-periodic  Euler solutions in $\overline{\mathcal{O}_{\omega_0}}^*\cap    \{ {\mathsf E}= {\mathsf E}_0\}$?  
\end{question}
Note that these sets \emph{always} contain at least one $L^2$--precompact Euler orbit since  $\Omega_+(\omega_0) \subset\overline{\mathcal{O}_{\omega_0}}^*\cap    \{ {\mathsf E}= {\mathsf E}_0\}$ and those are known to contain them (see Theorem \ref{contL2orb} in Appendix \ref{secomegalim},  due to {\v{S}}ver{\'a}k). These precompact orbits are, in fact, minimizers of some strictly convex functional in the set $\Omega_+(\omega_0)$. However, on this set (unlike on $\overline{\mathcal{O}_{\omega_0}}^*\cap    \{ {\mathsf E}= {\mathsf E}_0\}$) there is a-priori no reason why such a minimizer should be a steady state. 

Finally, observations indicate that, over time, the vorticity/velocity fields that emerge are far less diverse than the phase space of the dynamics.  This apparent ``decrease of entropy" demands a theoretical explanation. We ask

\begin{question}[Entropy Decrease for Euler]\label{entquest}
Is  $\Omega_+(X):= \bigcup_{\omega_0\in X} \Omega_+(\omega_0)$ a strict subset of $X$?  The conjecture of Shnirelman is that the set $\Omega_+(X)$ consists of all Euler orbits in $X$ which are precompact in $L^2$.  Together with the conjecture of {\v{S}}ver{\'a}k  that generic orbits in $X$ should not be compact, this suggests that $\Omega_+(X)$ should be a ``meagre set" in $X$.
\end{question}

We point the reader to the works \cite{modin2020casimir,modin2021canonical,modin2021integrability} for very interesting conjectures related to the structure of  typical members of $\Omega_+(X)$.  See further discussion in \cite{drivas2022singularity}.

\section{The weak-$*$ closure of the orbit}
\label{sec:charOmset}

Here we prove the following characterizations of the weak-$*$ closure of the orbit:
\begin{proposition} \label{prop:sverak}
	Consider $X, \, \mathscr{K} $ as in \eqref{def:X} and \eqref{def:bist} respectively. Given any $\omega_0\in X$, we have
\begin{align}\label{def:shnirset}
\overline{\mathcal{O}_{\omega_0}}^*&=\{\omega \in X: \ \omega=K\omega_0\ \text{ for } K\in\mathscr{K}\},\\ \label{def:sverakset}
&= \left\{\omega\in X \ : \ \int_M \omega \, \dd x=\int_M \omega_0\, \dd x, \ \ \text{and }  \ \  \int_M (\omega-c)_+\, \dd x\leq \int_M (\omega_0-c)_+\, \dd x \ \ \text{ for all } \ \ c\in \mathbb{R}  \right\},\\\label{def:casimirset}
&= \left\{\omega\in X \ : \ \int_M \omega\, \dd x =\int_M \omega_0\, \dd x, \ \ \text{and }  \ \  \int_M f(\omega)\, \dd x\leq \int_M f(\omega_0)\, \dd x\ \ \text{ for any convex } f  \right\}.
\end{align}
\end{proposition}

	The description of the set $\overline{\cO_{\omega_0}}^*$ has been a classical topic in rearrangement inequalities theory and the characterizations \eqref{def:shnirset}-\eqref{def:casimirset} can be found for example in \cite{chong1974some,chong1976doubly,ryff1965orbits}.  The \eqref{def:shnirset} has been used by Shnirelman \cite{shnirelman1993lattice}, while the characterization \eqref{def:sverakset} also appears in the lecture notes of {\v{S}}ver{\'a}k \cite{vsverak2012selected}.  In the following, we present a self contained proof of these characterizations.

\begin{proof} We divide the proof in several steps.
	
	\noindent $\diamond$ \textsc{Step 1: ($\overline{\cO_{\omega_0}}^*=$ \eqref{def:shnirset})}  This characterization is a direct consequence of Proposition \ref{prop:weak}, whose proof can be found for example in \cite{brenier2003p} or \cite[Sec 1.4]{brenier2020examples}. We review the main arguments of the proof since in the sequel we need to exploit  some technical lemma used for it.

As observed in Remark \ref{rem:poly}, we know $\mathscr{E}_\mu(M)\subset\mathscr{K}(M)$. Since $\mathscr{K}(M)$  is weak-$*$ closed, we infer $\overline{\mathscr{E}_\mu(M)}^*\subseteq\mathscr{K}(M)$. Thus, to prove \eqref{eq:weakK} we only have to show that for every $K\in \mathscr{K}$ there exists a sequence $\{\phi_n\}\subset {\mathscr{E}_{\mu}(M)}$ such that 
\begin{equation}
	\label{eq:convergence}
	\lim_{n\to +\infty}\int_Mf(x,\phi_n(x))\rmd x=\iint_{M\times M }f(x,y)K(x,y)\rmd x\rmd y, \qquad \text{ for all } f\in C(M\times M).
\end{equation} 
Indeed,	choosing $f(x,y)=g(x)\omega_0(y)$, we see that any element in $\overline{\mathcal{O}_{\omega_0}}^*$ is of the form $K\omega_0$, meaning that the characterization \eqref{def:shnirset} is proved. The proof of \eqref{eq:convergence} relies on the following key lemma, which we prove below.
\begin{lemma}
	\label{lemma:perm}
	Let $Q_1, Q_2\subset M$ be two squares with centers $x_1,x_2$ respectively and $|Q_1|=|Q_2|$. Let $p:M\to M$ be a permutation of these two squares, namely 
	\begin{equation}
		p(x)=\begin{cases}
			x-x_1+x_2 \qquad &\text{if }x\in Q_1,\\
			x-x_2+x_1 \qquad &\text{if }x\in Q_2,\\
			x\qquad &\text{otherwise }.
		\end{cases}
	\end{equation}
	Then, there exists $\{\varphi_n\}\in\mathscr{E}_\mu(M)$ such that $\varphi_n\to p$ in $L^2(M)$.
\end{lemma}
In particular, permutations of squares are in $\overline{\cO_{{\omega_0}}}^{L_2}$.  The main idea is to discretize the problem \eqref{eq:convergence} and use permutations of squares as building blocks to construct the approximating sequence $\phi_n$. This is analogous to the decomposition of a doubly stochastic matrix in terms of permutation matrices, which is the classical {Birkhoff's theorem}. More precisely, given $m$ sufficiently large, we can cover the interior of $M$ with $N_m<+\infty$ squares $\{Q_i^m\}_{i=1}^{N_m}$ of area $4^{-m}$ up to an error $O(2^{-m})$. Then, approximate the measure $$\mu_K(x,y)=K(x,y)\rmd x\rmd y$$ by 
\begin{equation}
	\gamma_m=\sum_{i,j} \mu_K(Q_{i}^m\times Q_{j}^m)\delta_{(x_{i}^m,x_j^m)},
\end{equation}
where $x_i^m$ is the center of the cube $Q_i^m$. The measure $\gamma_m$ is discrete and can be identified with a matrix $A=(a_{ij})$ where $a_{ij}=4^m\mu_K(Q_{i}^m\times Q_{j}^m)$. Since $K$ is bistochastic, the matrix $A$ is also bistochastic, i.e. $\sum_i a_{ij}=\sum_ja_{ij}=1$. We can therefore apply the Birkhoff's theorem to rewrite the matrix as a convex combination of permutation matrices, namely
\begin{equation}
	a_{ij}=\sum_{k=1}^{K}\theta_k\delta_{\sigma_k(i),j}, \qquad \sum_{k=1}^K\theta_k=1,
\end{equation}
where $K\leq N_m^2$ and $\sigma$ is a permutation of $\{1,\dots N_m\}$. A permutation of squares can be approximated with a permutation matrix. Indeed, if $p_{\sigma}$ is the permutation of the squares $Q_i^m,Q_{\sigma(i)}^m$, then 
\begin{equation} 
	\sum_{i}\int_{Q_{i}^m}f(x,p_\sigma(x))\rmd x=4^{-m} \sum_{i}f(x_i^m,x_{\sigma(i)}^m)+C\eta(2^{-m}),
\end{equation}
where $\eta$ is the modulus of continuity of $f$. We are associating the discrete measure $4^{-m}\delta_{(x_i^m,x_{\sigma(i)}^m)}$ to $p_\sigma$ up to a small error. Therefore, the proof of \eqref{eq:convergence} is a standard approximation argument combined with the Birkhoff theorem and Lemma \ref{lemma:perm}. We refer to \cite[Sec 1.4]{brenier2020examples} for a detailed proof of the approximation argument. Instead, let us show the proof of Lemma \ref{lemma:perm}, see \cite[Lemma 1.2]{brenier2003p}, which we are going to use also in the proof of Theorem \ref{varprop}.

\begin{proof}[Proof of Lemma \ref{lemma:perm}]
	First observe that if $\varphi^1_n,\varphi^2_n\in \mathscr{E}_\mu (M)$ and $\varphi^1_n\to h_1$, $\varphi^2_n\to h_2$ in $L^2(M)$ then $\varphi^1_n\circ \varphi^2_n\to h_1\circ h_2$ in $L^2(M)$. Hence, it is enough to prove that we can exchange two adjacent squares, since any permutation of squares can be written as a combination of exchanges between adjacent squares (refining further the grid covering $M$ if necessary). To exchange adjacent squares, it is enough to approximate the central symmetry with respect to squares and rectangles\footnote{Equivalently, we could also exchange adjacent triangles. This can be useful to extend the proof to smooth compact manifolds.}. For instance, given $Q=[-a,a]^2$, we need to approximate the map $c(x)=-x$ if $x\in Q$ and $c(x)=x$ otherwise. Notice that $Q$ can be written as the union of the level sets for the function
	\begin{equation*}
		g(x)=\max\{|x_1|,|x_2|\}, \quad \text{so that } \quad Q=\{x| \ g(x)\leq a\}.
	\end{equation*}
 The idea is now to use the function $g$ to construct a velocity field which moves the particles along the streamlines, where the velocity can be tuned in order to reach the point $-x$ at time $t=1$ (a rigid rotation), see Figure \ref{cubes}.
	\begin{figure}[h!]
	\begin{tikzpicture}[scale=1]
		\foreach \x in{0,5.5,...,11}
		{
		\draw[very thick]  (\x,0)--(\x+4,0)--(\x+4,2)--(\x,2)--cycle;
		\draw[very thick] (\x+2,0)--(\x+2,2);
	}
		\foreach \x in{0,5.5}
		{
		\draw[very thick,->] (\x+4.2,1)--(\x+5.2,1);}

	   \filldraw [fill=gray,  opacity = .5] (0,0)--(2,0)--(2,2)--(0,2)--cycle;
        	
			\draw[dashed, blue] (0,2)--(4,0);
			\draw[dashed, blue] (0,0)--(4,2);

			  \path [draw, redroma, thick, postaction={on each segment={mid arrow=black}}]
			    (1,1.5)-- (1,.5)-- (3,.5)-- (3,1.5)--cycle;
			  
			        \filldraw[fill=redroma ] (.5,.25) circle (3pt);
			  \filldraw[fill=yroma ] (2.5,.25) circle (3pt);

\filldraw [fill=gray,  opacity = .5] (2+5.5,2)--(4+5.5,2)--(4+5.5,0)--(2+5.5,0)--cycle;

\draw[dashed, blue] (5.5,2)--(2+5.5,0);
\draw[dashed, blue] (5.5,0)--(2+5.5,2);
\draw[dashed, blue] (2+5.5,2)--(4+5.5,0);
\draw[dashed, blue] (2+5.5,0)--(4+5.5,2);

\path [draw, redroma,  thick, postaction={on each segment={mid arrow=black}}]
(.5+5.5,.5)-- (1.5+5.5,.5)-- (1.5+5.5,1.5)-- (.5+5.5,1.5)--cycle;
\path [draw, redroma,  thick, postaction={on each segment={mid arrow=black}}]
(.5+5.5,.5)-- (1.5+5.5,.5)-- (1.5+5.5,1.5)-- (.5+5.5,1.5)--cycle;
\path [draw, redroma,  thick, postaction={on each segment={mid arrow=black}}]
(2.5+5.5,.5)-- (3.5+5.5,.5)-- (3.5+5.5,1.5)-- (2.5+5.5,1.5)--cycle;

\filldraw[fill=redroma ] (3.5+5.5,1.75) circle (3pt);
\filldraw[fill=yroma ] (.5+5.5,1.75) circle (3pt);

\filldraw [fill=gray,  opacity = .5] (2+11,2)--(4+11,2)--(4+11,0)--(2+11,0)--cycle;

\filldraw[fill=yroma ] (.5+11,.25) circle (3pt);
\filldraw[fill=redroma ] (2.5+11,.25) circle (3pt);

		\draw  (2,2) node[anchor=south] {$t=0$};
\draw  (2+5.5,2) node[anchor=south] {$t=1/2$};
\draw  (2+11,2) node[anchor=south] {$t=1$};
	\end{tikzpicture}
		
		\caption{First we act with the central symmetry for the rectangle. Then we use the central symmetry in each squares.}
		\label{cubes}
	\end{figure}
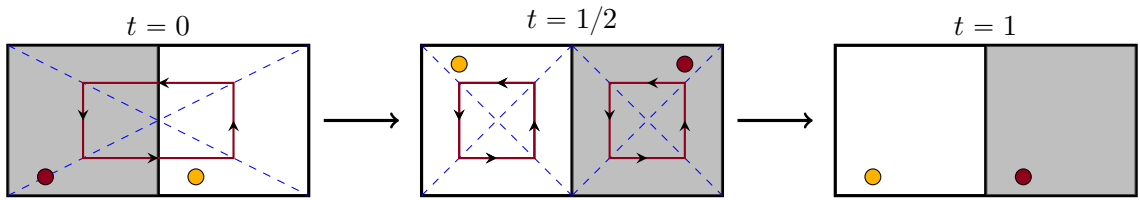  

 In this case, since $g$ is not differentiable everywhere, we cannot directly use $\nabla^\perp g$. However, it is enough to approximate $g$ on a smaller domain. In polar coordinates one has 
	\begin{equation}
		g(r,\theta)=r^2\max\{\cos^2(\theta),\sin^2(\theta)\}=r^2(1+|\cos(2\theta)|):=r^2f(\theta),
	\end{equation}
	so that a possible approximation of $g$ is given by 
	\begin{equation}
		r^2f_{\eps}(\theta):=r^2(1+\sqrt{\eps^2+\cos^2(2\theta)}).
	\end{equation}
	Also at the origin we may have problems, but since we are looking for an approximation up to zero Lebesgue measure sets, it is enough to prove that we approximate the central symmetry on the set $Q_\eps=\{\eps<r^2f_\eps(\theta)\leq 2-\eps\}$.  This can be proved by defining the streamfunction 
	\begin{equation}
		\psi_\eps(r,\theta)=\frac12\lambda_\eps r^2f_\eps(\theta), \qquad \lambda_\eps =\int_0^\pi \frac{\dd s}{f(s)}>0,
	\end{equation}
	with associated velocity field $\bv_\eps=\nabla^\perp \psi_\eps$, for which it is not difficult to show that $\bv$ moves a particle $x$ to $-x$ in time $t=1$, see \cite{brenier2003p}. The flow generated by $\de_t\phi_\eps=\bv_\eps(\phi_\eps)$ is such that $\phi_\eps(1,x)=-x$ on $Q_\eps$. Once this is is done, we can  choose $\eps=2^{-n}$ and define $c_n= {\rm id}$ on $M\setminus Q$, $c_n(x)=\phi_{\eps}(1,x)$ on $Q_\eps$ and any smooth approximation between ${\rm id}$ and $\phi_\eps$ on $Q\setminus Q_\eps$. Then, $c$ and $c_n$ are equal up to a set of measure $O(2^{-n})$ and thus, being clearly uniformly bounded, $c_n \to c$ in $L^2(M)$. 
	
	For the central symmetry with respect to a rectangle $R=[-a,a]\times[-b,b]$, just notice that $R=\{x| \max\{|x_1|/a,|x_2|/b\}\leq 1\}$, so we can repeat the construction above modifying the function $g$.
\end{proof}

$\diamond$ \textsc{Step 2: (\eqref{def:shnirset} = \eqref{def:sverakset})} 
Since $K$ is bistochastic and $(\cdot)_+$ is convex, by Jensen's inequality (see \eqref{eq:bdJ}) it follows that  
\begin{align}
\mathscr{K}_{\omega_0}&:=\{\omega \in X: \ \omega=K\omega_0\ \text{ for } K\in\mathscr{K}\}\\
&\subseteq\mathscr{S}_{\omega_0}:= \left\{\omega\in X \ : \ \int_M \omega\, \dd x =\int_M \omega_0\, \dd x, \  \ \notag \  \int_M (\omega-c)_+\, \dd x\leq \int_M (\omega_0-c)_+ \, \dd x\ \ \text{ for all } \ \ c\in \mathbb{R}  \right\}.
\end{align}
 It thus remain to prove that given an element $\omega\in \mathscr{S}_{\omega_0}$, there exists $K\in \mathscr{K}_{\omega_0}$ such that $\omega=K\omega_0$. This is indeed a classical result in \textit{rearrangement inequalities} \cite{chong1976doubly} which we prove below.
 
 For any set $A\in \RR^2$ define $A^\#$ as the ball centered at the origin such that $|A|=|A^\#|$. Given a function $f$, its distribution function is given by 
 \begin{equation}
 \label{def:df}
 d_f(t)=|\{x\in M: \ f(x)> t\}| \qquad \text{for any } t\in \RR.
 \end{equation}
 The \textit{Hardy-Littlewood-Polya decreasing rearrangement} \cite{hardy1929some} is defined as
 \begin{equation}
 \label{def:HLPrear}
 f^*(s)=\sup\{\tau \in \RR: \ d_f(\tau)>s\} \qquad \text{for } s\in [0,|M|),
 \end{equation} 
 and the \textit{Schwarz spherical decreasing rearrangement} is given by
 \begin{equation}
f^\#(x)=f^*(\pi |x|^2), \qquad \text{for } x\in B_R(0), \quad R=\sqrt{|M|/\pi}.
\end{equation}
 The function $f^\#$ is obtained by rearranging the level sets of $f$ in a symmetric and radially decreasing way. The functions $f,f^*$ are equimeasurable, and hence also $f^\#$. This implies that 
\begin{equation}
\label{eq:hash}
\{f>t\}^\#=\{f^\#>t\},
\end{equation}
since both sets are balls centered at the origin with the same volume. We are also going to use the \textit{Hardy-Littlewood-Polya inequality} which reads as 
\begin{equation}
\label{bd:HLP}
\int_M fg \, \dd x\leq \int_{B_R} f^\#g^\# \,\dd x=\int_{[0,R]} f^* g^* \, \dd s.
\end{equation}
This can be easily proved through the layer cake decomposition.
To prove $\mathscr{S}_{\omega_0}=\mathscr{K}_{\omega_0}$, we first observe that the following conditions are equivalent:
\begin{itemize}
\item[(i)] $f \in \mathscr{S}_{g}$
\item[(ii)] $\int_0^s f^*(\tau)\dd \tau \leq \int_0^s g^*(\tau)\dd \tau$ for any $s\in [0,|M|]$ and $\int_0^{|M|}f^*(\tau)\dd \tau=\int_0^{|M|} g^*(\tau)\dd \tau$
\item[(iii)] $\int_{B_r} f^\#\dd x\leq \int_{B_r} g^\#\dd x$ for any $r\in [0,R]$ and $\int_{B_R}f^\#\dd x=\int_{B_R} g^\#\dd x$, where $R=\sqrt{|M|/\pi}$.
\end{itemize} 
The equivalence between (ii) and (iii) is straightforward (using the relation $s = \pi r^2$ to transition between the area of the level sets and the radius of the ball), while (i)$\iff$ (ii) is proved in \cite[Theorem 1.6]{chong1974some}.
If (ii) holds, for any $u\geq 0$ we have 
\begin{equation}
\label{bd:HLP1}
\int_M fu\, \dd x\leq \int_{B_R} f^\# u^\#\,\dd x\leq \int_{B_R}g^\#u^\#\,\dd x
\end{equation}
where the first inequality is \eqref{bd:HLP}. To prove the last inequality above, we can rewrite the integrals in terms of the 1D rearrangements over the area variable $s \in [0, |M|]$. Let $u=\sum_{i=0}^{N} a_i \chi_{A_i}$ with $a_i> a_{i+1}\geq 0$. Then $u^*=\sum_{i=0}^N a_i \chi_{[s_i,s_{i+1}]}$ with $$s_0=0, \quad s_{N+1}=|M|, \quad s_{i+1}-s_i=|A_i|.$$ 
Defining $F(s)=\int_0^s f^*(\tau)\dd \tau$ and $G(s)=\int_0^s g^*(\tau)\dd \tau$, observe that 
\begin{equation}
\int_{B_R} (f^{\#}-g^\#)u^\#\,\dd x=\int_0^{|M|} u^*(s)\frac{\dd}{\dd s}(F-G) \dd s=\sum_{i=0}^N a_i\Big((F(s_{i+1})-G(s_{i+1}))-(F(s_{i})-G(s_{i}))\Big).
\end{equation}
Then, since $F(0)=G(0)=0$ and $F(|M|)=G(|M|)$ by the conservation of the mean, applying summation by parts we deduce that
\begin{equation}
\sum_{i=0}^N a_i\Big((F(s_{i+1})-G(s_{i+1}))-(F(s_{i})-G(s_{i}))\Big)=\sum_{i=1}^{N}(a_{i-1}-a_{i})(F(s_i)-G(s_i))\leq 0
\end{equation}
and the last inequality follows since $a_{i-1}\geq a_i$ and $F(s_i)\leq G(s_i)$ on account of (ii) above. The general case is recovered by a standard approximation argument.

Finally, given $\omega\in \mathscr{S}_{\omega_0}$, assume that $\omega \notin \mathscr{K}_{\omega_0}$. We now follow the arguments in \cite{ryff1965orbits,day1973decreasing}. Since $\mathscr{K}_{\omega_0}$ is convex and weakly closed in $L^1$, if $\omega \in L^\infty(M)\setminus\mathscr{K}_{\omega_0}\subset L^{1}(M)\setminus \mathscr{K}_{\omega_0}$, by the Hahn-Banach theorem there exists $g\in L^\infty$ such that 
\begin{equation}
	\label{bd:hbcontr}
\int_M gK \omega_0\, \dd x<\int_M g \omega\, \dd x, \qquad \text{for any } K\in \mathscr{K}.
\end{equation} 
Since $\int K\omega_0=\int\omega_0=\int \omega$, we can assume that $g\geq 0$. Then, for each $f\in L^1$ there exists a measure preserving map $\sigma_f:M\to B_R$ such that $f=f^\#\circ \sigma_f$ \cite{day1973decreasing}.
 Hence, let $g=g^\#\circ \sigma_g$ and $\omega_0=\omega_0^\#\circ \sigma_{\omega_0}$. Now, if $\sigma_g$ and $\sigma_{\omega_0}$ are one-to-one it is enough to choose $K(x,y)=\delta(y-(\sigma_{\omega_0}^{-1}\circ \sigma_g)(x))$ to get 
 \begin{align}
 \int_{M} gK\omega_0  \, \dd x=\int_M (g^\#\omega_0^\#)\circ\sigma_g\, \dd x=\int_{B_R} g^\#\omega_0^\#\, \dd x\geq\int_{B_R} g^\#\omega^\#\, \dd x\geq \int_M g\omega \, \dd x
\end{align} 
where the last two bounds follows by \eqref{bd:HLP1}, but this is a contradiction and hence $\omega \in \mathscr{K}_{\omega_0}$.
When $\sigma_g$ and $\sigma_{\omega_0}$ are not one-to-one, we need to define bistochastic operators $\widetilde{K}:L^1(B_R)\to L^1(M)$ with adjoint $\widetilde{K}^*: L^\infty(M)\to L^\infty(B_R)$ where 
\begin{equation}
\int_M f\widetilde{K} g\, \dd x=\int_{B_R}g\widetilde{K}^* f\, \dd x.
\end{equation}
The operators $\widetilde{K}$ are the weak-$*$ closure of area preserving diffeomorphisms from $B_R$ to $M$. If $\widetilde{K}$ is associated to an area preserving map then $$\widetilde{K}^*\widetilde{K}={\rm id}.$$ This extension is necessary since if $\widetilde{K}(x,y)=\delta(y-\sigma(x))$ for $\sigma:B_R\to M$ area preserving map then $\widetilde{K}^*$ is not in general associated to an area preserving map \cite{ryff1965orbits}. Anyway, we know that $g=g^\#\circ \sigma_g:=\widetilde{K}_1 g^\# $ and $\omega_0=\omega_0^\#\circ \sigma_{\omega_0}:=\widetilde{K}_{2}\omega_0^\#$. Choosing  $K=\widetilde{K_1}\widetilde{K_2}^*:L^\infty (M)\to L^\infty (M)$, with $K\in \mathscr{K}$, we conclude
 \begin{align}
 \int_{M} gK\omega_0\, \dd x=\int_M (\widetilde{K_1}g^\#)\widetilde{K_1}(\widetilde{K_2}^*\widetilde{K_2}\omega_0^\#)\, \dd x=\int_{B_R} g^\#\omega_0^\#\, \dd x\geq\int_{B_R} g^\#\omega^\#\, \dd x\geq \int_M g\omega\, \dd x ,
\end{align} 
which is a contradiction with \eqref{bd:hbcontr}, meaning that we must have $\omega \in \mathscr{K}_{\omega_0}$.
\vspace{2mm}

$\diamond$ \textsc{Step 3: (\eqref{def:sverakset} = \eqref{def:casimirset})} This was proved in \cite[Theorem 2.5]{chong1974some} and also used in \cite{turkington1999statistical,boucher2000derivation}.  The inclusion $ \eqref{def:sverakset}\subseteq \eqref{def:casimirset}$ is obvious. Let us show a short proof of the remaining inclusion. We first observe that in \eqref{def:sverakset} it is enough to consider $$c\in [\min\{\omega_0\},\max\{\omega_0\}]:=I_0.$$ Indeed, if $c\geq \max\{\omega_0\}$ the inequality is trivial. Then, from the characterization \eqref{def:shnirset} we know that $\omega \in I_0$. Thus, for all $c<\min\{\omega_0\}$ we have 
\begin{equation}
\int_M (\omega-c)_+\, \dd x=\int_M(\omega-c)\, \dd x=\int_M (\omega_0-c)\, \dd x=\int_M (\omega_0-c)_+\, \dd x,
\end{equation} 
where the identity in the middle follows by the conservation of the mean. Then, to prove that $\eqref{def:casimirset}\subset \eqref{def:sverakset}$ let us first consider $f\in C^2$. Given $s\in I_0$, integrating by parts we have
\begin{align}\nonumber
\int_{I_0}(s-c)_+ f''(c)\dd c&= \int_{\min\{\omega_0\}}^s (s-c) f''(c) \dd c = (s-c)f'(c) \big|_{\min\{\omega_0\}}^s + \int_{\min\{\omega_0\}}^s f'(c)\dd c\\
&=f(s)-f(\min\{\omega_0\})-f'(\min\{\omega_0\})(s-\min\{\omega_0\}). \label{eq:deco}
\end{align}
Using conservation of the mean again, since $f''\geq 0$ by convexity, combining \eqref{eq:deco} with \eqref{def:sverakset} we get 
\begin{align}\nonumber
\int_M f(\omega)-f(\omega_0)\dd x&=\int_M \dd x\int_{I_0}((\omega-c)_+-(\omega_0-c)_+) f''(c)\dd c\\
&=\int_{I_0} f''(c)\dd c \int_{M}((\omega-c)_+-(\omega_0-c)_+)  \dd x\leq 0.
\end{align}
For any convex function $f$, the representation \eqref{eq:deco} is given by 
\begin{equation}
f(s)=\alpha_0+\alpha_1 s+\int(s-c)_+\dd \alpha(c),
\end{equation}
where $\alpha_0, \alpha_1$ are constants and $\alpha(c)$ is a positive measure. This again follows by approximation.  
\end{proof}

\section{Characterization of the minimizers}\label{seccharmin}
In this section, we aim at proving Theorem \ref{varprop} in different steps. We first prove the existence of a minimizer. Then we show \textit{(i)} and \textit{(ii)}.

\subsection{Existence} 
The existence of a minimizer is standard and follows from the  observation that a lower semicontinuous functional on a compact space attains its minimum. Indeed, let $\{\omega_n\}$ be a minimizing sequence in $\overline{\mathcal{O}_{\omega_0}}^* \cap \{ {\mathsf E}= {\mathsf E}_0 \}$. By definition, $\overline{\mathcal{O}_{\omega_0}}^*$ is weakly-* compact, so we can extract a subsequence such that $\omega_{n_j} \rightharpoonup \omega^*$ weakly in $L^2$. The corresponding stream functions $\psi_{n_j}$ then converge strongly in $L^2$ to $\psi^*$. This guarantees that $\mathsf{E}_0 = \lim_{j \to \infty} \mathsf{E}(\omega_{n_j}) =\mathsf{E}(\omega^*) $. Thus, the limit state $\omega^*$ belongs to our admissible set. Finally, the functional $\mathsf{I}_f$ is weakly lower semicontinuous because $f$ is convex, meaning that $\mathsf{I}_f(\omega^*)$ achieves the minimum as desired. 

\subsection{Minimizers are minimal in the sense of Shnirelman}
\label{subsec:minimin}
To argue that the minimizers are minimal in the sense of Shnirelman according Definition \ref{def:orderShnirelman}, we must use Lemma \ref{prop:eqstric} which we restate (due to a slight change in notation) and now prove: 

\begin{lemma}[Lemma \ref{prop:eqstric}]
	Given $\omega \in X$ and $K_1\in \mathscr{K}$, let $\omega_1=K_1\omega$.  There exists $\widetilde{K}\in \mathscr{K}$ such that $\omega=\widetilde{K}\omega_1$ if and only if 
    there exists a strictly convex function    
    $f:\mathbb{R}\to \mathbb{R}$ such that $\mathsf{I}_f(\omega_1)=\mathsf{I}_f(\omega)$. 
\end{lemma}

\begin{proof}[Proof of  Lemma \ref{prop:eqstric}]
	We first show that $\mathsf{I}_f(K\omega)\leq \mathsf{I}_f(\omega)$ for any $K\in \mathscr{K}$. Since $K$ is bistochastic, we know that $\mu_x(\dd y)=K(x,y)\dd y$ is a probability measure. By Jensen's inequality we have 
	\begin{equation}
		\label{eq:bdJ}
		\mathsf{I}_f(K\omega)=\int_M f\left(\int_M\omega(y)\mu_x(\dd y)\right)\dd x\leq \iint_{M\times M} f(\omega(y))\mu_x(\dd y)\dd x=\mathsf{I}_f(\omega)
	\end{equation}
	where the last identity follows by $\int K(x,y)\dd x=1$. 
	Therefore, we know that $\mathsf{I}_f(\omega_1)\leq \mathsf{I}_f(\omega)$ and if $\omega=\widetilde{K}\omega_1$ also $\mathsf{I}_f(\omega)\leq \mathsf{I}_f(\omega_1)$ meaning that $\mathsf{I}_f(\omega)=\mathsf{I}_f(\omega_1)$.
	
	It thus remains to prove that if $\mathsf{I}_f(\omega_1)=\mathsf{I}_f(\omega)$ for a given strictly convex function $f$, then there exists $\widetilde{K}\in \mathscr{K}$ such that $\omega=\widetilde{K}\omega_1$. Since the bound \eqref{eq:bdJ} is obtained pointwise for the integrand, when equality holds we have that for a.a. $x \in M$ 
	\begin{equation}
		\label{eq:bdJ1}
		f\left(\int_M\omega(y)\mu_x(\dd y)\right)= \int_{M} f(\omega(y))\mu_x(\dd y).
	\end{equation}
	Since $f$ is strictly convex, the equality case in the Jensen's inequality holds if and only if $\omega(y)=c_x$ $\mu_x$-a.e. for $c_x$ constant in $y$. Given a function $g:M\to \RR$, define the set 
	\begin{equation}
		S^{g}_{a,b}=\{y\in M: \ a< g(y)< b\}.
	\end{equation}
	Since $\omega$ is $\mu_x$-a.e. constant, observe that 
	\begin{equation}
		\mu_x(S^{\omega}_{a,b})=\begin{cases}
			1 \qquad &a< c_x< b\\
			0 \qquad &\text{otherwise}
		\end{cases}.
	\end{equation}
	In addition, since $\omega_1=K_1\omega=\int \omega(y)\mu_x(\dd y )$ and $\omega(y)$ is $\mu_x$-a.e. constant, we infer
	\begin{align}
		\notag |S^{\omega_1}_{a,b}|&=|\{x\in M: \ a< \omega_1(x) < b\}|=|\{x\in M: \ a< \int \omega(y)\mu_x(\dd y) < b\}|\\
		&=|\{x\in M: \ a< c_x < b\}|=|\{x\in M: \ \mu_x(S^{\omega}_{a,b})=1\}|.
	\end{align}
	Since $K_1$ is bistochastic we also have
	\begin{align}
		\notag |S^{\omega}_{a,b}|&=\int_{ \{y\in M: \ a< \omega(y)< b\}} \dd y=\iint_{M\times \{y\in M: \ a< \omega(y)< b\}}K_1(x,y) \dd y\dd x=\iint_{M\times \{y\in M: \ a< \omega(y)< b\}} \mu_x(\dd y)\dd x\\
		\notag &=\int_{M}\mu_x(S^{\omega}_{a,b})\dd x=|\{x\in M: \ \mu_x(S^{\omega}_{a,b})=1\}|=|S^{\omega_1}_{a,b}|, 
	\end{align}
	meaning that $\omega$ and $\omega_1$ are equimeasurable. Indeed, choosing $a=-\infty$, $b=-t$ and $a=t, b=+\infty$, we get 
	\begin{equation}
		|\{x\in M: \ |\omega_1(x)|>t\}|=|\{x\in M: \ |\omega(x)|>t\}|.
	\end{equation}
	Through the layer-cake representation, this imply  
	\begin{equation}
		\label{bd:eqL2}
		\norm{\omega_1}_{L^p}=\norm{\omega}_{L^p}, \qquad \text{for any } 1\leq p <\infty.
	\end{equation}
   We now have to ``invert" $K_1$. From  \eqref{eq:convergence}, since $\omega_1=K_1\omega$ we know that there exists a sequence of permutations $p_n$ such that $\omega\circ p_n\rightharpoonup \omega_1$ in $L^2$. Combining the weak convergence with \eqref{bd:eqL2} and the fact that $p_n$ is area preserving, notice that 
	\begin{equation}
		\norm{\omega\circ p_n-\omega_1}_{L^2}^2=\norm{\omega\circ p_n}_{L^2}^2+\norm{\omega_1}_{L^2}^2-2\int_M \omega_1(\omega\circ p_n)\, \dd x\overset{n\to \infty}{\longrightarrow} \norm{\omega}_{L^2}^2+\norm{\omega_1}_{L^2}^2-2\norm{\omega_1}_{L^2}^2=0. 
	\end{equation}
	Namely $\omega \circ p_n\to \omega_1$ in $L^2$. Since $p_n$ is area preserving, we also get $\omega_1\circ p_n^{-1}\to \omega$ in $L^2$. We can then define $\widetilde{K}\in \mathscr{K}$ as the operator obtained in the weak limit of $i_{p_n^{-1}}$ (see Step 1 in \S 4), i.e. 
	\begin{equation}
		\lim_{n_j\to \infty}\int_M g(x)(\omega_1\circ p_{n_j}^{-1})(x)\dd x=\int_{M}g(x)\int_M\omega_1(y)\widetilde{K}(x,y)\dd y\dd x.
	\end{equation}
	Since $\omega_1\circ p_{n}^{-1}\to \omega$ in $L^2$, we get $\omega=\widetilde{K}\omega_1$, whence the lemma is proved.
\end{proof}

As a consequence of Lemma \ref{prop:eqstric} we have:
\begin{corollary}
	\label{cor:minmin}
	Any minimizer  $\omega^*$ of the variation problem \eqref{vp} is a minimal flow in the sense of Shnirelman.
\end{corollary}
\begin{proof}[Proof of Corollary \ref{cor:minmin}]
Let $\omega^*$ be a minimizer of $\mathsf{I}_f$ with $f$ strictly convex, see \eqref{vp}. Consider $\omega_1 \in\overline{\mathcal{O}_{\omega_0}}^*\cap    \{ {\mathsf E}= {\mathsf E}_0\}$ and let $K$ be such that $\omega_1=K\omega^*$. As shown in the proof of Lemma \ref{prop:eqstric}, we have $\mathsf{I}_f(\omega_1)\leq \mathsf{I}_f(\omega^*)$. However, being $\omega^*$ a minimizer we must have $\mathsf{I}_f(\omega_1)= \mathsf{I}_f(\omega^*)$. Applying Lemma \ref{prop:eqstric}, we then deduce that $\omega^*$ is minimal in the sense of Definition \ref{def:orderShnirelman}.
\end{proof}

\subsection{Minimal flows are stationary}
Having at hand an $f$-minimal flow, it remains to show that is indeed a stationary solution.
\begin{lemma}[Minimal flows are Euler steady states]\label{prop:minstead}
	For any $f$-minimal flow $\omega^* \in\overline{\mathcal{O}_{\omega_0}}^*\cap    \{ {\mathsf E}= {\mathsf E}_0\}$, there exists a bounded monotone function $F:\mathbb{R}\to \mathbb{R}$ such that $\omega^*=F(\psi^*)$ where $\Delta \psi^*=\omega^*$.
\end{lemma}
The lemma above was proved by Shnirelman in \cite{shnirelman1993lattice} with his alternative definition of minimal flows. We show that this can be directly proved using Definition \ref{def:convexminimal}.

\begin{proof}[Proof of Lemma \ref{prop:minstead}]
We adapt the variational argument used by Shnirelman in \cite[Theorem 2]{shnirelman1993lattice} (and also by Segre and Kida in \cite[Sec A.2]{segre1998late}) to our situation.  Let $\phi$ be a permutation of two arbitrary squares $Q_1,Q_2$ in $M$. By Lemma \ref{lemma:perm}, we know that we can associate a bistochastic operator to $\phi$. Then, by convexity of $\mathscr{K}$, notice that the operator $K_\eps$ defined as
\begin{equation}
K_{\eps}^\phi\omega=(1-\eps)\omega+\eps (\omega\circ\phi)
\end{equation} 
is bistochastic. See Figure \ref{fig:bisto} for a visualization of such operator.
\begin{figure}[h!]
	\begin{tikzpicture}
		\draw[very thick] (0,0) rectangle (2,2); 
		\draw[very thick] (5,0) rectangle (7,2);
				\foreach \x in{0,.5,...,2}
		{
			\draw[dashed]  (\x,0)--(\x,2);
			\draw[dashed]  (0,\x)--(2,\x);
			\draw[dashed]  (\x+5,0)--(\x+5,2);
			\draw[dashed]  (5,\x)--(7,\x);
		}
		
		\filldraw[red] (1,1) rectangle (1.5,1.5);
		\filldraw[blue] (.5,.5) rectangle (1,1);
		
		\filldraw[red!65!blue] (1+5,1) rectangle (1.5+5,1.5);
		\filldraw[red!35!blue] (.5+5,.5) rectangle (1+5,1);
		\draw[thick,->] (2.5,1)--(4.5,1);
		\draw  (3.5,1) node[anchor=south] {$K_\eps^\phi$};
		\draw [white] (1.25,1.25) node {$Q_2$};
		\draw [white] (.75,.75) node {$Q_1$};
	\end{tikzpicture}
\caption{The operator $K_\eps^\phi$ is a proper mixing if we are exchanging squares where $\omega$ has different values.}
\label{fig:bisto}
\end{figure}
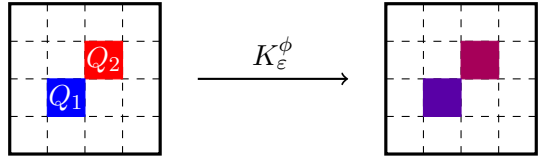

 By the choice of $\phi$, computing the first variation of the energy we have 
\begin{align}\nonumber
\frac{\dd}{\dd \eps}\mathsf{E}(K_\eps^\phi\omega^*)|_{\eps=0}&=\int_{Q_1\cup Q_2}(\omega^*(x)-\omega^*(\phi(x)))\psi^*(x)\dd x\\
&=\int_{Q_1}(\omega^*(x)-\omega^*(\phi(x)))(\psi^*(x)-\psi^*(\phi(x)))\dd x. \label{eq:first0}
\end{align}
We claim that the last integral in \eqref{eq:first0} cannot change sign. Otherwise, there are $K_\eps^{\phi_1}$ and $K_\eps^{\phi_2}$, exchanging squares $Q_1^\ell, Q_2^\ell$ with $\ell=1,2$, such that 
\begin{equation}
	\label{hyp:energy}
\mathsf{E}(K_\eps^{\phi_1}\omega^*)>\mathsf{E}(\omega^*), \qquad \mathsf{E}(K_\eps^{\phi_2}\omega^*)<\mathsf{E}(\omega^*).
\end{equation}
If the inequalities above hold, there exists $0<\lambda<1$ such that energy is preserved:
\begin{align}
&\tilde{K}_\eps\omega:=(\lambda K_\eps^{\phi_1}\omega+(1-\lambda)K_\eps^{\phi_2}\omega)=(1-\eps)\omega+\eps(\lambda \omega\circ \phi_1+(1-\lambda) \omega\circ \phi_2),\\
&\mathsf{E}(\tilde{K}_\eps\omega^*)=\mathsf{E}(\omega^*)=\mathsf{E}_0,
\end{align}
 meaning that $\tilde{K}_\eps\omega^*\in  \overline{\cO_{\omega_0}}^*\cap \{\mathsf{E}=\mathsf{E_0}\}$. Notice that the condition \eqref{hyp:energy} implies that $\omega\circ\phi_i$ is not equal to $\omega$ a.e. (namely, we are not exchanging squares where the vorticity is a constant). Then, since $f$ is strictly convex and $\omega\circ \phi_i$ is not equal to $\omega$ a.e., notice that 
\begin{align} \nonumber
	\mathsf{I}_f(\tilde{K}_\eps \omega^*)-\mathsf{I}_f(\omega^*)&=\int _M \left(f\left((1-\eps)\omega^*+\eps(\lambda \omega^*\circ \phi_1+(1-\lambda) \omega^*\circ \phi_2)\right)-f(\omega^*)\right)\, \dd x\\
	&<\eps\int_M \left(\lambda f(\omega^*\circ \phi_1)+(1-\lambda)f(\omega^*\circ \phi_2)-f(\omega^*)\right)\, \dd x=0,
\end{align}
where the last identity follows since $\phi_i$ are area preserving maps.
Therefore $\mathsf{I}_f(\tilde{K}_\eps\omega^*)<\mathsf{I}_f(\omega^*)$. Since  $\tilde{K}_\eps\omega^*\in \overline{\cO_{\omega_0}}^*\cap \{\mathsf{E}=\mathsf{E_0}\}$,  this contradicts $\omega^*$ being minimal according to Definition \ref{def:convexminimal}. Hence, this implies
\begin{equation}
	\label{eq:first00}
	\frac{\dd}{\dd \eps}\mathsf{E}(K_\eps^\phi\omega^*)|_{\eps=0}=\int_{Q_1}(\omega^*(x)-\omega^*(\phi(x)))(\psi^*(x)-\psi^*(\phi(x)))\dd x\geq 0 \ (\text{or } \leq 0).
\end{equation}
Since the choice of $Q_1$ is arbitrary, a.e. in $M$ we get
\begin{equation}
	\label{bd:monotonicity}
(\omega^*(x)-\omega^*(y))(\psi^*(x)-\psi^*(y))\geq 0,\quad \text{or} \quad (\omega^*(x)-\omega^*(y))(\psi^*(x)-\psi^*(y))\leq  0.
\end{equation}
One can now directly apply \cite[Lemma 1]{shnirelman1993lattice}, but we provide here a small generalization  with a more detailed proof.
\begin{lemma}
\label{lem:Shnirelmankey}
Let $\omega\in L^p(M)$ with $1<p\leq \infty$, and let $\psi$ be the solution to $\Delta\psi=\omega$ with $\psi=0$ on $\partial M$ (assuming $\int_M\omega=0$ if $\partial M=\emptyset$). If $(\omega(x)-\omega(y))(\psi(x)-\psi(y))\geq 0$ (or $\leq0$) for a.e. $x,y\in M$, then there exists a function $F:\mathbb{R}\to \mathbb{R}$ which is monotone nondecreasing (nonincreasing) such that $\omega=F(\psi)$ a.e. in $M$, and $F$ is bounded if $p=\infty$. Moreover, let $\omega^\star$ be the decreasing rearrangement of $\omega$  as in \eqref{def:HLPrear}. Then $F(t) = \omega^\star(d_\psi(t))$ a.e., with $d_\psi$ being the distribution function of $\psi$ in \eqref{def:df}.
\end{lemma}
Since a minimizer of $\mathsf{I}_f$ is $\omega^*\in L^\infty(M)$,  which is $f$-minimal, we know that \eqref{bd:monotonicity} holds true. We can thus apply the 
lemma above and conclude the proof $(i)$ in Thm \ref{varprop}.
\end{proof}
It remains to prove Lemma \ref{lem:Shnirelmankey}.

\begin{proof}
We consider only the case with the $\geq$ inequality, the other follows analogously.
First of all, by standard elliptic regularity theory and Sobolev embeddings, we know that $\psi \in C^{0,\beta}(\overline{M})$ for $\beta=2-2/p$ if $1<p<2$, and for any $\beta\in (0,1)$ if $p\geq2$.
Consider now the function $G=\omega+\psi$. We claim that if $G(x)=G(y)$ then $\psi(x)=\psi(y)$ and $\omega(x)=\omega(y)$. Indeed, if $\psi(x)>\psi(y)$, then the hypothesis $(\omega(x)-\omega(y))(\psi(x)-\psi(y))\geq0$ implies that $\omega(x)\geq\omega(y)$. Hence $G(x)>G(y)$, which is a contradiction. Changing the roles of $x$ and $y$ concludes the desired claim. Thus, we can write $\psi=\Psi(G)$ and $\omega=\Omega(G)$ for some functions $\Psi$ and $\Omega$ that are monotone nondecreasing. Denoting with $\Psi^{-1}$ a generalized inverse (see \cite{embrechts2013note}) of $\Psi$, chosen, for instance, to be right-continuous, we define $F=\Omega\circ \Psi^{-1}$. Because $\Psi^{-1}$ is defined everywhere, $F: [\psi_{\min},\psi_{\max}] := I \to \mathbb{R}$ is also defined everywhere and is monotone nondecreasing. 

However, the relation $\omega(x) = F(\psi(x))$ is guaranteed to hold only outside the set of values $\{\psi_j\}_{j\in J} \subset I$ where $\Psi$ is constant, i.e., where the generalized inverse is not the standard inverse. Note that because $\Psi$ is monotone, it can be constant on at most countably many disjoint intervals, so $J$ is at most countable. If $\psi(x)$ does not attain these values on sets of positive measure, $\omega = F(\psi)$ holds a.e. in $M$ and we are done. Suppose instead there exists a $j\in J$ such that the level set $M_j=\{x\in M \, :\, \psi(x) = \psi_j\}$ has positive measure, $|M_{j}|>0$. Since $M_{j}=\psi^{-1}(\psi_{j})$ and $\psi\in W^{1,1}$, by \cite[Theorem 6.19]{lieb2001analysis} we know that $\nabla \psi=0$ a.e. on $M_{j}$. We also know that $\nabla \psi\in W^{1,p}$. Hence, by \cite[Theorem 6.17]{lieb2001analysis} and the fact that $\nabla \psi=0$ a.e. on $M_{j}$, we deduce that $\Delta\psi=\omega=0$ a.e. on $M_{j}$. Therefore, in order for the relation $\omega(x) = F(\psi(x))$ to hold a.e. on $M$, we are forced to define $F(\psi_j) = 0$ for all such $j$ where $|M_j| > 0$. We must now verify that this assignment is compatible with the monotonicity of $F$. 

Let $\psi_{j_0}$ be any such value. Taking a sequence $\{\psi_{n}\}\subset I\setminus\{\psi_j\}_{j\in J}$ such that $\psi_n\downarrow \psi_{j_0}$, and taking $x_n\in \{\psi=\psi_n\}$ and $y\in M_{j_0}$, the hypothesis $(\omega(x_n)-\omega(y))(\psi_n-\psi_{j_0})\geq 0$ (with $\omega(y)=0$) implies $F(\psi_n)=\omega(x_n)\geq0$ for all $n$. Arguing analogously with a sequence $\psi_m \uparrow \psi_{j_0}$, we conclude that
$$ \lim_{\psi\to \psi_{j_0}^-}F(\psi)\leq 0 \leq \lim_{\psi\to \psi_{j_0}^+}F(\psi). $$
This guarantees that setting $F(\psi_{j_0})=0$ is compatible with the monotonicity of $F$.
Suppose now that there exists another $\psi_{j_1}$ such that $|M_{j_1}|>0$. Since $I$ is compact, assume without loss of generality that $\psi_{j_1}$ is the largest value in $I$ where this happens, and $\psi_{j_0}$ is the smallest. Repeating the arguments above, we conclude again that $F(\psi_{j_1})=0$.
By the monotonicity of $F$ and our hypothesis, we must have $F|_{[\psi_{j_0},\psi_{j_1}]}\equiv 0$. Indeed, for any $\psi_l\in(\psi_{j_0},\psi_{j_1})$ we have $F(\psi_l)\geq F(\psi_{j_0})=0$, while testing our assumption with $\psi_{j_1}$ yields $F(\psi_l)(\psi_{l}-\psi_{j_1})\geq 0$. Since $(\psi_{l}-\psi_{j_1})<0$, this forces $F(\psi_l)\leq 0$. This implies $F(\psi_l)=0$ and therefore $F|_{[\psi_{j_0},\psi_{j_1}]}\equiv 0$. 
By the choice of $\psi_{j_0}$ and $\psi_{j_1}$ as the extreme bounds, we have $F(\psi_i) \leq 0$ for any $\psi_{\min}\leq \psi_i<\psi_{j_0}$, and $F(\psi_k) \geq 0$ for any $\psi_{j_1}<\psi_k\leq \psi_{\max}$. Because $F$ is identically zero on the convex hull of all flat points with positive measure, the assignment $F(\psi_j)=0$ consistently yields a monotone nondecreasing function $F$ such that $\omega = F(\psi)$ a.e.

Finally, having established the existence and properties of $F$, we explicitly identify it almost everywhere using decreasing rearrangements. Let $d_g(t) = |\{x \in M : g(x) > t\}|$ denote the distribution function for $g \in \{\psi, \omega\}$. We claim that
$$F(t) = \omega^\star(d_\psi(t)).$$
To verify this, we distinguish two cases. First, consider a value $t$ where $F$ is strictly increasing. Because $\omega=F(\psi)$, the corresponding super-level sets coincide up to sets of measure zero, that is $\{x \in M : \omega(x) > F(t)\} = \{x \in M : \psi(x) > t\}$. Taking the Lebesgue measure of both sides yields $d_\omega(F(t)) = d_\psi(t)$, and therefore $\omega^\star(d_\omega(F(t))) = \omega^\star(d_\psi(t))$. By definition,
$$\omega^\star(d_\omega(F(t))) = \sup\{\tau \in \mathbb{R} : d_\omega(\tau) > d_\omega(F(t))\}.$$
Since $d_\omega$ is non-increasing, the inequality $d_\omega(\tau) > d_\omega(F(t))$ fails for any $\tau \geq F(t)$, meaning the supremum is at most $F(t)$. Conversely, because $F$ is strictly increasing at $t$, for any $\tau < F(t)$ there exists a $\delta > 0$ such that $F(t-\delta) \geq \tau$. Since $\psi$ is continuous, the preimage $\psi^{-1}((t-\delta, t))$ is a non-empty open set in $M$, which therefore has positive Lebesgue measure. On this set, $\omega$ takes values in $[\tau, F(t))$, guaranteeing that $d_\omega(\tau) > d_\omega(F(t))$ for all $\tau < F(t)$. Thus, the supremum is exactly $F(t)$, yielding the explicit identity $F(t) = \omega^\star(d_\psi(t))$. 

Second, suppose $t_0$ is a value where $\psi$ is constant on a set of positive measure, meaning $M_{t_0} = \{x \in M : \psi(x) = t_0\}$ has $|M_{t_0}| > 0$. This corresponds to a jump discontinuity in $d_\psi$ at $t_0$. As established earlier in the proof, $\Delta\psi = \omega = 0$ a.e. on $M_{t_0}$, which consistently forced the assignment $F(t_0) = 0$. Because $\omega$ is exactly $0$ on this set of positive measure, its decreasing rearrangement $\omega^\star(s)$ is identically $0$ over the corresponding interval of sizes $[d_\psi(t_0), d_\psi(t_0^-)]$. Consequently, evaluating $\omega^\star$ at $d_\psi(t_0)$ yields $0$, which is in perfect agreement with the assignment $F(t_0)=0$. Thus, the explicit formula $F(t) = \omega^\star(d_\psi(t))$ holds for almost every $t \in [\psi_{\min}, \psi_{\max}]$, concluding the proof.
\end{proof}
	
\subsection{Unconstrained characterization}
We exploit the abstract optimality theorem given by Rakotoson and Serre in \cite[Theorem 2]{rakotoson1993probleme}, which reads as follows. 
\begin{theorem}
	\label{th:rakserre}
	Let $\mathsf{X},\mathsf{Y}$ be two normed real vector spaces whose dual spaces are respectively $\mathsf{X}^*,\mathsf{Y}^*$. Let $\mathsf{C}\subset \mathsf{Y}$ be a convex cone\footnote{$\mathsf{C}$ is a convex cone if for each $k \in \mathsf{C}$ and $\alpha \in \mathbb{R}_+$ then $\alpha k\in \mathsf{C}$ and $\mathsf{C}+\mathsf{C}\subseteq \mathsf{C}$.} with non-empty interior. Let $g_0$ be an optimal solution of the problem 
	\begin{equation}
		J(g_0)=\inf \{J(g)  :  g\in \mathsf{X}, \quad Sg \in -\mathsf{C}\},
	\end{equation}
	where $J:\mathsf{X}\to \RR$ and $S:\mathsf{X}\to \mathsf{Y}$. Suppose that: 
	\begin{itemize}
		\item[(H1)] For all $h\in X$, the first variations of $J,S$ along $h$ are well defined at $g_0$, i.e. 
		\begin{equation}
			\lim_{\eps \to 0^+} \frac{J(g_0+\eps h)-J(g_0)}{\eps}=J'(g_0;h), \qquad 	\lim_{\eps \to 0^+} \frac{S(g_0+\eps h)-S(g_0)}{\eps}=S'(g_0;h).
		\end{equation} 
		
		\item[(H2)] The map $h\mapsto J'(g_0;h)\in \RR$ is convex. The map $h\mapsto S'(g_0;h)\in \mathsf{Y}$ is convex in the following sense: for all $\lambda\in[0,1]$ and $h_1,h_2\in \mathsf{X}$ one has 
		\begin{equation}
			\label{eq:convexity}
			S'(g_0;\lambda h_1+(1-\lambda)h_2)-\lambda S'(g_0;h)-(1-\lambda)S'(g_0;h)\in -\mathsf{C}.
		\end{equation}
	\end{itemize}
		Then, there exists $c_0\geq 0$ and $\lambda^* \in \mathsf{C}^*=\{L\in Y^* : \text{ for all } f\in \mathsf{C}, \ \langle L,f \rangle \geq0  \},$ such that the following holds true: for all $h\in \mathsf{X}$ 
\begin{align}
	\label{bd:genKTT}c_0 J'(g_0;h)+\langle \lambda^*,S'(g_0;h)\rangle &\geq0,\\
	\label{eq:orthgenKKT}\langle \lambda^*,Sg_0\rangle &= 0.
\end{align}
with $(c_0,\lambda^*)\neq(0,0)$.
\end{theorem}
\begin{remark} Theorem \ref{th:rakserre} is a natural generalization of the Karush-Kuhn-Tucker theory \cite{zeidler2013nonlinear} to the case with an infinite number of inequality constraints, see also Appendix \ref{finiteconst}.
\end{remark}	
We aim at applying Theorem \eqref{th:rakserre} in the following setting: let $\mathsf{C}$ be the convex cone 
\begin{align}\nonumber
	\mathsf{C}&=\{f\in L^{\infty}(\RR) \ :\ f(x)\geq 0, \ \text{ for a.e. } x\in \RR\}\times [0,\infty)\times \{0\}\\
	&=:\mathsf{C}_1\times [0,\infty)\times \{0\}. 	\label{def:K}
\end{align}
Observe that 
\begin{equation}
	\mathsf{C}^*=\mathsf{C}_1^*\times[0,+\infty)\times \mathbb{R},
\end{equation}
where $\mathsf{C}_1^*$ consists of non-negative, bounded and finitely additive measures that are absolutely continuous with respect to the Lebesgue measure. 
For any $\omega\in X$, with $X$ given in \eqref{def:X}, we define the functional $S:X\to L^{\infty}(\RR)\times \RR^2$ as 
\begin{align}
	\label{def:S}
	&S\omega:=(S_1\omega,S_2\omega,S_3\omega),\\
	&(S_1\omega)(c)=\int_M((\omega-c)_+-(\omega_0-c)_+)\, \dd x, \qquad \text{ for } c\in \RR, \\
	&S_2\omega=\int_M(\omega_0-\omega) \, \dd x,\\
	&S_3\omega=\mathsf{E}(\omega_0)-\mathsf{E}(\omega).
\end{align}
In account of the characterization  \eqref{def:sverakset}, imposing $S\omega\in -\mathsf{C}$ is equivalent to ask that $\omega \in\overline{\mathcal{O}_{\omega_0}}^*\cap    \{ {\mathsf E}= {\mathsf E}_0\}$. In fact, we just need to check that the mean is conserved. Take $c^*=\min\{\min\{\omega_0-1\}, \min\{\omega-1\}\}$. Then 
\begin{align}
	0\geq S_1(\omega)(c^*)= \int_M(\omega-c^*-(\omega_0-c^*))\, \dd x=\int_M\omega \, \dd x-\int_M \omega_0 \, \dd x.
\end{align}
Combining the inequality above with $S_2\omega\leq 0$ we recover the conservation of the mean. The first variation of $S_1$ is
\begin{equation}
	\label{eq:varS}
	\lim_{\eps\to 0}\frac{1}{\eps}\left(\int_M((\omega^*+\eps h-c)_+-(\omega^*-c)_+)\right)=\int_M (\chi_{\omega^*>c}h+\chi_{\omega^*=c}h_+),
\end{equation}
so we get that 
\begin{align}
	\label{hp:S}
	S'(\omega,h)(c)=(S_1'(\omega,h)(c),S_2'h,S_3'(\omega,h))=\left(\int_M(\chi_{\omega>c}h+\chi_{\omega=c}h_+), \ -\int_M h, \ \int_M \psi h\right),
\end{align}where $\Delta \psi=\omega$.
From the identity above, we deduce that \eqref{eq:convexity} holds true. Notice that the linearity with respect to $h$ of $S_3'$ is crucial.

Thanks to this construction, we can rewrite the variational problem \eqref{vp} as
\begin{equation}
	\label{def:eqpb}
	\min_{\omega \in \overline{\mathcal{O}_{\omega_0}}^*\cap    \{ {\mathsf E}= {\mathsf E}_0\}}\mathsf{I}_f(\omega)=\min\{\mathsf{I}_f(\omega): \omega \in X, \quad S\omega \in-\mathsf{C} \},
\end{equation}
where $S,\mathsf{C}$ are respectively defined in \eqref{def:S}, \eqref{def:K} and $X$ in \eqref{def:X}.

Since $\mathsf{I}_f,S$ satisfy the hypotheses \textit{(H1)-(H2)} in Theorem \ref{th:rakserre}
, we obtain the following as a consequence of Theorem \ref{th:rakserre}.

\begin{proposition}
	\label{prop:charRS}
	Let $f\in C^1(\RR)$ be a convex function. Let $\omega^*$ be an optimal solution to \eqref{def:eqpb}. There exists a non-negative measure $\lambda^* \in \mathsf{C}_1^*$, $\lambda_f,\lambda_{\mathsf{E}} \in \RR$, $\lambda_f^2+\lambda_{\mathsf{E}}^2 \neq 0$, $\lambda_m\geq 0$ such that for all $h\in X$
	\begin{align}
		\label{bd:char1}	&\int_M (\lambda_ff'(\omega^*)-\lambda_{\mathsf{E}}\psi^*-\lambda_{m})h\, \dd x +\int_\RR \left(\int_M(\chi_{\{\omega^*>c\}}h+\chi_{\{\omega^*=c\}}h_+)\, \dd x\right)\dd \lambda^*(c)\geq 0\\
		\label{bd:char2}&\int_\RR \left(\int_M\big((\omega^*-c)_+-(\omega_0-c)_+\big)\, \dd x\right)\dd \lambda^*(c)=0,
	\end{align}
	Moreover, defining the \textit{Plateau} set 
	\begin{equation}
		\label{def:plateu}
		P(\omega^*)=\{c\in[\operatorname{ess \ inf}\omega^*,\operatorname{ess \ sup} \omega^*]\, :\, |\omega^*=c|>0\},
	\end{equation}
	in the sense of distribution we have 
	\begin{align}
	\label{eq:interval1}	&\lambda_ff'(\omega^*)-\lambda_{\mathsf{E}}\psi^* \in [-\Gamma_2,-\Gamma_1],\\
		\label{def:Gamma}&\Gamma_1=\int_\RR\chi_{\{\omega^*(x)>c\}}\dd \lambda^*(c)-\lambda_m, \qquad \Gamma_2=\Gamma_1+\chi_{P(\omega^*)}(x)\int_{P(\omega^*)}\dd \lambda^*(c).
	\end{align}
	This implies that there exists a convex function $\Phi$ such that an optimal solution to \eqref{def:eqpb} is a minimizer in $X$ of the unconstrained functional 
	\begin{equation}
	\label{eq:unco}	J_\Phi(\omega)=\lambda_f\mathsf{I}_f(\omega)+\mathsf{I}_\Phi(\omega)+\lambda_{\mathsf{E}}(\mathsf{E}(\omega)-\mathsf{E}_0).
	\end{equation}
\end{proposition}
\begin{remark}
	At this level of generality, we are not able to exclude the case $\lambda_f=0$. This degenerate scenario might include stationary states with constant vorticity.  See Appendix \ref{constvort} for an example of an $f$-minimal flow having this property in a region. We also stress that the conservation of the energy and the mean for  $\omega^*$ follows by the fact that $S\omega^*\in -\mathsf{C}$.
\end{remark}
\begin{remark}
	The fact that one can rewrite a constrained minimization problem as an unconstrained one as \eqref{eq:unco}, it is standard with a finite number of inequality constraints (the Lagrange multiplier rule or the more general Karush-Kuhn-Tucker theory). That the same happens also with infinite number of constraints was observed, for instance, by Rakotoson and Serre in \cite[Remark after Theorem 1]{rakotoson1993probleme}. Our Proposition \ref{prop:charRS} is different to the result obtained in \cite{rakotoson1993probleme} because we use the characterization \eqref{def:sverakset} instead of the one with the symmetric decreasing rearrangement, see Step 2 in \S \ref{sec:charOmset}.
\end{remark}
\begin{proof}
	We can apply Theorem \ref{th:rakserre} to the problem \eqref{def:eqpb} to obtain that
	\begin{equation}
		\label{bd:pfeq}
		\int_M (\lambda_ff'(\omega^*)-\lambda_{\mathsf{E}}\psi^*-\lambda_m)h\, \dd x +\int_\RR \left(\int_M(\chi_{\{\omega^*>c\}}h+\chi_{\{\omega^*=c\}}h_+)\, \dd x\right)\dd \lambda^*(c)\geq 0,
	\end{equation}
	with 
	\begin{equation}
		\label{eq:nondeg}
		(\lambda_f,\lambda_{\mathsf{E}},\lambda_m,\lambda^*)\neq (0,0,0,0).
	\end{equation}
 The equality \eqref{bd:char2} is the orthogonality condition \eqref{eq:orthgenKKT} for $S_1$.
	
	We then have to prove that $\lambda_f^2+\lambda_{\mathsf{E}}^2\neq 0$. Assume by contradiction that $\lambda_f=\lambda_{\mathsf{E}}=0$. We can assume $\lambda_m,\, \lambda^*\neq 0$, since if one of the two is zero, also the other must be zero by the arbitrariness of $h$, whence contradicting \eqref{eq:nondeg}. By a slight abuse of notation we can set $\lambda_m=1$. From \eqref{bd:pfeq}, we have 
	\begin{equation} \label{bd:absurd}
		\int_M h\, \dd x\leq	\int_\RR \left(\int_M(\chi_{\{\omega^*>c\}}h+\chi_{\{\omega^*=c\}}h_+)\, \dd x\right)\dd \lambda^*(c). 
	\end{equation}
If $|\omega^*=\operatorname{ess \ sup}{\omega^*}|>0$, take $h=\chi_{\{\omega^*=\operatorname{ess \ sup}{\omega^*}\}}$.  On the left hand side of the inequality above, we have something strictly positive. Therefore, $\lambda^*$ must be a Dirac mass at $c=\operatorname{ess \ sup}{\omega^*}$, but this contradicts $\lambda^*\in \mathsf{C}_1^*$ (the Dirac mass is not absolutely continuous with respect to the Lebesgue measure for instance).

Otherwise, recall the definition of the {Plateau} set given in \eqref{def:plateu}. Taking $h=g\chi_{M\setminus P(\omega^*)}$ or $h=-g\chi_{M\setminus P(\omega^*)}$, we see that the inequality \eqref{bd:absurd} become the identity  
	\begin{equation} \label{bd:absurd}
	\int_{\{M\setminus P(\omega^*)\}} g\, \dd x=	\int_\RR \left(\int_{\{M\setminus P(\omega^*)\}}\chi_{\{\omega^*>c\}}g\, \dd x\right)\dd \lambda^*(c). 
\end{equation}
By Fubini's theorem, we rewrite the above as 
\begin{equation}
	\label{eq:fubini}
		\int_{\{M\setminus P(\omega^*)\}} \left(1-\int_\RR \chi_{\{\omega^*>c\}}\dd \lambda^*(c)\right)g\, \dd x=0.
\end{equation}
Taking $g$ to be  the term inside brackets, we get 
\begin{equation}
	\label{eq:esssup}
	1=\int_{-\infty}^{\operatorname{ess \ sup}{\omega^*}}\chi_{\{\omega^*(x)>c\}}\dd \lambda^*(c), \qquad \text{for a.a. } x\in \{M\setminus P(\omega^*)\}.
\end{equation}
Since we are considering the case $|\omega^*=\operatorname{ess \ sup}\omega^*|=0$ and we are taking the essential supremum, we can find at least one point $\tilde{x}\in \{M\setminus P(\omega^*)\}$ such that $\omega^*(\tilde{x})=\operatorname{ess \ sup}\omega^*$ can be taken in \eqref{eq:esssup}. But \eqref{eq:esssup} would imply that $\lambda^*$ is a Dirac mass concentrated in $\operatorname{ess \ sup}\omega^*$, whence contradicting the continuity w.r.t. the Lebesgue measure of $\lambda^*$. Therefore, $\lambda_f^2+\lambda_{\mathsf{E}}^2=0$ is not possible.

To prove \eqref{eq:interval1}, considering $h\leq 0$ in \eqref{bd:char1} and using Fubini's theorem as in \eqref{eq:fubini}, by the definition of $\Gamma_1$ in \eqref{def:Gamma} we get 
\begin{equation}
	\int_M (\lambda_ff'(\omega^*)-\lambda_{\mathsf{E}}\psi^*+\Gamma_1)(-h)\, \dd x\leq 0.
\end{equation} 
This means $\lambda_ff'(\omega^*)-\lambda_{\mathsf{E}}\psi^*+\Gamma_1\leq 0$ in the sense of distribution since we are testing against the non-negative function $-h$. The lower bound with $-\Gamma_2$ follows analogously by testing against $h\geq 0$ in \eqref{bd:char1}.

Finally, to prove \eqref{eq:unco}, it is enough to observe that $\Gamma_1$ and $\Gamma_2$ are of the form $g_1\circ \omega^*$ and $g_2\circ \omega^*$ with 
\begin{equation}
	g_1(s)=\int_\RR\chi_{\{s>c\}}\dd \lambda^*(c)-\lambda_m, \qquad g_2(s)=g_1(s)+\chi_{P(s)}\int_{P(s)}\dd \lambda^*(c).
\end{equation}
The second term in $g_2$ is interpreted as $P(s)=1$ if $s=c$ for some $c\in[\operatorname{ess \ inf}\omega^*,\operatorname{ess \ sup} \omega^*]$. We can now argue as in \cite{rakotoson1993probleme}. Namely, since $g_1\leq g_2$ and are both decreasing, it means that there exists a convex function $\Phi$ whose sub-differential, denoted by $\partial \Phi$,\footnote{For instance, $(\partial |x|)(0)=[-1,1]$.} at $s$ is the interval $[g_1(s),g_2(s)]$.  Thus 
\begin{equation}
	\lambda_f f'(\omega^*)-\lambda_{\mathsf{E}}\psi^*\in \partial \Phi(\omega^*),
\end{equation}
which, by the definition of the subdifferential, is equivalent to be a minimizer of \eqref{eq:unco} \cite{zeidler2013nonlinear}. Note that since $\Phi$ is convex, it has at most finitely many discontinuities in its derivative and hence it can be chosen Lipschitz.
 \end{proof}

\section{Excluding shear flows at infinite times}
We now turn our attention to the proof of Theorem \ref{damp}. In the sequel, we denote $x=(x_1,x_2)$ with
$x_1\in \TT$ and $x_2\in [0,1]$. We exploit the periodicity in $x_1$ by taking the Fourier transform on the horizontal variable. For any $f\in L^2(M)$, let 
\begin{equation}
	\label{def:Fourier}
	f(x_1,x_2)=\sum_{k\in \ZZ} \hat{f}_k(x_2)e^{ i k x_1}, \qquad \hat{f}_k(x_2)=\frac{1}{2\pi}\int_0^{2\pi}e^{-i kx_1}f(x_1,x_2)\dd x_1. 
\end{equation}

Given a vorticity $\omega$, the associated streamfunction $\psi$ satisfy 
\begin{equation}
	\begin{cases}
		(\de_{x_2}^2-k^2) \widehat{\psi}_k=\widehat{\omega}_k,\\
		k\widehat{\psi}_k(0)=k\widehat{\psi}_k(1)=0.  
	\end{cases}
\end{equation}
When $k=0$, we set $\widehat{\psi}_0(0)=0$ and, in principle, we could choose $\widehat{\psi}_0(1)$ as we wish. We fix the value of this constant exploiting the conservation of the linear	 momentum $\mathsf{M}$. Namely, we set 
\begin{equation}
	\label{eq:consmom}
	\widehat{\psi}_0(1)=\frac{\mathsf{M}}{2\pi}.
\end{equation} 
With this choice,  the Green's function in the periodic channel with Dirichlet boundary conditions is
\begin{equation}
	\label{Green}
	G_k(x_2,z)=-\begin{cases}\displaystyle x_2(1-z)\mathbbm{1}_{k=0}+\frac{1}{k\sinh (k)}
		\sinh(kx_2)\sinh(k(1-z)), \qquad &\text{ for } 0\leq x_2\leq z \leq 1,\\
		\displaystyle (1-x_2)z\mathbbm{1}_{k=0}+\frac{1}{k\sinh (k)}\sinh(k(1-x_2))\sinh(kz), \qquad &\text{ for } 0\leq z\leq x_2 \leq 1,
	\end{cases}
\end{equation}
and $\widehat{\psi}_k$ is then given by 
\begin{equation}
	\label{eq:psiG}
	\widehat{\psi}_k(x_2)= \frac{\mathsf{M}}{2\pi} x_2 \mathbbm{1}_{k=0} + \int_0^1 G_k(x_2,z)\widehat{\omega}_k(z)\dd z.
\end{equation}

Recall that $\omega_b$ is the background vorticity with energy $\mathsf{E}_b$, momentum $\mathsf{M}_b$ and let $\delta>0$ be a given constant. Let $0<\eps<\delta $ be a small parameter. We define $\xi$ as a small spatial scales perturbation of $\omega_b$: 
\begin{align}\label{boxvortex}
	&\xi=\omega_{b}+\delta \eps^{-2}\mathbbm{1}_{\mathsf{B}_\eps}(x_1)\mathbbm{1}_{\mathsf{A}_\eps}(x_2),\qquad \mathsf{A}_\eps=[1/2-\eps,1/2+\eps], \, \mathsf{B}_\eps=[\pi-\eps,\pi+\eps].\\
    &\omega=\xi-\omega_{b}.
\end{align}
Notice that $\omega$ is an $L^\infty$ approximation of a point vortex (Dirac point mass vorticity).  The construction can easily be smoothed to make the approximations $ C^\infty$, since we measure closeness to the background only in an integral sense.
By the definition of $\omega$, one has 
\begin{equation}
	\label{bd:L1}
	\norm{\xi-\omega_b}_{L^1}=\norm{\omega}_{L^1}=\delta \eps^{-2}\int_{\mathsf{A}_\eps\times \mathsf{B_\eps}}\dd x_2\dd x_1=4\delta.
\end{equation}
Since the linear momentum  is a linear functional, we have $\mathsf{M}(\xi)=\mathsf{M}_b+\mathsf{M}(\omega)$. Thus as $y\in[0,1]$, a bound analogous to \eqref{bd:L1} readily give us that $|\mathsf{M}(\xi)-\mathsf{M}_b|\lesssim \delta$.

We now turn our attention to the energy. It is natural to expect that $\mathsf{E}(\omega)$ is of order $|\log(\eps)|$ and, for our perturbation, we can compute this explicitly. The Fourier transform of $\omega$ is
\begin{equation}
	\label{eq:omFour}\widehat{\omega}_k(x_2)=\delta\frac{\eps^{-1}}{\pi}\mathbbm{1}_{\mathsf{A}_\eps}(x_2)\mathbbm{1}_{k=0}+\delta \eps^{-2}\mathbbm{1}_{\mathsf{A}_\eps}(x_2) \frac{e^{-ik\pi}}{\pi}\frac{\sin(k\eps)}{k} 
\end{equation}
 By Plancherel's theorem, the energy is given by 
 \begin{align}\nonumber
	\mathsf{E}(\omega)&=-\frac12\int_M\omega\psi \dd x=-\delta\eps^{-1}\int_{1/2-\eps}^{1/2+\eps}\widehat{\psi}_0(x_2)\dd x_2-\sum_{k\neq 0}\delta\frac{\eps^{-2}}{k}e^{ik\pi}\sin(k\eps)\int_{1/2-\eps}^{1/2+\eps}{\widehat{\psi}}_k(x_2)\dd x_2\\
	&:=\delta\big(\mathcal{E}_0+\sum_{k\neq 0}\mathcal{E}_k\big) .
\end{align}
For the $k=0$ part, since $\de_{x_2x_2}\widehat{\psi}_0=\widehat{\omega}_0$, by Taylor's theorem we get 
\begin{equation}
\widehat{\psi}_0(x_2)=\widehat{\psi}_0(1/2)+\de_{x_2}\widehat{\psi}_0(1/2)(x_2-1/2)+\frac{\widehat{\omega}_0(\tilde{x}_2)}{2}(x_2-1/2)^2
\end{equation}
with $\tilde{x}_2$ between $x_2$ and $1/2$. Therefore,  
\begin{equation}
\mathcal{E}_0=-2\widehat{\psi}_0(1/2)-\delta \frac{\eps^{-2}}{2\pi}\int_{1/2-\eps}^{1/2+\eps}(x_2-1/2)^2\dd x_2=-2\widehat{\psi}_0(1/2)-\delta \frac{\eps}{3\pi}.
\end{equation}
Using \eqref{eq:psiG} and \eqref{eq:omFour}, by the continuity of the Green's function, we infer 
\begin{equation}
	\widehat{\psi}_0(1/2)= \frac{\mathsf{M}(\omega)}{4\pi} + \delta\frac{2}{\pi}G_0(1/2,1/2)+\delta O(\eps)= O(\delta),
\end{equation}
where we used that the momentum of the perturbation is $\mathsf{M}(\omega) = O(\delta)$ and $G_0(1/2,1/2) = -1/4$. Consequently, 
\begin{equation}
	\label{bd:E0}
	\mathcal{E}_0= O(\delta).
\end{equation}
To compute $\mathcal{E}_k$, we need to know the stream function for $x_2\in [1/2-\eps,1/2+\eps]$. By \eqref{eq:psiG}, when $x_2\in [1/2-\eps,1/2+\eps]$ and $k\neq 0$ we have 
\begin{align}
	\notag e^{ik\pi}\widehat{\psi}_k(x_2)=\,&-\delta\eps^{-2}\frac{ \sin(k\eps)}{\pi k}\frac{1}{k\sinh(k)}\bigg(\sinh(k(1-x_2))\int_{1/2-\eps}^{x_2}\sinh(kz)\dd z\\ \nonumber
	&\qquad +\sinh(kx_2)\int_{x_2}^{1/2+\eps}\sinh(k(1-z))\dd z\bigg) \\ \nonumber
	=\,&-\delta\eps^{-2}\frac{ \sin(k\eps)}{\pi k}\frac{\sinh(k(1-x_2))}{k^2\sinh(k)}(\cosh(kx_2)-\cosh(k(1/2-\eps)))\\ 
	&\qquad  -\delta\eps^{-2}\frac{ \sin(k\eps)}{\pi k}\frac{\sinh(kx_2)}{k^2\sinh(k)}(\cosh(k(1-x_2))-\cosh(k(1/2-\eps))).
\end{align}
Since $\sinh(a+b)=\sinh(a)\cosh(b)+\sinh(b)\cosh(a)$ we rewrite the identity above as 
\begin{align}
	\label{eq:psik}
 e^{ik\pi}	\widehat{\psi}_k(x_2)=\delta\eps^{-2}\frac{\sin(k\eps)}{\pi k^3}\left(\frac{\cosh(k(1/2-\eps))}{\sinh(k)}(\sinh(k(1-x_2))+\sinh(kx_2))-1\right).
\end{align}
Computing the integral and using $\cosh(a+b)-\cosh(a-b)=2\sinh(a)\sinh(b)$, we get
\begin{align}\nonumber
	e^{ik\pi}\int_{1/2-\eps}^{1/2+\eps}{\widehat{\psi}}_k(x_2)\dd x_2&=2\delta \eps^{-1}\frac{\sin(k\eps)}{\pi k^3}\left(\frac{\cosh(k/2-k\eps)}{(k\eps)\sinh(k)}(\cosh(k/2+k\eps)-\cosh(k/2-k\eps))-1\right)\\
\label{eq:intpsi}	&=2\delta \eps^{-1}\frac{\sin(k\eps)}{\pi k^3}\left(2\frac{\cosh(k/2-k\eps)}{\sinh(k)}\sinh(k/2)\frac{\sinh(k\eps)}{k\eps}-1\right).
\end{align}
From standard properties of the hyperbolic functions, notice that
\begin{equation}
\label{eq:coolcosh}	2\frac{\cosh(k(1/2-\eps))}{\sinh(k)}\sinh(k/2)=\frac{\cosh(k(1/2-\eps))}{\cosh(k/2)}=e^{-|k|\eps}+e^{-|k|(1+\eps)}\left(\frac{e^{2|k|\eps}-1}{1+e^{-|k|}}\right).
\end{equation}
If $|k\eps|<1/10$, combining \eqref{eq:intpsi} with \eqref{eq:coolcosh}, using Taylor's formula we infer 
\begin{align}
\notag \mathcal{E}_k&=	-\frac{2\delta/\eps^{3}}{\pi}\frac{(\sin(k\eps))^2}{ k^4}\left((e^{-|k|\eps}-1)+e^{-|k|(1+\eps)}\frac{\sinh(k\eps)}{k\eps}\left(\frac{e^{2|k|\eps}-1}{1+e^{-|k|}}\right)+e^{-|k|\eps}\left(\frac{\sinh(k\eps)}{k\eps}-1\right)\right)\\
\label{bd:Ekleq}&\approx \frac{2\delta}{\pi |k|}+O(\delta) e^{-|k\eps|/4}.
\end{align}
From the identity above for $\mathcal{E}_k$ we deduce the following (rough) bound at large frequencies
\begin{align}
	\label{bd:Ekgeq}
	\left|\mathcal{E}_k\right|&\lesssim 	\frac{\eps}{(\eps k)^4} \qquad \text{for} \qquad |k\eps|\geq\tfrac{1}{10}.
\end{align}
Putting together \eqref{bd:E0}, \eqref{bd:Ekleq} and \eqref{bd:Ekgeq} we obtain
\begin{align}\nonumber
	\mathsf{E}(\omega)&\approx \delta^2\bigg(1+O(\eps)+\sum_{|k|< (10\eps)^{-1}}\left(\frac{1}{|k|}+ O(1)e^{-|k\eps|/4}\right)+\sum_{|k|\geq (10 \eps)^{-1}}\frac{O(\eps)}{(\eps k)^4}\bigg)\\
	&\approx  \delta^2\big(|\log(\eps)|+1+O(\eps)\big).
\end{align}
Taking $\eps$ sufficiently small, we finally get
\begin{equation}
\mathsf{E}(\omega)\approx \delta^2|\log(\eps)|.
\end{equation}
Moreover, the main contribution to the energy of $\xi$ is given by $\omega$. Indeed, 
\begin{align}
	\mathsf{E}(\xi)=\mathsf{E}(\omega)+\mathsf{E}_b+2\int_M \psi_b\omega\dd x
\end{align}
Since $\norm{\psi_b}_{L^\infty}\lesssim \norm{\omega_b}_{L^\infty}$, we have 
\begin{equation}
	\left|\int_M \psi_b\omega\dd x\right|\lesssim \norm{\omega_b}_{L^\infty}\norm{\omega}_{L^1}\lesssim \delta \norm{\omega_b}_{L^\infty}.
\end{equation}
Taking $\eps$ sufficiently small so that $\mathsf{E}(\omega)\approx	\delta^2|\log(\eps)|\gg \mathsf{E}_b +\delta \norm{\omega_b}_{L^\infty}$
we have 
\begin{equation}
	\label{energypert}
	\mathsf{E}(\xi)\approx \delta^2|\log(\eps)| .	
\end{equation}

We are now ready to prove that $\xi$ cannot be rearranged into a shear flow in the set $\overline{\mathcal{O}_{\xi}}^*\cap \{ \mathsf{E}={\mathsf E}(\xi)\}\cap \{ \mathsf{M}={\mathsf M}(\xi)\}$. Assume by contradiction that $\widetilde{\omega}_{\mathsf{s}}\in\overline{\mathcal{O}_{\xi}}^*\cap \{ \mathsf{E}={\mathsf E}(\xi)\}\cap \{ \mathsf{M}={\mathsf M}(\xi)\}$ is a shear flow, namely $\widetilde{\omega}_{\mathsf{s}}\equiv \widetilde{\omega}_{\mathsf{s}}(x_2)$.  By the characterization given in \eqref{def:shnirset}, we know that
\begin{equation}
\norm{\tilde\omega_{\mathsf{s}}}_{L^\infty}\lesssim \eps^{-2}	
\end{equation}
Moreover, to obtain a shear flow from $\xi$ there are two possibilities: 
\begin{enumerate}
\item rearrange the value $\eps^{-2}$ in horizontal strips whose total size is $\eps^2$,
\item  do a proper mixing of $\omega$ and $\omega_b$ and rearrange everything to get a function depending only on $x_2$.
\end{enumerate}
 This last procedure creates a shear flow whose $L^\infty$ norm is smaller with respect to a rearrangement but the resulting shear flow could be big in a larger set. In particular, the worst case scenario is to create a shear flow of the form  
\begin{equation}
	\label{def:worstshear}
	 \widetilde{\omega}_{\mathsf{s}}(x_2)= \begin{cases}
	O(\mu^{-p}) & \text{ on } \widetilde{A}_{\mu^2}, \, |\widetilde{A}_{\mu^2}|\leq  \mu^2,  \\
	 O(1) & \text{ on } [0,1]\setminus A_{\mu^2}, 
	\end{cases}
\end{equation}
where $0<\mu\ll 1$ and $p>0$ are numbers that need to be controlled with the constraints imposed on $\omega_{\mathsf{s}}$ to belong to $\overline{\cO_{\xi}}^*$. For instance, having $\mu=1/|\log(\eps)|$ and $p$ too large will give rise to an energy even larger than the one of $\xi$. However,  thanks to the characterization \eqref{def:sverakset}, if $\widetilde{\omega}_{\mathsf{s}}(x_2)\in \overline{\cO_{\xi}}^*$ one has 
\begin{equation}
	\int_M|\widetilde{\omega}_{\mathsf{s}}(x_2)|\dd x\leq \int_{M}|\xi|\dd x\leq \norm{\omega_b}_{L^1}+4\delta,
\end{equation}
which implies
\begin{equation}
	\mu^{-p}\lesssim \mu^{-2}(\norm{\omega_b}_{L^1}+4\delta +1)=O(\mu^{-2})
\end{equation}
Therefore, the worst case scenario in $\overline{\cO_{\xi}}^*$ is \eqref{def:worstshear} with $p=2$. Split now the vorticity into the large and $O(1)$ part as 
\begin{equation}
	\widetilde{\omega}_\mathsf{s}(x_2)=\widetilde{\omega}_\mathsf{s}(x_2)(\mathbbm{1}_{A_{\mu^2}}+\mathbbm{1}_{[0,1]\setminus A_{\mu^2}})(x_2):=\widetilde{\omega}_{\mathsf{s}}^{L}(x_2)+\widetilde{\omega}_{\mathsf{s}}^{1}(x_2).
\end{equation}
We rewrite the energy as
\begin{align} \nonumber
	\mathsf{E}(\widetilde{\omega}_{\mathsf{s}})&=-\frac12\int_0^1\widetilde{\omega}_\mathsf{s}(x_2)\left(\int_0^1G_0(y,z)\widetilde{\omega}_\mathsf{s}(z)\dd z\right)\dd x_2\\ \nonumber
	&=-\frac12\int_0^1(\widetilde{\omega}_{\mathsf{s}}^{L}(x_2)+\widetilde{\omega}_{\mathsf{s}}^{1}(x_2))\left(\int_0^1G_0(y,z)(\widetilde{\omega}_{\mathsf{s}}^{L}(z)+\widetilde{\omega}_{\mathsf{s}}^{1}(z))\dd z\right)\dd x_2\\
	&:= I_{L,L}+I_{L,1}+I_{1,L}+I_{1,1},
\end{align}
where $I_{L,L}$ is the integral containing two large vorticities and so on. Using the boundedness of $G_0$, we control each term as follows 
\begin{align}
	|I_{L,L}|&\lesssim \mu^{-4}\int_{A_{\mu^2}\times A_{\mu^2}}\dd x_2 \dd z=O(1),\\
	|I_{L,1}|+|I_{1,L}|&\lesssim \mu^{-2}\int_{A_{\mu^2}\times ([0,1]\setminus A_{\mu^2})}\dd x_2 \dd z=O(1),\\
	|I_{1,1}|&\lesssim \int_{([0,1]\setminus A_{\mu^2})\times ([0,1]\setminus A_{\mu^2})}\dd x_2 \dd z=O(1).
\end{align}
Therefore, the energy of shear flow $\widetilde{\omega}_{\mathsf{s}}$ obtained through a rearrangement of $\xi$ would satisfy
\begin{equation}
		\mathsf{E}(\widetilde{\omega}_{\mathsf{s}})\lesssim 1.
\end{equation}
In view of \eqref{energypert}, there is a large energy gap $\mathsf{E}(\xi)\gg \mathsf{E}(\widetilde{\omega}_{\mathsf{s}})$.  Thus for any $\widetilde{\omega}_{\mathsf{s}}\in\overline{\mathcal{O}_{\xi}}^* \cap \{ \mathsf{M}={\mathsf M}(\xi)\}$, we have $\widetilde{\omega}_{\mathsf{s}}\notin\overline{\mathcal{O}_{\xi}}^*\cap \{ \mathsf{E}={\mathsf E}(\xi)\}\cap \{ \mathsf{M}={\mathsf M}(\xi)\}$. This completes the proof. \qed

\begin{remark}[Conservation of the mean]
	It is crucial to use the conservation of momentum to fix the constant \eqref{eq:consmom} and use the Green's function with Dirichlet boundary conditions. If we do not impose the constraint on the momentum, one may always shift a shear flow by an arbitrary large constant to obtain an energy of size $|\log(\eps)|$.
\end{remark}

\begin{remark}[Elliptical vortex]\label{ellipticalvortex}
In the above proof, an alternative to using the box vortex \eqref{boxvortex} would be to use an elliptical patch
\be
\omega_E:= \mathsf{m}\chi_{E}, \qquad E = \{ x^2/a^2 + y^2/b^2=1\}, \qquad \mathsf{m}\in \mathbb{R}.
\ee
Regarded as a solution on $\mathbb{R}^2$, it has renormalized energy $\mathsf{E}[\omega_E]:=-\frac{1}{2} \int \psi_E \omega_E$ where $\psi_E$ is the corresponding stream function that can be written down explicitly:
\be\label{energyellipse}
\mathsf{E}[\omega_E] = -\frac{\Gamma^2}{4\pi}\left[\log \left(\frac{a+b}{2}\right)-\frac{1}{4}\right],
\ee
where $\Gamma = \pi ab\mathsf{m}$ is the circulation of the vortex.  From this formula, the effect of the logarithmic singularity is evident.  On bounded domains, \eqref{energyellipse} is the leading order behavior of the energy for $a,b\ll 1$. 
\end{remark}

\begin{remark}[On the smallness of $\eps$]
	If we consider data of the form comprised of two patches
	\be
	\omega_0=a_1\chi_{A_1}+a_2\chi_{A_2}.
	\ee
  Prop. \ref{prop:patch} in this specific setting says 
 	\begin{align}
 		\label{eq:COpatch22}
 		\overline{\mathcal{O}_{\omega_0}}^*= \Bigg\{a_1\leq \omega\leq a_2 \ : \ \int_M \omega = a_1|A_1|+ a_2|A_2|, \quad  \int_M (\omega-a_1)_+ \dd x \leq  (a_2-a_1)|A_2| \Bigg\}.
 	\end{align}
Given this explicit characterization, one can hope to bound better how small $\varepsilon$ must be taken in the argument to rule out shear flows.  For example,
take $a_1=-1$, $a_2= \mathsf{m}$ and $|A_2|= \varepsilon^2$, $|A_1|= |\mathbb{T}\times [0,1]|- \varepsilon^2$. 
 One can compute the maximal energy of all shears in the set \eqref{eq:COpatch22}, depending on parameters $ \mathsf{m}$ and $\varepsilon$.  If $A_2$ is chosen e.g. to be an ellipse as in Remark \ref{ellipticalvortex}, then there is a formula for its energy  which is explicit up to corrections coming from the bounded domain.  These facts combined can give precise estimates as to how small $\varepsilon$ must be taken in order to exclude shears in the set $ \overline{\mathcal{O}_{\omega_0}}^*\cap \{ \mathsf{E}={\mathsf E}_0\}\cap \{ \mathsf{M}={\mathsf M}_0\}$.
\end{remark}

\appendix

\section{Maximal mixing theory prediction for vortex patches}\label{finiteconst}

Let us now give some examples of the predictions of Shnirelman's maximal mixing theory (Theorem \ref{varprop} herein). Consider the vorticity to be a finite number of vortex patches, namely $\omega_0\in X$ given by
\begin{equation}
\label{eq:discrint}
	\omega_0=\sum_{i=1}^Na_i\chi_{A_i},
\end{equation}
 with $a_i\in \RR$, $a_i\neq a_j$ and $\cup_i A_i=M$. Without loss of generality, assume that $a_1\leq a_2\leq \dots \leq a_N$. In this case, the characterization  \eqref{def:sverakset} of the weak-$*$ closure of the orbit of $\omega_0$  can be refined as follows:
 \begin{proposition} \label{prop:patch} Given any $\omega_0\in X$ of the form \eqref{eq:discrint}, we have
 	\begin{align}
 		\overline{\mathcal{O}_{\omega_0}}^*= \Bigg\{\omega\in X \ : \ \int_M \omega =\int_M \omega_0, \ \ & \ \ \  \int_M(\omega-a_i)_+\leq \int_M(\omega_0-a_i)_+ \quad \text{for all } \quad i=1,\dots,N \Bigg\}. 
 		\label{bd:finite}
 	\end{align}
 \end{proposition}
 \begin{remark}[case of two equal patches]
 	\label{rem:twopatch}
 	If $\omega_0=a_1\chi_{A_1}+a_2\chi_{A_2}$ is comprised of two patches of equal magnitude but opposite strength (e.g. $a_1=-a_2=1$) occupying equal areas ($|A_1|=|A_2|=\frac{1}{2}|M|$),  Prop. \ref{prop:patch} gives
 	\begin{align}
 		\label{eq:COpatch}
 		\overline{\mathcal{O}_{\omega_0}}^*= \Bigg\{\omega\in X \ : \ \int_M \omega =\int_M \omega_0 \Bigg\}.
 	\end{align}
 Indeed, considering $a_1<a_2$, from $\int_M(\omega-a_2)_+\leq \int_M(\omega_0-a_2)_+=0$ we readily deduce that $\omega\leq a_2$. To get the lower bound, let $\eps>0$. Then 
 \begin{equation}
 	\label{bd:twopatc}
 	\int_M(\omega-(a_1+\eps))_+\leq \int_M(\omega_0-(a_1+\eps))_+=\int_M(\omega_0-(a_1+\eps))=\int_M(\omega-(a_1+\eps)),
 \end{equation}
where in the last identity we used the conservation of the mean. Recalling the definition of the positive and negative part, i.e. $(f)_+=(f+|f|)/2$ and $(f)_-=(|f|-f)/2$, from \eqref{bd:twopatc} we get $\int_M(\omega-(a_1+\eps))_-\leq 0$. This imply $\omega \geq a_1+\eps$. Sending $\eps\to 0$, we obtain $a_1\leq\omega\leq a_2$. Since we do not have any other constraints, we deduce that any $\omega \in X$ (assuming $a_1=-a_2$) can be taken.
 \end{remark}
 The main point of \eqref{bd:finite} is that we have a \textit{finite} number of inequality constraints. 
 This is extremely useful since we can characterize the minimizer of \eqref{vp} through the Karush-Kuhn-Tucker (KKT) theory \footnote{The extension of the Lagrange multiplier rule when we have inequality constraints. } \cite{zeidler2013nonlinear}. 
 We thus obtain the following.
 \begin{proposition}
 	\label{prop:KKT}
 	Let $f\in C^1$ be a convex function. Consider $\omega_0$ as in \eqref{eq:discrint} with energy ${\mathsf E}_0$. Then, there exists $\mu_0,\mu_1,\{\lambda_i\}_{i=1}^{N}\in \RR$ such that $\omega^*\in X$ solving
 	\begin{align}
 		\label{eq:KKT}
\int_M \left(f'(\omega^*)+\mu_0\psi^*+\mu_1+\sum_{i=i}^N\lambda_i(\chi_{\omega^*>a_i}+\chi_{\omega^*=a_i}\chi_{w>0})\right)w\geq 0, \qquad \text{for all } w\in X
 	\end{align} 
 	is a minimizer of \eqref{vp}. Moreover, for $i=1,\dots, N$
    \begin{align}
    	\label{eq:KKTcond}
    	&\lambda_i\geq 0, \qquad \lambda_i \int_M \left((\omega^*-a_i)_+-(\omega_0-a_i)_+\right)=0,\\
    	&\int_M (\omega^*-a_i)_+\leq \int_M (\omega_0-a_i)_+.
    \end{align}
 \end{proposition}
 
 \begin{remark}\label{2patchrem}
In the case of equal patches with opposite strength, i.e. $\omega_0=a_1\chi_{A_1}+a_2\chi_{A_2}$  with $a_1=-a_2 =1$ and $|A_1|=|A_2|=|M|/2$, we can even obtain a stronger characterization. Indeed, from \eqref{eq:COpatch} we know that it is enough to just impose the conservation of the mean and that $\omega\in X$. In this special case, a standard trick \cite{zeidler2013nonlinear,vsverak2012selected} in variational problems is to first modify the convex function $f$ as 
\begin{equation}
	F(\omega)=\begin{cases}
		f(\omega), \qquad& \text{if } |\omega|\leq 1,\\
		+\infty, \qquad & \text{otherwise}.
	\end{cases}
\end{equation}
This modified convex function will automatically impose that the minimizer $\omega^*$ belongs to $X$. Hence, thanks to \eqref{eq:COpatch} and the standard Lagrange multiplier rule, any minimizer for the problem \eqref{vp} satisfies
\begin{equation}
	\label{Lag2patch}
	F'(\omega^*)+\mu_0\psi^*+\mu_1=0,
\end{equation}
where $\mu_0,\mu_1$ are chosen to guarantee the conservation of the energy and the mean respectively.  We note that different convex functions $f$ correspond to different minimal flows, thereby showing they need not be unique in the set $ \overline{\cO_{\omega_0}}^*\cap \{\mathsf{E}=\mathsf{E_0}\}$ for certain $\omega_0$.
 \end{remark}

\begin{remark}
	We point out that for the general case of multiple patches, the non-differentiability of the positive part function is the main impediment to obtain an identity as in \eqref{Lag2patch} rather than an inequality as in \eqref{eq:KKT}. We can however define an approximate problem that can be helpful for practical purposes. When $\int f=\int g$ we know that $\int f_+\leq \int g_+$ is equivalent to $\int |f|\leq \int |g|$, so that we can approximate the set \eqref{bd:finite} as 
	\begin{equation}\nonumber
		\cO_{\omega_0}^\delta:=\Bigg\{\omega\in X  : 1 \int_M \omega =\int_M \omega_0, \   \int_M\sqrt{(\omega-a_i)^2+\delta}\leq \int_M\sqrt{(\omega_0-a_i)^2+\delta} \quad \text{for all } \ i=1,\dots,N \Bigg\}.
	\end{equation}
In this set, the inequality \eqref{eq:KKT} is replaced by
\begin{equation}
f'(\omega^*)+\mu_0\psi^*+\mu_1+\sum_{i=i}^N\lambda_i \frac{1}{\sqrt{(\omega^*-a_i)^2+\delta}}=0,
\end{equation}
with the same conditions \eqref{eq:KKTcond}.
\end{remark}

\begin{proof}[Proof of Proposition. \ref{prop:patch}]
 	We just have to show that 
	\begin{equation}\nonumber
		 		\overline{\mathcal{O}_{\omega_0}}^*\supseteq S_{\omega_0}:=\Bigg\{\omega\in X \ : \ \int_M \omega =\int_M \omega_0, \ \  \ \ \  \int_M(\omega-a_i)_+\leq \int_M(\omega_0-a_i)_+ \quad \text{for all } \quad i=1,\dots,N \Bigg\}, 
	\end{equation}
since the reverse inclusion directly follows from \eqref{def:sverakset}. We are going to exploit the characterization \eqref{def:sverakset} of $\overline{\mathcal{O}_{\omega_0}}^*$. We first observe that for $\omega \in \overline{\cO_{\omega_0}}^*$ it is enough to consider $c\in [a_1,a_N]$ (one can argue as in Remark \ref{rem:twopatch}). Hence, it is enough to prove that for any $\omega \in S_{\omega_0}$ one has 
\begin{equation}
	\label{bd:c+}
	\int_{M}(\omega-c)_+\leq \int_{M}(\omega_0-c)_+ \qquad \text{for all } c\in[a_1,a_N].
\end{equation}
When $c=a_i$, $i=1,\dots,N$ there is nothing to prove. Assume that $a_\ell<c<a_{\ell+1}$ for some $\ell\in \{1,\dots,N\}$. Let $\lambda>0$ be such that $c=\lambda a_\ell+(1-\lambda)a_{\ell+1}$. By the convexity of the positive part function, we deduce 
\begin{equation}
	\label{bd:conv+}
	\int_M(\omega-c)_+= \int_M(\lambda(\omega -a_\ell)+(1-\lambda)(\omega-a_{\ell+1}))_+\leq \lambda\int_M(\omega_0-a_\ell)_++(1-\lambda)\int_M(\omega_0-a_{\ell+1})_+,
\end{equation}
where in the last inequality we  used $\omega \in S_{\omega_0}$.
Since $\omega_0$ is of the form \eqref{eq:discrint} and $a_\ell<c<a_{\ell+1}$, we have 
\begin{align}
	&\int_{M}(\omega_0-c)_+= \sum_{i\geq \ell+1}(a_i-c)|A_i|,\qquad
	\int_{M}(\omega_0-a_\ell)_+= \sum_{i\geq \ell+1}(a_i-a_{\ell})|A_i|,\\
	&\int_{M}(\omega_0-a_{\ell+1})_+= \sum_{i\geq \ell+1}(a_i-a_{\ell+1})|A_i|.
\end{align}
Therefore 
\begin{align}\nonumber
	\lambda\int_M(\omega_0-a_\ell)_++(1-\lambda)\int_M(\omega_0-a_{\ell+1})_+&=\sum_{i\geq \ell+1}(a_i-(\lambda a_{\ell}+(1-\lambda)a_{\ell+1}))|A_i|\\
	&= 	\sum_{i\geq \ell+1}(a_i-c)|A_i|=\int_M(\omega_0-c)_+.
\end{align}
Combining the identities above with \eqref{bd:conv+} we prove \eqref{bd:c+}, so that $S_{\omega_0}=\overline{\cO_{\omega_0}}^*$.
\end{proof}
We now turn our attention to the proof of Proposition \ref{prop:KKT}.
\begin{proof}[Proof of Proposition \ref{prop:KKT}]
	Thanks to the characterization \eqref{bd:finite}, solving \eqref{vp} corresponds to solve a minimum problem with a \textit{finite} number of inequality constraints. We can therefore construct the Lagrange function 
	\begin{align}
		\nonumber
		L(\omega,\boldsymbol{\lambda})=&\,\lambda^*I_f(\omega)+\mu_0(E(\omega)-E(\omega_0))+\mu_1\left(\int_M (\omega-\omega_0)\right)+\\ \label{lagrange}
		&\qquad +\sum_{i=1}^N\lambda_i\int_M\left((\omega-a_i)_+-(\omega_0-a_i)_+\right),
	\end{align}
with $\boldsymbol{\lambda}=(\lambda^*,\mu_0,\mu_1,\lambda_1,\dots,\lambda_N)\in \mathbb{R}^{N+3}$. Appealing to the KKT theory \cite[Theorem 47.E]{zeidler2013nonlinear}, we know that solving \eqref{vp} is equivalent to solve 
\begin{align}
	\label{vpLag}
	\min_{\omega \in X} L(\omega,\boldsymbol{\lambda}),
\end{align}
for a fixed $\boldsymbol{\lambda}$ such that for all $i=1,\dots,N$
\begin{equation}
	\label{eq:condKKT}
	\begin{cases}
	\displaystyle \lambda_i\geq 0, \qquad \lambda_i\int_M\left((\omega-a_i)_+-(\omega_0-a_i)_+\right)=0,\\
	\displaystyle  \int_M\left((\omega-a_i)_+-(\omega_0-a_i)_+\right)\leq 0.
	\end{cases}
\end{equation}
In addition, the so-called Slater condition, i.e. there exists $\widetilde{\omega}\in X$ such that $\int_M(\widetilde{\omega}-a_i)_+<\int_M(\omega_0-a_i)_+$, guarantees that $\lambda^*\neq 0$. This is clearly satisfied in our case, and therefore we consider $\lambda^*=1$. The coefficients $\mu_0,\mu_1$ are chosen to guarantee the conservation of the energy and the mean respectively. 

Then, let $\omega^*$ be a minimizer of \eqref{vp} or equivalently to \eqref{vpLag}. 
Defining $\omega_\eps=\omega^*+\eps w$, with $w\in X$, we have 
\begin{align}
	\frac{\dd}{\dd \eps}L(\omega_\eps,\boldsymbol{\lambda})|_{\eps=0}=\int_M (f'(\omega^*)+\mu_0\psi^*+\mu_1)w+\sum_{i=1}^N\lambda_i \lim_{\eps\to 0}\frac{1}{\eps}\left(\int_M((\omega^*+\eps w-a_i)_+-(\omega^*-a_i)_+)\right),
\end{align}
so that 
\begin{align}
	\label{eq:varL}
	\frac{\dd}{\dd \eps}L(\omega_\eps,\boldsymbol{\lambda})|_{\eps=0}=\int_M \left(f'(\omega^*)+\mu_0\psi^*+\mu_1+\sum_{i=i}^N\lambda_i(\chi_{\omega^*>a_i}+\chi_{\omega^*=a_i}\chi_{w>0})\right)w. 
\end{align}
If the term on the right hand side of \eqref{eq:varL} is negative, we would have found $L(\omega_\eps,\boldsymbol{\lambda})<L(\omega^*,\boldsymbol{\lambda})$, whence contradicting the optimality of $\omega^*$. Therefore,  \eqref{eq:KKT} is proved.
\end{proof}

\section{Properties of the Euler Omega limit set $\Omega_+(\omega_0)$}\label{secomegalim}

The orbit of the solution passing through $\omega_0$ well be denoted using the solution map by $\omega(t)= S_t(\omega_0)$. We recall that $\omega \in  C_w(\mathbb{R}; L^\infty(M))$ which denotes the space of function that are continuous in time with values in the weak-$*$ topology of $L^\infty(M)$ \cite{nguyen2021remarks}. 
 
We define the Omega limit set $\Omega_+(\omega_0)$ by \eqref{omlimset}.
This set is compact in $X$, which follows from the fact that $X$ is equipped with the weak-$*$ topology. Below we collect some facts about the structure of this set.

\begin{lemma}\label{longlongtime}
If $\omega_*\in \Omega_+(\omega_0)$ then $\{ S_t(\omega_*)\}_{ t\geq 0}\subset  \Omega_+(\omega_0)$.  Consequently $ \Omega_+(\omega_*)\subset \Omega_+(\omega_0)$
\end{lemma}
\begin{proof}
Since $\omega_*\in \Omega_+(\omega_0)$ there exists a sequence $t_n\to\infty$ such that $S_{t_n}(\omega_0)\wc \omega_*$. It follows from the weak-$*$ continuity of the solution map $S_t$  that $S_t(S_{t_n}(\omega_0))\wc  S_t(\omega_*)$ for any $t\in (-\infty, \infty)$.  The semigroup property shows that $S_{t + t_n}(\omega_0)\wc  S_t(\omega_*)$.  Thus $S_t(\omega_*)\in \Omega_+(\omega_0)$.  Since $\Omega_+(\omega_0)$ is compact in $X$, we have that any weak limit of $\{S_t(\omega_*)\}_{t\geq 0}$ is in  $\Omega_+(\omega_0)$.  Thus $ \Omega_+(\omega_*)\subset\Omega_+(\omega_0)$.
\end{proof}

\begin{theorem}
	For any $\omega_0\in X$,  $\Omega_+(\omega_0)$ is connected.
\end{theorem}
\begin{proof}
First, we consider the standard metric on $X$ which generates the weak-$*$ topology. More precisely, being $L^1$ separable we have a countable set of functions $f_n$ which is dense in $L^1$. For any $\omega\in X_M$ define the seminorm $p_n(\omega)=|\int_M f_n\omega |$. The metric is then defined as $d(\omega_1,\omega_2)=\sum_n 2^{-n} p_n(\omega_1-\omega_2)/(1+p_n(\omega_1-\omega_2))$.
	
	Assume now that $\Omega_+(\omega_0)$ is not connected, namely there exists two closed sets $A,B$ such that $\Omega_+(\omega_0)=A\cup B$ and $A\cap B=\emptyset$. Since $A, B\subset X$ with $X$ compact in the weak-$*$ topology, we also know that $A,B$  are compact. Therefore, we know that 
\footnote{Given any metric space $X$ with $A$ compact, $B$ closed and $A\cap B=\emptyset$ then for any $\alpha\in A$, $\beta \in B$ one has $d(\alpha,\beta)= c>0$. Indeed if this is not the case then let $\alpha_n,\beta_n$ be such that $d(\alpha_n,\beta_n)\to 0$. Being $A$ compact $\alpha_n\to \alpha\in A$ (up to subsequence). Therefore $d(\alpha,\beta_n)\to 0$ and since $B$ is closed this imply $\alpha \in B$, which is a contradiction since $A\cap B=\emptyset$.}	\begin{equation}
	\label{def:domAomB}
	d(A,B)=\inf \{d(\omega_{\mathsf{A}},\omega_B): \omega_{\mathsf{A}}\in A, \omega_B\in B\}=c>0.
	\end{equation}
	Since $A, B\subset \Omega_+(\omega_0)$, for any $\omega_{\mathsf{A}}\in A,\omega_B\in B$ there exists sequences $a_j, b_j$ such that $S_{a_j}(\omega_0)\wc \omega_{\mathsf{A}}$ and $S_{b_j}(\omega_0)\wc \omega_B$.  Since $a_j,b_j\to \infty $ as $j\to \infty$, up to a relabelling we can assume that $a_j<b_j<a_{j+1}$. In addition, taking $j>j_0$ with $j_0$ large enough, the weak-$*$ convergence imply 
	\begin{equation}
	\label{prop1}
	d(S_{a_j}(\omega_0),A)<c/2, \qquad d(S_{b_j}(\omega_0),B)<c/2, \qquad \text{for any } j>j_0.
\end{equation}	
Then, define the following sequence 
	\begin{equation}
	S_{t_j}(\omega_0)=S_{a_{2j}}(\omega_0), \qquad S_{t_{j+1}}(\omega_0)=S_{b_{2j+1}}(\omega_0).
	\end{equation}
From \eqref{prop1}, notice that $d(S_{t_j}(\omega_0),A)<c/2$ and $d(S_{t_{j+1}}(\omega_0),B)<c/2$, meaning that $S_t(\omega_0)$ with $t\in[t_j,t_{j+1}]$ is a curve joining $A$ and $B$. Therefore, by the continuity of the flow map, there exists $t_j<r_j<t_{j+1}$ such that 
\begin{equation}
d(S_{r_j}(\omega_0),A)=d(S_{r_j}(\omega_0),B)=c/2.
\end{equation}
In addition, by the compactness of $X$ we know that $S_{r_j}(\omega_0)\wc \omega_1$ (up to subsequence), and by definition of $\Omega_+(\omega_0)$ we know that $\omega_1\in \Omega_+(\omega_0)$. However, from the identity above we deduce that $\omega_1 \notin A\cup B=\Omega_+(\omega_0)$, which is a contradiction.
		\end{proof}

Most interestingly, we have

\begin{theorem}[{\v{S}}ver{\'a}k \cite{vsverak2012selected}]\label{contL2orb}
For any $\omega_0\in X$, the set $ \Omega_+(\omega_0)$ contains an $L^2$--precompact orbit $S_t(\omega_*)$.
\end{theorem}

\begin{proof}
Let $f:\mathbb{R}\to \mathbb{R}$ be any strictly convex function and consider the Casimir
\be
\mathsf{I}_f(\omega) = \int_\Omega f(\omega(x))\rmd x.
\ee
Note that the functional $\mathsf{I}_f$ is sequentially weakly$^*$ lower semi-continuous on $L^\infty(\Omega)$ and thus  weakly$^*$ lower semi-continuous on $X$. 
Moreover,  $\mathsf{I}_f$ attains its minimum on $\Omega_+(\omega_0)$.  To see this, observe that there exists a sequence $\omega^j\in\Omega_+(\omega_0)$ such that $\omega^j\wc \omega_*\in \Omega_+(\omega_0)$ (which follows from the fact that $X$ is compact with the weak-$*$ topology) and $ \mathsf{I}_f(\omega^j)\to \inf_{\omega\in \Omega_+(\omega_0)}  \mathsf{I}_f(\omega)=: m$. Since $\mathsf{I}_f(\omega_*) \leq m$ by the lower semicontinuity we deduce that $\mathsf{I}_f(\omega_*)= m$.

We now consider any weak limit $\omega_1$ of the orbit $S_t(\omega_*)$, i.e. define by a sequence of times $t_j$ such that $S_{t_j}(\omega_*)\wc \omega_1$.  By Lemma \ref{longlongtime}, 
$\omega_1\in \Omega_+(\omega_0)$. 
Therefore $\mathsf{I}_f(\omega_1)\geq m$ by the definition of $m$ as the minimum of $\mathsf{I}_f$ on $\Omega_+(\omega_0)$. On the other hand, we have   by lower semicontinuity and conservation of the Casimirs at finite times that
\be
\mathsf{I}_f(\omega_1)\leq\liminf_{j\to\infty} \mathsf{I}_f(S_{t_j}(\omega_*)),\quad \mathsf{I}_f(S_{t_j}(\omega_*))= \mathsf{I}_f(\omega_*)=m.
\ee
Thus we conclude that $\mathsf{I}_f(S_{t_j}(\omega_*))= \mathsf{I}_f(\omega_1)=m$.
Finally, we will show that $\omega_j \wc \omega$ together with $\mathsf{I}_f(\omega_j) \to \mathsf{I}_f(\omega)$ implies stong convergence in ${L^p}$ for any $p\in [1,\infty)$ when $f$ is strictly convex. Indeed, by Taylor's remainder theorem
\be
f(\omega_j) - f(\omega) = f'(\omega) (\omega_j-\omega) + \tfrac{1}{2} f''(\overline{\omega}_j) (\omega_j-\omega)^2
\ee
for some $\overline{\omega}_j\in [\omega_j, \omega]$.  Integrating the above over $\Omega$ and denoting $c_0:=\inf_{x\in \mathbb{R}} f''(x)/2>0$ we have
\be
\mathsf{I}_f(\omega_j)- \mathsf{I}_f(\omega) \geq  \int_\Omega  f'(\omega) (\omega_j-\omega)\rmd x + c_0 \|\omega_j-\omega\|_{L^2(\Omega)}^2
\ee
Taking the limit $j\to \infty$ and using the facts that $\mathsf{I}_f(\omega_j)\to \mathsf{I}_f(\omega)$ and that $\omega_j\wc \omega$ together with $f'(\omega)\in L^1(\Omega)$, we find $ \|\omega_j-\omega\|_{L^2(\Omega)}\to 0$.
 Convergence in $L^p$ follows by interpolation with $L^\infty$.
Thus we have that $S_{t_j}(\omega_*)\to \omega_1$ in any $L^p$ and thus the orbit is compact in $L^p$  for any $p\in [1,\infty)$.
\end{proof}

It would be interesting to understand whether different choices of the Casimirs $ \mathsf{I}_f$ can correspond to different compact orbits.

\section{An example of Shnirelman minimal flow with piecewise constant vorticity}\label{constvort}

Consider a shear flow $u(x_1,x_2)=(U(x_2),0)$ in a periodic channel $M=\TT\times[-1,1]$ with a convex profile $U(x_2)$. It is Shnirelman's minimal, even if the function $U(x_2)$ is not strictly convex, and has a flat piece (e.g. $U(x_2)={\rm const.}$ for $a\leq x_2 \leq b$). Consider the specific example with streamfunction $\psi$ and vorticity $\omega = - U'(x_2)$ defined by  
\begin{equation}\nonumber
		\psi(x_2)=\begin{cases} \frac12x_2+\frac18  &x_2\in[-1,-\frac12],\\
		-\frac12x_2^2  & x_2\in[-\frac12,\frac12],\\
		-\frac12x_2+\frac18  &x_2\in[\frac12,1].
	\end{cases}  \quad
U(x_2)=\begin{cases}- \frac12  &x_2\in[-1,-\frac12],\\
		x_2 & x_2\in[-\frac12,\frac12],\\
		\frac12  &x_2\in[\frac12,1].
	\end{cases}  	\quad 
	\omega(x_2)=\begin{cases}-1  &x_2\in(-\frac12,\frac12),\\
		0  & \text{otherwise}.
	\end{cases} 
\end{equation}
We now define squares where the vorticity is $0$ and $-1$. For instance
\begin{align}
	&Q_0=\begin{cases}[-\frac{\delta}{2},\frac{\delta}{2}]\times [x_{2,0},x_{2,0}+\delta]\qquad &\tfrac12<x_{2,0}<1-\delta,\\
		[-\frac{\delta}{2},\frac{\delta}{2}]\times [x_{2,0}-\delta,x_{2,0}]\qquad &-1+\delta<x_{2,0}<-\tfrac12.
	\end{cases}
	\\
	&Q_{-1}=\begin{cases}[-\frac{\delta}{2},\frac{\delta}{2}]\times [x_{2,0},x_{2,0}+\delta]\qquad &0<x_{2,0}<\tfrac12-\delta,\\
		[-\frac{\delta}{2},\frac{\delta}{2}]\times [x_{2,0}-\delta,x_{2,0}]\qquad &-\tfrac12+\delta<x_{2,0}\leq 0.
	\end{cases}  
\end{align}
Notice that $Q_0\subset \TT\times ([1/2,1]\cup [-1,-1/2])$ and $Q_{-1}\subset \TT\times [-1/2,1/2]$. Let $\Phi$ be the area preserving map that exchanges $Q_0$ with $Q_{-1}$ and define 
$
	K_\eps\omega=((1-\eps){\rm id}+\eps \Phi)(\omega).
$
Then, by definition we know that 
\begin{align*}
	&\omega|_{Q_0}=0, \qquad \omega\circ\Phi|_{Q_0}=\omega|_{Q_{-1}}=-1\\
	&\psi|_{Q_0}=\begin{cases}
		\frac12x_2+\frac18  \qquad & x_2<-\tfrac12,\\
		-\frac12x_2+\frac18 \qquad &x_2>\tfrac12,
	\end{cases} \qquad \psi\circ\Phi|_{Q_0}=\psi|_{Q_{-1}}=-\tfrac12 (x_2')^2,	
\end{align*}
where $x_2'$ denotes the vertical coordinate of the mapped point $\Phi(x) \in Q_{-1}$. Computing the first variation of the energy we have 
\begin{equation}
	\frac{\dd}{\dd \eps}E(K_\eps \omega)=\int_{Q_0}(\omega-\omega\circ \Phi)(\psi-\psi\circ \Phi).
\end{equation}
Consequently, integrating over the $x_1$-width of $\delta$, we obtain the following:
\begin{align*}
	\frac{\dd}{\dd \eps}E(K_\eps \omega)&= \delta\int_{x_{2,0}}^{x_{2,0}+\delta} (0 - (-1)) \left( (-\tfrac12 x_2+\tfrac18) - (-\tfrac12 (x_2')^2) \right) \dd x_2 < 0, \qquad \text{for} \qquad x_{2,0}>\tfrac{1}{2},\\
	\frac{\dd}{\dd \eps}E(K_\eps \omega)&= \delta\int_{x_{2,0}-\delta}^{x_{2,0}} (0 - (-1)) \left( (\tfrac12 x_2+\tfrac18) - (-\tfrac12 (x_2')^2) \right) \dd x_2 < 0, \qquad \text{for} \qquad x_{2,0}<-\tfrac{1}{2}.
\end{align*}
Note that, the $<0$ follows because the maximum of $\psi$ is $0$ at the origin, meaning $\psi|_{Q_0} \leq -1/8$ while $\psi|_{Q_{-1}} > -1/8$, meaning that $\psi - \psi\circ\Phi < 0$ everywhere in the integration domain.

Therefore, for any proper mixing we decrease the energy. Thus, $\omega$ is an ``energy excessive"  minimal flow according to the definition in \cite{shnirelman1993lattice}.
Another possible example is the circular flow inside a disk $\mathbb{D}$ given by  $u(r,\theta)=V(r)e_\theta$ where $V(r)=0$ for $0\leq r \leq a$ and convex $V(r)$  which grows $a\leq r\leq 1$.

 \subsection*{Acknowledgments} 

 We thank A. Shnirelman for numerous inspiring discussions.  We thank B. Khesin for closely reading an early version of this manuscript, as well as for many insightful comments. The authors are grateful to Luigi De Rosa for pointing out an error in the proof of Lemma \ref{prop:minstead}, and for his valuable suggestions which led to the current version of Lemma \ref{lem:Shnirelmankey}. We happily acknowledge also J. Bedrossian, T. Elgindi and F. Torres de Lizaur for many useful comments and discussions.  The research of T.D. was partially supported by NSF-DMS grant 2106233 and the Charles Simonyi Endowment at the Institute for Advanced Study. The research of M.D. was supported by the Royal Society through the University Research Fellowship  (URF$\backslash$R1$\backslash$191492) and  the GNAMPA-INdAM.

 \subsection*{Data availability Statement} The data that support (for visualization only)  the findings of this study are openly available in Github at  \url{https://github.com/navidcy/2D-Euler}.  See also  \cite{CD21,C21}.

 \bibliographystyle{siam}
\bibliography{bibentro}
\end{document}